\documentclass[11pt, a4paper]{amsart}

\usepackage{amsmath,graphicx, amsthm, amssymb, epsfig}
\usepackage{xcolor}
\usepackage[normalem]{ulem}
\numberwithin{equation}{section}

\usepackage{hyperref}
\hypersetup{
    colorlinks=true,
    linkcolor=blue,
    filecolor=magenta,      
    urlcolor=cyan,
    pdftitle={Overleaf Example},
    pdfpagemode=FullScreen,
}

\newcommand{\e}{\epsilon}
\newcommand{\ga}{\gamma}
\newcommand{\de}{\delta}
\newcommand{\br}{\mathbb{R}}
\newcommand{\N}{\mathbb{N}}
\newcommand{\ik}{\varphi}
\newcommand{\pa}{\partial}
\newcommand{\bt}{\beta}
\newcommand{\al}{\alpha}
\newcommand{\la}{\lambda}

\newcommand{\coi}{C_0^{\infty}}
\newcommand{\ioi}{\int_0^{\infty}}

\newcommand{\be}{\begin{equation}}
\newcommand{\ee}{\end{equation}}
\newcommand{\bs}{\begin{split}}

\newcommand{\fs}{\tilde\upsilon}

\newcommand{\dd}{\text{d}}
\newcommand{\op}{\text{Op}}

\newcommand{\chx}{\check x}

\newcommand{\chw}{\check w}
\newcommand{\chz}{\check z}
\newcommand{\cht}{\check t}
\newcommand{\chq}{\check q}
\newcommand{\hxi}{\hat  \xi}

\newcommand{\WF}{\text{WF}}

\newcommand{\os}{{(1)}}
\newcommand{\tp}{{(2)}}

\newcommand{\bma}{\begin{pmatrix}}
\newcommand{\ema}{\end{pmatrix}}

\newcommand{\us}{\mathcal U}
\newcommand{\vs}{\mathcal V}

\newcommand{\s}{\mathcal S}
\newcommand{\T}{\mathcal T}
\newcommand{\R}{\mathcal R}
\newcommand{\CF}{\mathcal F}
\newcommand{\G}{\mathcal G}

\newcommand{\I}{\mathcal I}
\newcommand{\chii}{\chi_{1\text{D}}}

\newcommand{\CB}{\mathcal{B}}
\newcommand{\CH}{\mathcal{H}}

\newcommand{\CE}{\mathcal E}
\newcommand{\CD}{\mathcal D}
\newcommand{\Q}{\mathcal{Q}}

\newcommand{\DTB}{\text{DTB}}
\newcommand{\tsp}{\text{supp}\,}
\newcommand{\ssp}{\text{sing\hspace{1.0pt}supp}}

\newcommand{\tsum}{\textstyle\sum}

\newtheorem{theorem}{Theorem}[section]
\newtheorem{lemma}[theorem]{Lemma}

\theoremstyle{definition}
\newtheorem{assumptions}[theorem]{Assumption}
\newtheorem{definition}[theorem]{Definition}

\begin{document}

\title[LRA of iterative reconstruction]{Local analysis of iterative reconstruction from discrete generalized Radon transform data in the plane}
\author[A Katsevich]{Alexander Katsevich$^1$}
\thanks{$^1$Department of Mathematics, University of Central Florida, Orlando, FL 32816 (Alexander.Katsevich@ucf.edu).}

\begin{abstract} 
Local reconstruction analysis (LRA) is a powerful and flexible technique to study images reconstructed from discrete generalized Radon transform (GRT) data, $g=\mathcal R f$. The main idea of LRA is to obtain a simple formula to accurately approximate an image, $f_\epsilon(x)$, reconstructed from discrete data $g(y_j)$ in an $\epsilon$-neighborhood of a point, $x_0$. The points $y_j$ lie on a grid with step size of order $\epsilon$ in each direction. In this paper we study an iterative reconstruction algorithm, which consists of minimizing a quadratic cost functional. The cost functional is the sum of a data fidelity term and a Tikhonov regularization term. The function $f$ to be reconstructed has a jump discontinuity across a smooth surface $\mathcal S$. Fix a point $x_0\in\mathcal S$ and any $A>0$. The main result of the paper is the computation of the limit $\Delta F_0(\check x;x_0):=\lim_{\epsilon\to0}(f_\epsilon(x_0+\epsilon\check x)-f_\epsilon(x_0))$, where $f_\epsilon$ is the solution to the minimization problem and $|\check x|\le A$. A numerical experiment with a circular GRT demonstrates that $\Delta F_0(\check x;x_0)$ accurately approximates the actual reconstruction obtained by the cost functional minimization.
\end{abstract}


\maketitle

\section{Introduction}\label{intro}

\subsection{Local reconstruction analysis (LRA)}
Let $\R$ be the generalized Radon transform (GRT), \textcolor{black}{which integrates a function $f$ on $\us\subset\br^n$ over a family of surfaces parameterized by a set} $\vs\subset\br^n$. 
Thus, $f$ is the function to be reconstructed. Let $\hat f(y)=(\R f)(y)$, $y\in\vs$, denote the GRT of $f$. The open sets $\us$ and $\vs$ represent the image and data domains, respectively. The values of $\hat f(y)$ are given on a rectangular grid $y_j$, $j\in\mathbb Z^n$. The step size of the grid along each coordinate is proportional to a small parameter $\e>0$. 

Consider a reconstruction formula of the form $f_0=\R^*\CB \hat f$, where $\R$ is the adjoint transform and $\CB$ is a Pseudo-Differential Operator ($\Psi$DO). Depending on the choice of $\CB$, the reconstruction can be theoretically exact ($f_0=f$), quasi-exact ($f_0-f$ is less singular than $f$), edge-enhancing ($f_0$ is more singular than $f$), or smoothing of a finite degree ($f_0$ is less singular than $f$). \textcolor{black}{The common property of each of these kinds of reconstructions is that, microlocally, $\WF(f)=\WF(f_0)$.}

Next, let $\ik$ be a compactly supported, sufficiently smooth interpolation kernel, and let $\hat f_\e(y)$, $y\in\vs$, denote the function obtained by interpolating the values $\hat f(y_j)$, see eq. \eqref{interp-data}. Denote $f_\e=\R^*\CB \hat f_\e$. Thus, $f_\e$ is reconstructed from discrete data. 

LRA is a powerful and flexible technique proposed by the author to study images $f_\e$ reconstructed from discrete GRT data. \textcolor{black}{The reconstruction can be of any kind, e.g., exact, quasi-exact, edge-enhancing, or smoothing as described above.}
The main idea of LRA is to obtain a simple formula to accurately approximate $f_\e$ in an $\e$-neighbor\-hood of a point, $x_0\in\us$. 
\textcolor{black}{We assume that} $f$ is conormal with respect to some hypersurface $\s$ \cite[Definition 18.2.6]{hor3}. Fix any $x_0\in\s$ and $A>0$. We show that, in a variety of settings, the following limit exists
\be\label{generic LRA}
\lim_{\e\to0}\e^\nu f_\e(x_0+\e\check x),\ |\chx|<A,
\ee
the limit is uniform in $\chx$, and compute the limit explicitly. Here $\nu\ge0$ is some constant, which depends on the strength of the singularity of $f_0$ at $x_0$.

\textcolor{black}{To help understand \eqref{generic LRA}, we consider an example of a GRT in the plane.} Suppose $f$ has a jump across $\s$, where $\s\subset\br^2$ is a smooth curve. Suppose also that the reconstruction is either exact or quasi-exact. In this case $\R^*\CB\R$ is a $\Psi$DO with the principal symbol 1 and $\nu=0$. \textcolor{black}{Let $\vec\Theta_0$ be a unit vector perpendicular to $\s$ at $x_0$.} 
Under some mild conditions on $\s$ and $x_0$, we show
\be\label{DTB new use}
\lim_{\e\to0}f_\e(x_0+\e\check x)=f_0(x_0^-)+\Delta f(x_0)\DTB(\vec\Theta_0\cdot\check x;x_0),\ |\chx|<A,
\ee
uniformly in $\chx$. Here 
\begin{enumerate}
\item \textcolor{black}{$f_0(x_0^\pm):=\lim_{t\to0^\pm}f_0(x_0+t\vec\Theta_0)$;}  
\item $\Delta f(x_0):=f_0(x_0^+)-f_0(x_0^-)$ is the value of the jump of $f$ across $\s$ at $x_0$; and 
\item DTB, which stands for the \textit{Discrete Transition Behavior}, is an easily computable function independent of $f$. 
\end{enumerate}
The DTB function depends only on the interpolation kernel, $\ik$.
When $\e$ is sufficiently small, the right-hand side of \eqref{DTB new use} is an accurate approximation of $f_\e$, and the DTB function accurately describes the smoothing of the jump of $f$ in $f_\e$ due to the discrete nature of the tomographic data. 

As is seen from \eqref{DTB new use}, {\it LRA provides a uniform approximation to $f_\e(x)$ in domains of size $\sim\e$, which is comparable to the data step size, i.e. at native resolution.} Given the DTB function, one can study local properties of reconstruction from discrete data, such as spatial resolution. 

In \cite{Katsevich2017a, kat19a, Katsevich2020a, Katsevich2020b, Katsevich2021a} we compute limits of the kind \eqref{generic LRA} when $f$ has jumps across {\it smooth} curves (and surfaces in higher dimensions). We consider GRT in any $\br^n$, $n\ge 2$, which integrates over submanifolds of any dimension $N$, $1\le N\le n-1$, $f$ may have a fairly general (conormal) singularity at $\s$, and the reconstruction operator is a fairly general Fourier Integral Operator (FIO). 

In many applications, discontinuities of $f$ occur across {\it rough} surfaces. Examples include soil and rock imaging, where the surfaces of cracks and pores, and the boundaries between different regions are highly irregular and frequently simulated by fractals \cite{GouyetRosso1996, Li2019, soilfractals2000, PowerTullis1991}. Another example is cancer detection in CT. Cancerous lesions have rougher boundaries than benign ones \cite{dhara2016differential,dhara2016combination}. In \cite{Katsevich2023a, Katsevich2022a, Katsevich_2025_BV}, we extend LRA to functions on the plane with jumps across {\it rough} boundaries (i.e., $\ssp f$ is no longer a smooth curve).

In \cite{Katsevich_aliasing_2023}, LRA is extended to view aliasing by obtaining a formula for a limit similar to \eqref{generic LRA} that accurately describes {\it aliasing artifacts away from $\s$}, i.e. $x_0\not\in\s$. These are rapidly oscillating artifacts (ripples) that are caused by aliasing from $\s$.

Besides the discrete nature of observed data, the second major factor affecting image quality is noise in the data. Very recently, Anuj Abhishek, James Webber and the author showed that the LRA approach allows one to obtain an accurate, complete, and simple description of the {\it reconstructed noise} in an $O(\e)$-size domain \cite{AKW2024_1, Kats_noise_2024} if $g(y_j)$ are known with random errors. 

Some of the above questions have been investigated using other techniques, most notably within the framework of sampling theory. See, for example, \cite{nat93, pal95, far04} just to name a few papers, and more recently, \cite{stef20, Monard2021}. However, these approaches rely on certain assumptions, such as $f$ being bandlimited or essentially bandlimited. Assumptions of this sort usually hold only approximately and do not apply to functions with jumps. In contrast, LRA does not require $f$ to be bandlimited in any way, either exactly or approximately.

Even though LRA applies in a wide variety of settings, one of its major limitations is that so far it has only been applied when $f_\e$ is computed by a linear reconstruction algorithm, i.e., by an application of a reconstruction formula to data: $f_\e=\R^* \CB \hat f_\e$. Very few GRTs admit an exact inversion formula. In such cases, at best, formula-based reconstruction can only guarantee an accurate recovery of the singularities of $f$ (which we called quasi-exact reconstruction above). 

\subsection{New result - LRA for iterative reconstruction}
A popular flexible approach, which can numerically exactly invert a wide range of GRTs without relying on an inversion formula, consists of iterative minimization of a cost functional, see \cite{HJL21} and references therein. This approach is called iterative reconstruction (IR). A typical cost functional is the sum of a data fidelity and  regularization terms. 

In this paper we extend LRA to the study of 2D IR when $f$ has jumps across smooth boundaries, $\s$. We consider a GRT, $\R$, in the plane. The transform $\R$ maps a function $f(x)$, $x\in\us\subset \br^2$, into its weighted integrals, $g(y)=(\R f)(y)$, along a family of curves $\s_y$, $y\in\vs\subset\br^2$. A more precise description of $\vs$ is at the beginning of Section~\ref{ssec:grt}. The discrete data, $g(y_j)$, are given at the points $y_j$, $j\in\mathbb Z^2$. The reconstructed function, still denoted $f_\e$, is computed as the solution to the following quadratic minimization problem:
\be\label{min pb 0}
f_\e=\text{argmin}_{f\in H_0^1(\us_b)}\Psi(f),\ 
\Psi(f):=\Vert \R f-\hat f_\e \Vert_{L^2(\vs)}^2+\kappa \e^{3}\Vert \pa_x f\Vert_{L^2(\us)}^2,
\ee
where $\kappa>0$ is the regularization parameter (which is predetermined and independent of $\e$), \textcolor{black}{$U_b\Subset U$ is a domain, and $H_0^1(\us_b)$ is the closure of $\coi(U_b)$ in the $H^1$ norm.} In the spirit of LRA, the factor $\e^3$ in front of the Tikhonov regularization term ensures that the resolution of the reconstruction is at the native scale $\e$.

To solve \eqref{min pb 0} numerically, one usually assumes that $f_\e(x)=\sum_l c_l\psi_l(x)$, where $\psi_l$ are some basis functions, and their number is of the same order of magnitude as the number of data points. The reconstruction consists of computing the coefficients $c_l$. For conventional analysis of reconstruction at a scale $\gg\e$, the effects of this finite-dimensional approximation of $f_\e$ are largely invisible. However, LRA is designed to study reconstruction at the native scale, $\e$. This is the scale of greatest practical importance. At this scale, the two types of effects, (a) due to the discrete nature of data and (b) due to the representation of $f_\e$ as a linear combination of basis functions, become comparable and relevant. 

It is clear that the two effects are fundamentally different. The first effect cannot be overcome, because the data are always discrete. The second effect can be largely mitigated by selecting a sufficiently fine reconstruction grid, making it negligible compared to the first one. 

An important feature of the minimization problem \eqref{min pb 0} is that it looks for a solution, $f_\e$, in an (infinite dimensional) Hilbert space. Of course, this is numerically impossible. In the end, we still have to find a solution in a finite-dimensional vector space. However, numerically, the effects of this replacement of the solution space can be made as small as one likes relative to the effects of discrete data. We consider \eqref{min pb 0} because our goal is to study the most fundamental limitation on image quality, namely the limitation due to discrete data. While the effects of finite-dimensional approximation of a solution are less fundamental (because they can be effectively mitigated), they are important from the point of view of computational complexity. Nevertheless, this line of research is beyond the scope of the paper.

The main result of the paper, stated in Theorems~\ref{thm:main res} and \ref{thm:main res mult}, consists of formulas of the type \eqref{DTB new use} for the solution $f_\e$ to the minimization problem \eqref{min pb 0}.

To the best of the author’s knowledge, this is the first paper that investigates IR at native resolution in a mathematically rigorous fashion. Usually, analysis of IR consists of showing that the algorithm converges in some global norm \cite{BKN15}. Due to its practical significance, the local analysis of resolution of various tomographic modalities has been extensively studied in the applied literature \cite{QiLeahy00, Fess03, StayFess04a, StayFess04b, ShiFess06, ShiFess09, AhnLeahy08, ChunFess12}. However, rigorous theoretical analysis has been absent.

\subsection{Organization of the paper}
The paper is organized as follows. In Section~\ref{ssec:spaces}, we introduce various function spaces used in the paper as well as remind the reader the definitions of a $\Psi$DO and FIO. These are not the most general definitions, but they will suffice for our purposes. In Section~\ref{sec:setting} we describe the setting of the problem, including the GRT $\R$, the class of functions $f$ to be reconstructed, and the minimization problem, whose solution, $f_\e$, approximates $f$. We also state our main results, Theorems~\ref{thm:main res} and \ref{thm:main res mult}. 

The rest of the paper contains the proofs of the two theorems. In Section~\ref{GRT-behavior} we obtain a formula for  $\R f$ in a neighborhood of its singular support, \eqref{g-lead-sing}. In Section~\ref{sec:reg_sol} we obtain an alternative formula for computing $f_\e$, equation \eqref{eqn checks oper 2}. The equation is of the form $\mathcal Tf_\e=F_\e$, where $\mathcal T$ is some $\Psi$DO, and $F_\e$ is computed from the data $\hat f(y_j)$. Also, we find another function, $\G_\e$, such that (a) $\G_\e$ is easier to analyze than $f_\e$ and (b) the local behavior of $f_\e$ and $\G_\e$ near $x_0$ are the same from the perspective of LRA, see Lemma~\ref{lem:femge}. In Section~\ref{sec:RHS bdd} we prove that $\Vert F_\e\Vert_{L^\infty(\us)}$ is uniformly bounded when $0<\e\ll1$. In Section~\ref{sec:fplt} we compute the limit
\be\label{del F0 lim}
\Delta F_0(\chx):=\lim_{\e\to0}\big(F_\e^{(l)}(x_0+\e\chx)-F_\e^{(l)}(x_0)\big),
\ee
where $F_\e^{(l)}(x)$ is the leading singular term of $F_\e(x)$ near $x_0$, see \eqref{last sum lim3}. 
In Section~\ref{sec:DTB fn} we compute the DTB function, $\Delta \G_0$, as the limit
\be\label{del G0 lim}
\Delta \G_0(\chx):=\lim_{\e\to0}\big(\G_\e(x_0+\e\chx)-\G_\e(x_0)\big),
\ee
see equation \eqref{rhs lim 5}. The limits in \eqref{del F0 lim} and \eqref{del G0 lim} are uniform in $\chx$ confined to any bounded set. The calculation assumes that the singularity of $f$ at $x_0$ is visible from the data only once. \textcolor{black}{This means that there is only one $y\in\vs$ such that $\s_y$ is tangent to $\s$ at $x_0$.} This completes the proof of Theorem~\ref{thm:main res}.

A more general formula for the DTB function when the data are redundant (i.e., the singularity is visible several times) is obtained in Section~\ref{sec:red dat}, see equation \eqref{upsilon simple mult}. This completes the proof of Theorem~\ref{thm:main res mult}. Results of a numerical experiment, which validate the developed theory, are presented in Section~\ref{sec:numerics}. Finally, the proofs of all the technical lemmas are included in the appendices at the end of the paper.


\section{Basic notation, function spaces, symbols, operators}\label{ssec:spaces}

Denote $\N=\{1,2,\dots\}$ and $\N_0=\{0\}\cup\N$. Let $U\subset\br^n$, $n\in\N$, be a domain, which is defined as a non-empty, connected, open set. For clarity, the zero vector in $\br^n$, $n\ge2$, is denoted $0_n$. 

\begin{definition} Let $f$ and $g$ be two functions defined on a domain $U$. We say $f(x)\asymp g(x)$ for $x\in U$ if there exist $c_{1,2}>0$ such that
$f(x)\le c_1 g(x)$ and $g(x)\le c_2 f(x)$ for any $x\in U$.
\end{definition}

\textcolor{black}{We identify tangent and cotangent spaces on Euclidean spaces with the underlying Euclidean spaces. Thus, for a function $f(x)$ defined on $\br^n$ we have
\be\begin{split}
&\pa_x f(x):=f_x^\prime(x):=(\pa_{x_1} f(x),\dots,\pa_{x_n}f(x))^T,\\ 
&\dd_x f=(\pa_{x_1} f(x),\dots,\pa_{x_n}f(x)),\\
&|\pa_x f(x)|=|\dd_x f(x)|=\big[(\pa_{x_1} f(x))^2+\dots+(\pa_{x_n}f(x))^2\big]^{1/2}.
\end{split}
\ee
We use $f_x^\prime$ or $\pa_x f$ interchangeably depending on what is more convenient from the notational perspective. 
}

We use several types of function spaces. First, $C^k(U)$, $k\in \mathbb N$, is the space of functions with bounded derivatives up to order $k$ with the norm
\be\label{ck-def}
\Vert h\Vert_{C^k(U)}:=\max_{|m|\le k}\Vert \pa_x^m h\Vert_{L^\infty(U)},\ m\in\N_0^n.
\ee
The subscript `0' in $C_0^k$ means that we consider the subspace of compactly supported functions, $C_0^k(U)\subset C^k(U)$. Further, $C^\infty(U)=\bigcap_{k\ge 1}C^k(U)$ and $C_0^\infty(U)=\bigcap_{k\ge 1}C_0^k(U)$.

Another type is the H{\"o}lder spaces $C^s(\br^n)$, $s>0$, $s\not\in\N$. If $s=k+\ga$, $k\in\N_0$, $0<\ga<1$, then $C^s(\br^n)$ is the space of $C^k(\br^n)$ functions (or, $L^\infty(\br^n)$ functions if $k=0$), which have H{\"o}lder continuous $k$-th order derivatives (see \cite[Definition 5.28]{Alazard2024}), with the norm
\be\label{holder}\bs
\Vert f\Vert_{C^s}:=&\Vert f\Vert_{C^k}+\max_{|m|=k}\sup_{x\in\br^n,|h|>0}\frac{|\pa_x^m (f(x+h)-f(x))|}{|h|^{\ga}},\ k>0,\\
\Vert f\Vert_{C^s}:=&\Vert f\Vert_{L^\infty}+\sup_{x\in\br^n,|h|>0}\frac{|f(x+h)-f(x)|}{|h|^{\ga}},\ k=0.
\end{split}
\ee
As before, $C_0^s(U)$ denotes the subspace of $C^s(\br^n)$ functions supported in $U$. 

For $s>0$, $s\not\in\N$, H{\"o}lder spaces coincide with Zygmund spaces, $C_*^s(\br^n)$. To define the latter, pick any $\mu_0\in\coi(\br^n)$ such that $\mu_0(\eta)=1$ for $|\eta|\le 1$, $\mu_0(\eta)=0$ for $|\eta|\ge 2$. Define $\mu_j(\eta):=\mu_0(2^{-j}\eta)-\mu_0(2^{-j+1}\eta)$, $j\in\mathbb N$ \cite[Definition 5.30]{Alazard2024} and \cite[Section 6.2]{Abels12}. Then 
\be\label{hz-sp}\begin{split}
&C_*^s(\br^n):=\{f\in L^\infty(\br^n):\,\Vert f\Vert_{C_*^s(\br^n)}<\infty\},\\ 
&\Vert f\Vert_{C_*^s(\br^n)}:=\sup_{j\in\mathbb N_0}2^{js}\Vert \CF^{-1}(\mu_j(\eta)\tilde f(\eta))\Vert_{L^\infty(\br^n)},
\end{split}
\ee
where $\tilde f=\CF_n f$. Here $\CF_n$ is the Fourier transform in $\br^n$:
\be
\tilde f(\xi)=\frac1{(2\pi)^n}\int_{\br^n}f(x)e^{i\xi\cdot x}\dd x,\ f\in L^1(\br^n),
\ee
which extends to tempered distributions \cite[Section 7.1]{hor}. We have $C_*^s(\br^n)=C^s(\br^n)$ and the norms $\Vert \cdot \Vert_{C_*^s(\br^n)}$, $\Vert \cdot \Vert_{C^s(\br^n)}$ are equivalent, for any $s>0$, $s\not\in\N$ \cite[Remark 5.31]{Alazard2024}.

The H{\"o}lder-Zygmund spaces are a particular case of the Besov spaces: $C^s(\br^n)=B^s_{p,q}(\br^n)$, where $p,q=\infty$ \cite[item 2 in Remark 6.4]{Abels12}.

The Sobolev space $H^s(\br^n)$, $s\in\br$, is the space of all tempered distributions $f$ for which 
\be
\Vert f\Vert_{H^s(\br^n)}^2=(2\pi)^{-n}\int_{\br^n} |\tilde f(\xi)|^2(1+|\xi|^2)^s\dd\xi<\infty.
\ee
$H_0^s(U)$ is the closure of $\coi(U)$ in $H^s(\br^n)$. $H_{loc}^s(U)$ is the space of all distributions on $U$ such that $\chi f\in H^s(\br^n)$ for any $\chi\in\coi(U)$. See \cite[Section 5, Notation and Background]{trev1}

\begin{definition}
Given a domain $U\subset\br^n$, $r\in\br$, and $N\in\N$, $S^r(U\times \br^N)$ denotes the vector-space of $C^\infty(U\times (\br^N\setminus0_N))$ functions, $\tilde B(x,\xi)$, having the following properties
\be\label{symbols def}\begin{split}
&|\pa_x^m \tilde B(x,\xi)|\le c_m|\xi|^{-n+\de}\ \forall m\in\N_0^n,x\in U,0<|\xi|\le 1;\\
&|\pa_x^{m_1} \pa_\xi^{m_2}\tilde B(x,\xi)|\le c_{m_1,m_2}|\xi|^{r-|m_2|},\ \forall m_1\in\N_0^n,m_2\in\N_0^N,x\in U,|\xi|\ge 1;
\end{split}
\ee
for some constants $\de>0$ and $c_m,c_{m_1,m_2}>0$. 
\end{definition}

The elements of $S^r$ are called symbols of order $r$. When we are talking about symbols of $\Psi$DOs (in which case $N=n$), we use the notation $S^r(U)$. We modify the conventional definition slightly to allow for symbols to be non-smooth at the origin. This makes our analysis more streamlined, but otherwise has no effects. The space of the corresponding $\Psi$DOs, given by 
\be
(\op(\tilde B(x,\xi))f)(y):=\frac1{(2\pi)^n}\int_{\br^n}\tilde B(x,\xi)\tilde f(\xi)e^{-i\xi\cdot y}\dd\xi,\ 
f\in\coi(U),
\ee
is denoted $L^r(U)$. By \cite[Theorem 2.1, Section I.2]{trev1}, any $\Psi$DO in $L^r(U)$ is $H_0^s(U_b)\to H_{loc}^{s-r}(U)$ continuous for any domain $U_b\Subset U$.
\textcolor{black}{The notation $\us_b\Subset\us$ means that the closure $\overline{\us_b}$ is compact and $\overline{\us_b}\subset \us$.}

Let $U,V\subset\br^n$ be two domains, and let 
\be
\Phi(x,y,\Theta)\in C^\infty(U\times V\times(\br^N\setminus0_N)),\quad \Phi:U\times V\times\br^N\to\br, 
\ee
be a nondegenerate phase function \cite[Definition 11.1]{Grig1994}. The latter property means that 
\begin{enumerate}
\item $\Phi$ is positively homogeneous of degree one in $\Theta$: $\Phi(x,y,\la\Theta)=\la\Phi(x,y,\Theta)$, $\la>0$.
\item $\dd_{(x,y,\Theta)}\Phi$ does not vanish anywhere.
\item The differentials $\dd_{(x,y,\Theta)}\Phi_{\Theta_l}^\prime$, $1\le l\le N$, are linearly independent on the set $\Sigma$: 
\be\label{crit set general}
\Sigma:=\{(y,x,\Theta)\in V\times U\times (\br^N\setminus0_N):\ \Phi_\Theta^\prime(x,y,\Theta)=0_N\}.
\ee
\end{enumerate}
The phase $\Phi$ determines \textcolor{black}{the homogeneous canonical relation \cite[Chapter VIII, Definition 5.1]{trev2}}
\be\label{canon rel general}\begin{split}
\mathcal C:=&\{(y,\dd_y\Phi(x,y,\Theta));(x,-\dd_x\Phi(x,y,\Theta)):(x,y,\Theta)\in\Sigma\}\\
\subset& (T^*V\setminus 0_{2n})\times (T^*U\setminus 0_{2n}).
\end{split}
\ee
$I^r(V\times U\times\mathcal C)$ denotes the vector space of FIOs of order $r$ given by \cite[Section 11]{Grig1994}
\be\label{FIO def}\bs
&\coi(U)\ni f\to \frac1{(2\pi)^N}\int_{\br^N}\int_U \tilde B(x,y,\Theta) f(x)e^{i\Phi(x,y,\Theta)}\dd x\dd\Theta\in\CD^\prime(V),\\
&\tilde B\in S^{r+(n-N)/2}(U\times V\times\br^N).
\end{split}
\ee

For convenience, throughout the paper we use the following convention. If a constant $c$ is used in an equation, the qualifier ‘for some $c>0$’ is assumed. If several $c$ are used in a string of (in)equalities, then ‘for some’ applies to each of them, and the values of different $c$’s may all be different. 

Additional notation and conventions are introduced as needed.

\section{Setting of the problem, assumptions, and main results}\label{sec:setting}

\subsection{GRT and its properties}\label{ssec:grt}
Let $\Phi_1(x,y):\us\times\vs\to\br^2$ be a defining function for a GRT $\R$. \textcolor{black}{This is a function such that $\s_y:=\{x\in\us:\Phi_1(x,y)=0\}$ for any $y\in\vs$, and $\R$ integrates over $\s_y$.} An open set $\us\subset\br^2$ is the image domain. We assume $y=(\al,p)$, and $\vs=\I_\al\times \I_p$ is the data domain. 
\begin{assumptions}[Properties of the GRT -  I]\label{ass:Phi}$\hspace{1cm}$
\begin{enumerate}
\item\label{Ial} $\I_\al\subset\br$ is a \textcolor{black}{closed interval with the endpoints identified, so it can be viewed as a circle},
\item $\I_p\subset\br$ is an open interval (or all of $\br$), and 
\item\label{Phi sm} $\Phi_1(x,y)=p-\Phi(x,\al)$, $\Phi\in C^\infty(\us\times \I_\al)$. 
\end{enumerate}
\end{assumptions}
Both $\us$ and $\vs$ are endowed with the usual Euclidean metric. By a partition of unity we will always identify small subsets of $\vs$ with subsets of $\br^2$. The transform $\R$ integrates over smooth curves $\s_{(\al,p)}:=\{x\in\us: p=\Phi(x,\al)\}$:
\be\label{R def}\bs
\R f(y)=& \int_{\us} f(x)W(x,y)\de(p-\Phi(x,\al))\dd x\\
=& \frac1{2\pi}\int_{\br} \int_{\us} f(x)W(x,y)e^{i\nu(p-\Phi(x,\al))} \dd x\dd\nu,\ y\in\vs,
\end{split}
\ee
where $W\in C^\infty(\us\times\vs)$. Here and everywhere below we use the convention that whenever $y$, $\al$, and $p$ appear in the same equation or sentence, then $y=(\al,p)$. We assume $f$ is compactly supported, $\text{supp} f\subset\us$, and $f$ is sufficiently smooth, so $\hat f(y):=\R f(y)$ is a continuous function. Assumption~\ref{ass:Phi} implies that $\tsp\hat f$ is compact if $\tsp f$ is compact. 
\textcolor{black}{The smoothness of $\s_y$ follows from Assumption~\ref{geom GRT}\eqref{li} below.}

The critical set of the phase $\nu\Phi_1$ is
\be\label{crit set}
\Sigma:=\{(y,x,\nu)\in\vs\times\us\times(\br\setminus0):\ p-\Phi(x,\al)=0\}.
\ee
Since $\pa_\nu(\nu\Phi_1)=\Phi_1$ and $\dd_{(x,y,\nu)}\Phi_1\not=0$, the phase $\nu \Phi_1(x,\al)$ is clean and nondegenerate \cite[Definition 21.2.15]{hor3}. The phase $\nu \Phi_1(x,\al)$ determines the homogeneous canonical relation from $T^*\us$ to $T^*\vs$
\be\label{canon rel}\bs
\mathcal C:=&\{(y,\nu(-\dd_\al\Phi(x,\al),1));(x,\nu\dd_x\Phi(x,\al)):(y,x,\nu)\in\Sigma\}.
\end{split}
\ee
Also, $\R\in I^{-1/2}(\vs\times\us,\mathcal C)$ is an FIO of order $-1/2$ from $\us$ to $\vs$ associated with $\mathcal C$ \cite[Section VIII.5]{trev2}. Hence $\R:H_0^s(\us_b)\to H_{loc}^{s+(1/2)}(\vs)$ is continuous for any $s\in\br$ and any domain $U_b\Subset U$ \cite[Corollary 25.3.2]{hor4}. 

Next we compute the following determinant \cite[equations (6.3) and (6.4), Chapter VIII]{trev2}:
\be\label{Delta det}\bs
\Delta_\Phi(x,y,\nu):=&\text{det}
\bma [\nu\Phi_1(x,\al)]_{xy}^{\prime\prime} & [\nu\Phi_1(x,\al)]_{x\nu}^{\prime\prime}\\ [\nu\Phi_1(x,\al)]_{\nu y}^{\prime\prime} & [\nu\Phi_1(x,\al)]_{\nu \nu}^{\prime\prime} \ema\\
=&\text{det}
\bma -\nu\dd_x\Phi_\al^\prime & -\Phi_\al^\prime \\ 0\quad 0 & 1   \\ -\dd_x\Phi & 0\ema 
=\nu\text{det}
\bma \dd_x\Phi \\ \dd_x \Phi_\al^\prime \ema.
\end{split}
\ee
\textcolor{black}{
We will impose conditions to ensure that the last determinant is not zero. In this case, $\mathcal C$ is a local canonical graph \cite[Definition 6.1, Chapter VIII]{trev2}.} The derivatives on the right are computed at $(x,\al,p=\Phi(x,\al))$. As is easily seen,
\be\label{Delta det prop}
\Delta_\Phi(x,y,\nu)\textcolor{black}{=\nu\Delta_\Phi(x,y,1)=:}\nu\Delta_\Phi(x,\al).
\ee
Since the determinant on the last line in \eqref{Delta det} does not depend on $p$ and $\nu$, we use the simpler notation $\Delta_\Phi(x,\al)$ instead of $\Delta_\Phi(x,y,1)$.

\begin{assumptions}[Properties of the GRT - II]\label{geom GRT}$\hspace{1cm}$
\begin{enumerate}
\item\label{li} $\dd_x\Phi(x,\al)$ and $\dd_x\Phi_\al^\prime(x,\al)$ are linearly independent on $\us\times I_\al$.
\item\label{bolk} The Bolker condition: if $\Phi(x_1,\al)=\Phi(x_2,\al)$ for some $x_1,x_2\in\us$, $x_1\not=x_2$, and $\al\in I_\al$, then $\Phi_\al^\prime(x_1,\al)\not=\Phi_\al^\prime(x_2,\al)$. 
\item\label{visib} For each $(x,\xi)\in T^*\us$, there exist $\al\in \I_\al$ and $\nu\in\br$ such that $\xi=\nu\dd_x\Phi(x,\al)$.
\item\label{pos W} $W\in C^\infty(\us\times\vs)$ and $|W(x,y)|>c$ on $\us\times\vs$.
\item\label{invert} For any domain $\us_b\Subset\us$ there exists $c=c(\us_b)>0$ such that
\be\label{inverse bound}
\Vert f\Vert_{H^{-1/2}(\us)}\le c \Vert \R f\Vert_{L^2(\vs)}\text{  for any  } f\in H_0^{-1/2}(\us_b).
\ee
\end{enumerate}
\end{assumptions}

Assumption~\ref{geom GRT}\eqref{li} implies that $\Delta_\Phi(x,y,\nu)\not=0$ on $\us\times\vs\times(\br\setminus0)$. In turn, this implies that locally $\mathcal C$ is the graph of a diffeomorphism \cite[Section VI.4]{trev2}. Assumption~\ref{geom GRT}\eqref{bolk} implies that the natural projection $\mathcal C\to T^*\vs$ is an embedding (injective immersion). Assumption~\ref{geom GRT}\eqref{visib} asserts that any singularity of $f$ is visible from the data.  
Assumption~\ref{geom GRT}\eqref{invert} implies that $\R$ is boundedly invertible in the scale of Sobolev spaces. 

\color{black}
Let us discuss the last assumption in more detail. The most difficult part of the assumption is to check the injectivity of the GRT. There is no general theory that can ascertain whether a given GRT is injective or not. The question of injectivity is studied on a case by case basis. See, e.g., \cite{fsu-08, Homan2017} where the injectivity (and stability) are shown for some families of GRTs. If the injectivity is established, stability estimates of the kind \eqref{inverse bound} are obtained using the calculus of FIO, see \cite[proof of Theorem 2.3]{fsu-08} and \cite[proof of Theorem 2]{Homan2017}.

When the GRT integrates over simple curves, the injectivity is easier to establish. The injectivity of the classical Radon transform is well known \cite{nat3}. The case when the GRT integrates over circles has been studied as well, see \cite{AgrQ96, AmbK05} and references therein. 

Consider a circular GRT, $\R$. Suppose $\R$ integrates $C_0^\infty(\br^2)$ functions over circles of arbitrary radius, and their centers are confined to a set $\Xi$. One calls $\Xi$ a set of injectivity of $\R$ if $\R$ is injective. A complete description of sets of injectivity was first obtained in \cite{AgrQ96}. For example, the boundary of any bounded set in the plane is an injectivity set for $\R$ \cite[Corollary 5]{AmbK05}.
\color{black}

\subsection{The function to be reconstructed}

Suppose a compactly supported distribution, $f\in\CE^\prime(\us)$, is given by
\be\label{f-orig}
f(x)=\frac1{2\pi}\int_\br \tilde f(x,\la)e^{i\la H(x)}\dd\la,\ x\in\us,
\ee
where 
\be\label{f-lim}\begin{split}
&H\in C^\infty(\us),\quad \dd H\not=0 \text{  on  } \us;\\
&\tilde f(x,\la)= -i\tilde f_0(x)\la^{-1}+\tilde R(x,\la), \forall x\in \us,|\la|\ge 1;\\ 
& \tilde f(x,\la)\equiv0\ \forall x\in\us\setminus \us_b\text{ for some domain }\us_b\Subset \us;\\
& \tilde f\in S^{-1}(\us\times \br),\ 
\tilde f_0\in\coi(\us_b),\ \tilde R\in S^{-2}(\us\times \br), 
\end{split}
\ee 
for some $\tilde f_0$ and $\tilde R$. From \eqref{f-orig} and \eqref{f-lim} it follows that 
\be\label{ssup f}
\ssp(f)\subset \s:=\{x\in\us:\ H(x)=0\}. 
\ee
Thus $\s$ is a smooth curve (submanifold). 

Equations \eqref{f-orig} and \eqref{f-lim} imply that $f\in I^{-1}(\us,\s)$ is a conormal distribution \cite[Section 18.2]{hor3} and
\be\label{f-lim lo ev}\begin{split}
&f(x+\Delta x)\sim \tilde f_0(x)(\text{sgn} h/2),\ h:=\dd H(x)\cdot \Delta x,|\Delta x|\to0,\ \forall x\in\s.
\end{split}
\ee
Thus, $\tilde f_0(x)$ is the value of the jump of $f$ across $\s$ at $x\in\s$. \textcolor{black}{For this reason, in what follows we use the notation $\Delta f(x)$ instead of $\tilde f_0(x)$. The direction for computing the jump value across $\s$ is encoded in \eqref{f-orig}, \eqref{f-lim}.}

Pick any $x\in\s$ and $y\in\vs$ such that $\s_y$ is tangent to $\s$ at $x$. The curvatures of $\s$ and $\s_y$ at $x$ are given by
\be\label{curv S}
\varkappa_\s(x)=-\frac{\dd_x^2 H(x)(e,e)}{|\dd_x H(x)|},\quad
\varkappa_{\s_y}(x)=-\frac{\dd_x^2\Phi(x,\al)(e,e)}{|\dd_x \Phi(x,\al)|},
\ee
where $e$ is a unit vector tangent to $\s$ at $x$. If $\dd_x H(x)\cdot\dd_x \Phi(x,\al)<0$, the curvatures are not consistent with each other. In this case we flip the $p$-axis and replace $\Phi$ with $-\Phi$ to make them consistent. Thus, we can assume without loss of generality that 
\be\label{two vecs}
\dd_x H(x)\cdot\dd_x \Phi(x,\al)>0.
\ee

\begin{assumptions}[Properties of $f$]\label{ass:f props}$\hspace{1cm}$
\begin{enumerate}
\item\label{f first} $f$ satisfies \eqref{f-orig}, \eqref{f-lim}.
\item\label{del curv} For each pair $(x,y)\in\s\times\vs$ such that $\s_y$ is tangent to $\s$ at $x$, one has $\varkappa_{\s}(x)-\varkappa_{\s_y}(x)>0$.
\end{enumerate}
\end{assumptions}

\textcolor{black}{The requirement that the difference of the curvatures in Assumption~\ref{ass:f props}\eqref{del curv} be positive is not restrictive.} Basically, it says that $\varkappa_{\s}(x)\not=\varkappa_{\s_y}(x)$ whenever $\s$ and $\s_y$ are tangent at some $x$. If $\varkappa_{\s}(x)-\varkappa_{\s_y}(x)<0$, we change both functions: replace $H$ with $-H$, $\Phi$ with $-\Phi$, and flip the $p$-axis.

It may happen that for some $y\in\vs$, $\s_y$ is tangent to $\s$ at several distinct points $x_l$, and $\Phi$ cannot be made consistent with $H$ in the above sense at all $x_l$. Due to the linearity of the GRT, we can use a partition of unity and, for each $l$, adjust $\Phi$ and $H$ based solely on the pair $(x_l,y)$ independently of the adjustments at all the others pairs. This is always tacitly assumed in what follows.

\color{black}
Assumption~\ref{ass:f props}\eqref{del curv} imposes a nontrivial restriction on the types of functions $f$ we consider. The assumption prevents any $\s_y$ from being tangent to $\s$ of order higher than one, i.e., only simple tangency between $\s_y$ and $\s$ is allowed. In particular, $\s_y$ cannot be tangent to $\s$ along an entire curve segment. While the latter is generally not a significant issue in medical CT, such situations are more common in industrial applications of X-ray CT when scanning artificial (man-made) objects. In these cases, $\s_y$ are straight lines, and $\s$ may contain a flat segment.   

Mathematical analysis of these cases presents additional challenges. This is because $\hat f(y)$ is generally more singular in a neighborhood of the exceptional $y$ described above than in the case of simple tangency. An analysis of these more singular cases is beyond the scope of the paper.
\color{black}

\subsection{GRT data and the reconstruction algorithm}\label{ssec:data_recon}

Discrete data $g(y_j)$ are given at the points
\be\label{data-pts}
y_j=(\Delta\al j_1,\Delta p j_2)\in\vs,\ \Delta p=\e,\ \Delta\al=\mu\e,\ j=(j_1,j_2)\in\mathbb Z^2,
\ee
The interpolated data, denoted $\hat f_\e(y)$, is computed by: 
\be\label{interp-data}\bs
&\hat f_\e(y):=\sum_{y_j\in\vs} \ik_\e(y-y_j)\hat f(y_j),\ y\in\vs,\\
&\ik_\e(y):=\ik(y/\e),\ \ik(y):=\ik_\al(\al/\mu)\ik_p(p),
\end{split}
\ee
where $\ik$ is an interpolation kernel. 
\begin{assumptions}[Properties of the interpolation kernel, $\ik$]\label{ass:interp ker} 
Let $\ik_*$ denote any of the functions $\ik_\al$ and $\ik_p$. One has
\begin{enumerate}
\item\label{ikcont} $\ik_*\in C_0^2(\br)$.
\item\label{ikeven} $\ik_*$ is even: $\ik_*(u)=\ik_*(-u)$, $u\in\br$.
\item\label{ikexact} $\ik_*$ is exact up to order one, i.e. 
\be\label{exactness}
\sum_{j\in\mathbb Z} j^m\ik_*(u-j)\equiv u^m,\ m=0,1,\ u\in\br.
\ee
\end{enumerate}
\end{assumptions}

Assumption~\ref{ass:interp ker}\eqref{ikexact} with $m=0$ implies that $\ik_*$ is normalized: 
\be\label{norm-deriv}\begin{split}
1&=\int_0^1\sum_{j\in\mathbb Z}\ik_*(u-j)\dd u=\int_\br\ik_*(u)\dd u.
\end{split}
\ee

Let $f$ satisfy Assumption~\ref{ass:f props} and $\R$ satisfy Assumptions~\ref{ass:Phi}, \ref{geom GRT}. 
By Lemma~\ref{lem:conorm} below, $\Vert \hat f_\e\Vert_{L^2(\vs)}\le c$, $0<\e\ll1$. 
Reconstruction is achieved by minimizing the quadratic functional
\be\label{min pb}
f_\e=\text{argmin}_{f\in H_0^1(\us_b)}\Psi(f),\ 
\Psi(f):=\Vert \R f-\hat f_\e \Vert_{L^2(\vs)}^2+\kappa \e^{3}\Vert \pa_x f\Vert_{L^2(\us)}^2,
\ee
with some fixed regularization parameter $\kappa>0$ and with $\us_b$ the same as in \eqref{f-lim} (or with any other $\us_b^\prime$ that satisfies $\us_b\subset\us_b^\prime\Subset U$). 

The following result is proven in Appendix~\ref{sec: prf lem unique sol}.

\begin{lemma}\label{lem:unique sol} Suppose Assumptions~\ref{ass:Phi}--
\ref{ass:interp ker} are satisfied. The solution to \eqref{min pb}, $f_\e$, exists and is unique for each $0<\e\ll1$.
\end{lemma}

Our first main result is the following theorem, which is proven in Sections~\ref{GRT-behavior}--\ref{sec:DTB fn}.

\begin{theorem}\label{thm:main res} Pick any $x_0\in\s$ and $A_0>0$. Suppose there is only one $y_0=(\al_0,p_0)\in\vs$ such that $\s_{y_0}$ is tangent to $\s$ at $x_0$. Suppose Assumptions~\ref{ass:Phi}--\ref{ass:interp ker} are satisfied and $\mu\Phi_\al^\prime(x_0,\al_0)$ is irrational. One has
\be\label{DTB eq}
\lim_{\e\to0^+}\big(f_\e(x_0+\e\chx)-f_\e(x_0)\big)=\Delta f(x_0)\Upsilon\big(\vec\Theta_0\cdot\chx\big),\
|\chx|<A_0,
\ee
where the limit is uniform in $\chx$ and
\be\label{DTB aux}\bs
&\Upsilon(r)=\int_{\br} h_0(u)\int_{u_1}^{u_1+r}R(t)\dd t\dd u,\ \vec\Theta_0:=\frac{\dd_x \Phi(x_0,\al_0)}{|\dd_x \Phi(x_0,\al_0)|},\\
&h_0(u)=\int_\br \ik_\al(s)\ik_p(\mu s\Phi_\al^\prime(x_0,\al_0)+u)\dd s,\ u_1=\frac{u}{|\dd_x \Phi(x_0,\al_0)|},\\
&R(t)=\CF_{\la\to t}^{-1}\left(\left[1+\frac{\kappa}{2\pi\rho_0}|\la|^3\right]^{-1}\right),\
\rho_0=\frac{W^2(x_0,y_0)|\dd_x \Phi(x_0,\al_0)|}{|\Delta_\Phi(x_0,\al_0)|}.
\end{split}
\ee
Moreover, 
\be
\Upsilon(0)=0,\quad \Upsilon(\pm\infty)=\pm1/2.
\ee
Thus $\Upsilon(+\infty)-\Upsilon(-\infty)=1$, and $\Upsilon$ is indeed a DTB function.
\end{theorem}

In the proof of Theorem~\ref{thm:main res} we assume without loss of generality that $\al_0=0$. In the case of redundant data we obtain the following generalization, which is proven in Section~\ref{sec:red dat}.

\begin{theorem}\label{thm:main res mult} Pick any $x_0\in\s$ and $A_0>0$. Suppose there exist finitely many $y_l=(\al_l,p_l)\in\vs$ such that $\s_{y_l}$ is tangent to $\s$ at $x_0$. Suppose Assumptions~\ref{ass:Phi}--
\ref{ass:interp ker} are satisfied and all $\mu\Phi_\al^\prime(x_0,\al_l)$ are irrational. One has
\be\label{DTB eq mult}
\lim_{\e\to0^+}\big(f_\e(x_0+\e\chx)-f_\e(x_0)\big)=\Delta f(x_0)\frac{\tsum_l\nu_l\Upsilon_l\big(\vec\Theta_0\cdot\chx\big)}{\tsum_l\nu_l},\ |\chx|<A_0,
\ee
where the limit is uniform in $\chx$ and
\be\label{DTB aux mult}\bs
&\Upsilon_l(r)=\int_{\br} h_l(u)\int_{u_l}^{u_l+r}R(t)\dd t\dd u,\ \vec\Theta_0:=\frac{\dd_x H(x_0)}{|\dd_x H(x_0)|},\\
&h_l(u)=\int_\br \ik_\al(s)\ik_p(\mu s\Phi_\al^\prime(x_0,\al_l)+u)\dd s,\ u_l=\frac{u}{|\dd_x \Phi(x_0,\al_l)|},\\
&R(t)=\CF_{\la\to t}^{-1}\bigg(\bigg[1+\frac{\kappa}{2\pi\textstyle\sum_l\nu_l}|\la|^3\bigg]^{-1}\bigg),\
\nu_l:=\frac{W^2(x_0,y_l)|\dd_x \Phi(x_0,\al_l)|}{|\Delta_\Phi(x_0,\al_l)|}.
\end{split}
\ee
\end{theorem}

\color{black}
Taken together, Theorems~\ref{thm:main res} and \ref{thm:main res mult} provide an easy to use characterization of the reconstructed image in an $O(\e)$-size neighborhood of an edge. This characterization is in terms of universal functions $\Upsilon_0$ and $\Upsilon_l$, which are independent of $f$ and can be easily precomputed. The only contribution from $f$ is $\Delta f(x_0)$, the value of the jump of $f$ at $x_0$, which enters the formulas \eqref{DTB eq} and \eqref{DTB eq mult} as a multiplicative factor. Although the proofs are quite extensive, the final result is satisfying due to its simple form and straightforward applicability.
\color{black}

\color{black}
\subsection{Example}
To briefly illustrate the above theorems, we consider the classical Radon transform in the plane that integrates over the lines $\s_{(\al,p)}=\{x\in\br^2:\vec\al\cdot x=p\}$, $\vec\al:=(\cos\al,\sin\al)$. Then $\I_\al=[0,2\pi)\sim S^1$, $\I_p=\br$, $\Phi(x,\al)=\vec\al\cdot x$, $W(x,y)\equiv 1$, and
\be\label{CRT stuff}\begin{split}
&\dd_x\Phi(x,\al)=\vec\al,\ |\dd_x\Phi(x,\al)|=1,\ \Phi_\al^\prime(x,\al)=\vec\al^\perp\cdot x,\\ 
&\Delta_\Phi=\det\bma\vec\al \\ \vec\al^\perp\ema=1,\ \vec\al^\perp:=(-\sin\al,\cos\al).
\end{split}
\ee

Pick any $x_0\in\s$. Let $\vec\al_0$ be perpendicular to $\s$ at $x_0$. Strictly speaking, there are two data points $y_1=(\al_0,\vec\al_0\cdot x_0)$ and $y_2=(\al_0+\pi,-\vec\al_0\cdot x_0)$ such that $\s_{y_l}$ is tangent to $\s$ at $x_0$. Hence we should use Theorem~\ref{thm:main res mult}.

From \eqref{DTB aux mult} and \eqref{CRT stuff}, $u_l=u$, $\nu_l=1$, and
\be\label{DTB aux mult CRT}\bs
&\Upsilon_l(r)=\int_{\br} h_l(u)\int_u^{u+r}R(t)\dd t\dd u,\\
&h_l(u)=\int_\br \ik_\al(s)\ik_p(\mu s\vec\al_l^\perp\cdot x_0+u)\dd s,\ l=1,2,\\
&R(t)=\CF_{\la\to t}^{-1}\bigg(\big[1+\frac{\kappa/2}{2\pi}|\la|^3\big]^{-1}\bigg).
\end{split}
\ee
By Assumption~\ref{ass:interp ker}\eqref{ikeven}, $\ik_\al$ is even and $\vec\al_1^\perp\cdot x_0=-\vec\al_2^\perp\cdot x_0$. Therefore, $h_1(u)\equiv h_2(u)$ and (cf. \eqref{DTB eq mult})
\be
\frac{\tsum_l\nu_l\Upsilon_l(r)}{\tsum_l\nu_l}
=\int_{\br} h_1(u)\int_u^{u+r}R(t)\dd t\dd u.
\ee
This is the formula we would have obtained by formally using Theorem~\ref{thm:main res} and replacing $\kappa$ with $\kappa/2$. The result is not surprising, as $\s_{y_1}$ and $\s_{y_2}$ define the same tangent line. 

See Section~\ref{sec:numerics} for an example involving the GRT that integrates over circles. The functions $\Upsilon_l$ have been used to compute the predicted behavior of the reconstruction (the DTB function), which is shown as a blue curve in Figure~3.

\color{black}

\section{Behavior of $\R f$ near its singular support}\label{GRT-behavior}

Let $\Gamma$ be the set of all $y\in\vs$ such that $\s_y$ is tangent to $\s$. Since $\R$ is an FIO with the canonical relation \eqref{canon rel}, we have
\be\label{gamma can rel}
\Gamma=\ssp \hat f,\ \WF(\hat f)\subset N^*\Gamma=\mathcal C\circ N^*\s.
\ee
Here $N^*\s$ is the conormal bundle of $\s$, and similarly for $\Gamma$ \cite[section I.6]{trev1}. In particular, $\hat f$ is smooth away from $\Gamma$. In this section we describe the leading singular behavior of $\hat f$ near $\Gamma$.

Fix any $\tilde y\in\Gamma$, and let $\tilde x\in\s$ be the corresponding point of tangency, see Figure~\ref{fig:tangency}. Let $V$ be a sufficiently small neighborhood of $\tilde y$. Using a partition of unity, the linearity of $\R$, and Assumption~\ref{ass:f props}\eqref{del curv}, we can assume without loss of generality that $\s_y$, $y\in\Gamma\cap V$, is tangent to $\s$ only at one point. \textcolor{black}{We underscore that the main goal of this section is to study the singularities of $\R f$. The way these singularities are mapped into the reconstructed image is studied in Sections~\ref{sec:fplt}, \ref{sec:DTB fn}, and \ref{sec:red dat}.}

Substitute \eqref{f-orig} into \eqref{R def}. The asymptotics as $\la\to\infty$ of the resulting integral with respect to $x$ is computed by the stationary phase method \cite[Chapter VIII, eqs. (2.14)--(2.20)]{trev2} 
\textcolor{black}{using the properties of the Hessian at the stationary point established in Lemma~\ref{lem: matr M} below}:
\be\label{GRT-f-inner}\begin{split}
\hat f(y)=&\frac1{2\pi}\int_{\br} \tilde G(y,\la)\dd\la,\ y\in V,\\
\tilde G(y,\la):=&\frac1{2\pi}\int_{\br}\int_{\us} \tilde f(x,\la)e^{i\la H(x)}W(x,y)e^{i\nu(p-\Phi(x,\al))}\dd x\dd\nu\\
=&\frac{|\la|}{2\pi}\int_{\br}\int_{\us} W(x,y)\tilde f(x,\la)e^{i\la[\nu_1(p-\Phi(x,\al))+ H(x)]}\dd x\dd\nu_1\\
=&\left(\tilde f(x_*,\la)\frac{W(x_*,y)}{|\det M(y)|^{1/2}}
\left(\frac{2\pi}{|\la|}\right)^{1/2}e(-\text{sgn}\la/2)+\tilde R(y,\la)\right)\\
&\quad\times e^{i\la H(x_*)},\ |\la|\ge1,\ \tilde R\in S^{-3/2}(\vs\times\br),\ x_*=x_*(y),
\end{split}
\ee
for some $\tilde R$. Here $e(t):=\exp(i\pi t/2)$, and 
\be\label{M def}\begin{split}
&M(y)=\bma 0 & -\dd_x\Phi(x,\al)\\ (-\dd_x\Phi(x,\al))^T & \dd_x^2\big[H(x,\al)-\nu_1\Phi(x,\al)\big]\ema,\\ 
&\nu_1=\nu_{1*}(y),\ x=x_*(y),
\end{split}
\ee
is the Hessian matrix of the phase at the stationary point $(\nu_{1*}(y),x_*(y))$, which is found by solving
\be\label{stat in x}
p=\Phi(x,\al),\ \dd_x H(x)-\nu_1\dd_x\Phi(x,\al)=0,\ y\in V.
\ee
By construction, $x_*(\tilde y)=\tilde x$. As before, the existence, local uniqueness, and smoothness of the solution follows from Assumptions~\ref{geom GRT}(\ref{li},\ref{visib}) \textcolor{black}{and \ref{ass:f props}\eqref{del curv}}.

\textcolor{black}{Given a nondegenerate symmetric matrix $M$, let $n_+$ and $n_-$ denote the number of positive and negative eigenvalues of $M$, respectively. Then the signature of $M$, denoted $\text{sgn}\,M$, is the difference $n_+-n_-$.} The following result is proven in Appendix~\ref{sec: prf lem matr M}.

\begin{lemma}\label{lem: matr M} Suppose Assumptions~\ref{ass:Phi}--\ref{ass:f props} are satisfied. 
Let $\s_y$ be tangent to $\s$ at $x$. Let $\varkappa_{\s}(x)$ and $\varkappa_{\s_y}(x)$ denote the curvatures of $\s$ and $\s_y$ at $x$, respectively (see \eqref{curv S}). One has
\be\label{det sign}\bs
\det M(y)=&(\varkappa_{\s}(x)-\varkappa_{\s_y}(x)){|\dd_x \Phi(x,\al)|^2}{|\dd_x H(x)|}>0,\\
\text{sgn}\,M(y)=&-1.
\end{split}
\ee
\end{lemma}

It follows from the proof of Lemma~\ref{lem: matr M} (see \eqref{dHdp}) that we can locally solve $H(x_*(\al,p))=0$ for $p$ in terms of $\al$: $p=P(\al)$, where $P(\al)$ is locally smooth. Therefore 
\be\label{Pdef}
H(x_*(y))=\psi(y)(p-P(\al)),\ y\in V,
\ee
where $\psi(y)\not=0$, $y\in V$, and $\psi$ is locally smooth. Moreover, by \eqref{dHdp} and \eqref{Pdef},
\be\label{psi val}\bs
\psi(\al,P(\al))=&\pa_p H(x_*(\al,p))|_{p=P(\al)}=\frac{|\dd_x H(x_*(y))|}{|\dd_x \Phi(x_*(y),\al)|}>0,\\ 
(\al,P(\al))\in & V.
\end{split}
\ee
Therefore, $\psi(y)>0$, $y\in V$.

Instead of finding $P(\al)$ via the intermediary function $x_*(y)$, we can find $P(\al)$ by locally solving 
\be\label{for Q 1}
H(x(\al))=0,\ \dd_x H(x(\al))= \nu_1(\al)\dd_x \Phi(x(\al),\al),
\ee
for $\nu_1(\al)$, $x(\al)$ and setting $P(\al):=\Phi(x(\al),\al)$. Then $\nu_1(\al)\equiv\nu_{1*}(\al,P(\al))$. Since $x(\al)\in\s$, we can say that $P(\al)$ and the segment of $\Gamma$ it defines are associated with a segment of $\s$ (see Figure~\ref{fig:tangency}). 
If we have in mind one specific $\tilde x\in\s$, then we informally say that $P(\al)$ and the corresponding segment of $\Gamma$ are associated with $\tilde x$ (i.e., with a segment of $\s$ containing $\tilde x$).

\begin{figure}[h]
{\centerline{
{\epsfig{file={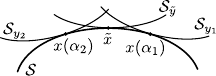}, width=5.5cm}}
}}
\caption{Illustration of $x(\al)$. In the figure, $y_k=(\al_k,P(\al_k))$, $k=1,2$ and $\tilde y=(\tilde\al,P(\tilde\al))$.}
\label{fig:tangency}
\end{figure}

By \eqref{GRT-f-inner}--\eqref{Pdef}, $\hat f$ can be written as
\be\label{GRT-f-alt}
\hat f(y) =\frac1{2\pi}\int_\br \fs(y,\la)e^{i\la (p-P(\al))}\dd\la,\ y\in V,
\ee
and, with some $c_m$, $\tilde R$, 
\be\label{Rf-lim-coefs}\begin{split}
&\fs(y,\la)=\fs_0(y)e(-3/2)(\la-i0)^{-3/2}+\tilde R(y,\la),\ |\la|\ge1;\\ 
&\fs\in S^{-3/2}(V\times\br);\ \tilde R\in S^{-5/2}(V\times\br);\\
&\fs_0(y)=(2\pi)^{1/2}\frac{W(x_*(y),y)\Delta f(x_*(y))\psi(y)^{1/2}}{|\det M(y)|^{1/2}}\in C^\infty(V).
\end{split}
\ee 
The function $\psi(y)$ is absorbed into the frequency variable $\la$ leading to the relation
\be\label{tvG}
\fs(y,\la)=\tilde G(y,\la/\psi(y))/\psi(y),
\ee
where $\tilde G$ is defined in \eqref{GRT-f-inner}. To obtain the formula for $\fs_0$ we use that the leading order term in $\fs$ is homogeneous of degree $-3/2$ and the definition of the distribution $(\la-i0)^a$ \cite[equations (28) and (30), p. 336]{gs}.

By \eqref{GRT-f-alt}, 
\be
\ssp\hat f\cap V\subset\Gamma:=\{(\al,p)\in V:p=P(\al)\}.
\ee
All of the above properties of $\hat f$ follow from the FIO calculus \cite[Section VIII.5]{trev2}, but for our purposes it is more convenient to provide their explicit derivation \eqref{GRT-f-inner}--\eqref{tvG}.

\color{black}
As mentioned above, Assumptions~\ref{ass:Phi}(\ref{Ial},\ref{Phi sm}) and \ref{ass:f props}(\ref{f first}) imply $\tsp \hat f$ is compact. Consider an open cover of $\tsp \hat f$ by sufficiently small open sets $V_l:=I_\al^l\times I_p^l$, $1\le l\le L$, for some  large $L$. Here $I_\al^l$ and $I_p^l$ are intervals. Using a partition of unity subordinate to this cover, we can consider one such set, $V_l$. Generally, the intersection $\Gamma\cap V_l$ is a single smooth curve segment.
If $\Gamma$ self-intersects in $V_l$, we formally duplicate $V_l$ as many times as needed to make sure each intersection $\Gamma\cap V_l$ is identified with no more than one smooth curve segment, denoted $\Gamma_l$. 

Let $p=P_l(\al)$, $\al\in I_\al^l$, be the local parameterization for $\Gamma_l$. Unless stated otherwise, we consider only one set from the partition, denoted $V=I_\al\times I_p$. The corresponding $P_l(\al)$ is denoted $P(\al)$. Also, we denote $J_\al:=\{j_1\in\mathbb Z:\al_{j_1}\in I_\al\}$, and similarly for $J_p$. There are $O(1/\e)$ points in $J_\al$ and $J_p$.
\color{black}

The following result is proven in Appendix~\ref{sec:prf conorm}.

\begin{lemma}\label{lem:conorm} Suppose Assumptions~\ref{ass:Phi}--
\ref{ass:interp ker} are satisfied. Pick any $\chi_V\in\coi(V)$. Let $g:=\chi_V\hat f$. Then
$g,g_\e\in C_0^{1/2}(V)$ and $\Vert g_\e\Vert_{C^{1/2}(V)}<c$, $0<\e\ll1$. Also,
\be\label{g-lead-sing}\bs
&g(y)=a_0(y)(p-P(\al))_+^{1/2}+\Delta g(y),\\
&a(y)=2\pi^{-1/2}\fs_0(y),\  a\in\coi(V),\ \Delta g\in C_0^{3/2}(\vs),
\end{split}
\ee
and $\Vert \Delta g_\e\Vert_{C^{3/2}(V)}<c$ for all $0<\e\ll1$, \textcolor{black}{where $\fs_0(y)$ is given in \eqref{Rf-lim-coefs}. Here
\be
g_\e(y):=\tsum_{j\in V} \ik_\e(y-y_j)g(y_j),
\ee
and $\Delta g_\e$ is defined analogously, see \eqref{interp-data}.} 
\end{lemma}

\textcolor{black}{If $\Gamma\cap V=\varnothing$, we assume in \eqref{g-lead-sing} that $a_0\equiv0$.}

\section{Regularized solution}\label{sec:reg_sol}

\subsection{Equation for the solution to \eqref{min pb}, $f_\e$}
Let $\R^*$ be the adjoint of $\R:L^2(\us)\to L^2(\vs)$:
\be\label{Radj def}\bs
(\R^* g)(x)=& \int_{\vs} g(y)W(x,y)\de(p-\Phi(x,\al)) \dd y\\
=&\int_{\I_\al} g(\al,\Phi(x,\al))W(x,(\al,\Phi(x,\al)))\dd \al,\ x\in\us.
\end{split}
\ee
As is well-known, Assumption~\ref{geom GRT} ensures that $\R^*\R$ is an elliptic pseudo-differential operator $(\Psi$DO) \cite{qu-80, Holm24}, \cite[Section VIII.6.2]{trev2} of order -1 in $\us$, i.e. $\R^*\R\in L^{-1}(\us)$. 

Clearly, the functional $\Psi$ in \eqref{min pb} is Fr\'echet-differentiable. Indeed, one has 
\be\bs
&\Psi(f+h)-\Psi(f)=\Psi^\prime(f;h)+O(\Vert h\Vert_{H^1(\us)}^2),\ f,h\in H_0^1(\us_b),\\
&(1/2)\Psi^\prime(f;h)=(\R f-\hat f_\e,\R h)_{L^2(\vs)}+\kappa\e^3(\pa_x f,\pa_x h)_{L^2(\us)},
\end{split}
\ee
and the functional $H_0^1(\us_b)\ni h\to \Psi^\prime(f;h)$ is linear and continuous \cite[p. 14]{Peyp2015}. Recall that $\hat f_\e$ is defined in \eqref{interp-data}. Then the subgradient of $\Psi$ coincides with the gradient \cite[Proposition 3.20]{Peyp2015}, and $f$ is the minimizer if and only if the gradient equals zero \cite[Theorem 3.24]{Peyp2015}. The fact that $\Psi$ is proper and convex, which is required for the latter conclusion, is established in the proof of Lemma~\ref{lem:unique sol} in Appendix~\ref{sec: prf lem unique sol}. Therefore the solution to \eqref{min pb} is the unique solution to the following equation, which is equivalent to the first order optimality condition: 
\be\label{reg main eq}
(\R^*\R-\kappa\e^3\Delta)f_\e=\R^*\hat f_\e,\ x\in\us_b,\ f_\e\in H_0^1(\us_b).
\ee
This follows by setting $\Psi^\prime(f_\e;h)\equiv0$, $h\in \coi(\us_b)$, integrating by parts, and using that $\coi(\us_b)$ is dense in $H_0^1(\us_b)$.

\begin{lemma}\label{lem:bdd sol} Suppose Assumptions~\ref{ass:Phi}--
\ref{ass:interp ker} are satisfied, and $f_\e$ is obtained by solving \eqref{min pb} or, equivalently, \eqref{reg main eq}. One has 
\be
\Vert f_\e\Vert_{H^{-1/2}(\us)}\le c,\ 0<\e\ll1.
\ee
\end{lemma}
\begin{proof} 
Since $\Psi(f\equiv0)=\Vert \hat f_\e \Vert_{L^2(\vs)}^2$ (cf. \eqref{min pb}), the solution $f_\e$ satisfies 
$\Vert \R f_\e\Vert_{L^2(\vs)}\le 2\Vert \hat f_\e\Vert_{L^2(\vs)}<c$, $0<\e\ll1$ (cf. the paragraph preceding \eqref{min pb}). 
For clarity, $\R f_\e$ is the GRT of the solution $f_\e$, while $\hat f_\e$ is the interpolated version of the exact discrete data, i.e. $\hat f_\e=(\R f)_\e$. By construction, $f_\e\in H_0^1(\us_b)\subset H_0^{-1/2}(\us_b)$, where $\us_b$ is bounded. Application of \eqref{inverse bound} completes the proof.
\end{proof}

Let $\tilde D$ be the complete symbol of $\R^*\R$, i.e. $\R^*\R=\op(\tilde D(x,\xi))$ (see \cite[Theorem 3.4]{Grig1994}). Rewrite \eqref{reg main eq} in an alternative form:
\be\bs\label{to scl PDO 1}
\big(\op(\tilde D(x,\xi)+\kappa\e^3|\xi|^2)f_\e\big)(x)
=(\R^* \hat f_\e)(x),\ x\in\us_b.
\end{split}
\ee
Let $\us_0$ be a sufficiently small neighborhood of $x_0\in\us_b$. Fix any $\chi\in\coi(\us_b)$ such that $\chi(x)\equiv1$ in a neighborhood of the closure of $\us_0$ (see Figure~\ref{fig:Us} in Appendix~\ref{sec:prf lem femge}). Apply $\op(\chi(x)|\xi|)\chi(x)$ on both sides of \eqref{to scl PDO 1}:
\be\label{eqn checks oper 2}\bs
&\op(\chi(x)|\xi|)\chi(x)\op\big(\tilde D(x,\xi)+\kappa\e^3|\xi|^2\big)f_\e
=\op(\chi(x)|\xi|)\chi(x)\R^* \hat f_\e=:F_\e.
\end{split}
\ee

Let $\tilde D_0(x,\xi):=\lim_{\la\to\infty}|\la\xi|\tilde D(x,\la\xi)$. The limit exists and $\tilde D_0$ is  homogeneous of degree 0 in $\xi$ (see \eqref{R*R D0} below). The following lemma is proven in Appendix~\ref{sec:prf lem femge}.
\begin{lemma}\label{lem:femge} Suppose Assumptions~\ref{ass:Phi}--\ref{ass:interp ker} are satisfied. Define 
\be\label{lead sing fe}
\G_\e:=\op\big(\big[\tilde D_0(x,\xi)+\kappa|\e\xi|^3\big]^{-1}\big) F_\e.
\ee
Fix any $A>0$. One has:
\be\label{fmG}\bs
&|(f_\e-\G_\e)(x)|\le c,\ x\in \us_0,\ 0<\e\ll 1,\\
&|(f_\e-\G_\e)(x_0+\e\chx)-(f_\e-\G_\e)(x_0)|\le c\e^{1/2},\ 0<\e\ll 1,\ |\chx|<A.
\end{split}
\ee
\end{lemma}

\section{Proving the boundedness of $F_\e$}\label{sec:RHS bdd}

In this section we prove the following key result, which is used in the proof of Lemma~\ref{lem:femge}.

\begin{lemma}\label{lem:F bnd final} Suppose Assumptions~\ref{ass:Phi}--\ref{ass:interp ker} are satisfied. Then $\Vert F_\e\Vert_{L^\infty(\us)}\le c$ for all $0<\e\ll1$.
\end{lemma}

\color{black}
By the linearity of \eqref{eqn checks oper 2}, to prove the lemma it suffices to consider $F_\e$ as defined on the right in \eqref{eqn checks oper 2}, but with $\hat f_\e$ replaced by $g_\e$ using a local piece of the data $g=\chi_V\R f$ introduced in Lemma~\ref{lem:conorm}.

\subsection{Preliminary calculations}\label{ssec:prelims}
Our goal is to show that the right-hand side of \eqref{eqn checks oper 2} is bounded.
Rewrite \eqref{eqn checks oper 2} as follows. 
\be\label{get lead term}\bs
F_\e(x)&=(\T g_\e)(x),\ \T:=\op(\chi(x)|\xi|)\chi(z)\R^*,\\
(\T g_\e)(x)&=\frac{\chi(x)}{(2\pi)^3}\int_{\br^2}\int_{\us}|\xi|e^{i\xi\cdot(z-x)}\chi(z)\\
&\hspace{1cm}\times\frac1{2\pi}\int_{\br}\int_{V} W(z,y)
e^{i\nu(p-\Phi(z,\al))}g_\e(y)\dd y\dd\nu\dd z\dd\xi.
\end{split}
\ee
Next, consider the integral
\be\label{symb all}\bs
K_1(x,\al,\nu):=\frac{\chi(x)}{(2\pi)^3}\int_{\br^2}\int_{\us}\chi(z)W(z,y)|\xi|e^{i\xi\cdot(z-x)-i\nu(\Phi(z,\al)-\Phi(x,\al))}\dd z\dd\xi.
\end{split}
\ee
Upon changing variables $\xi\to\xi_1=\xi/\nu$, the stationary point of the phase satisfies $z=x$, $\xi_1=\dd_x\Phi(x,\al)$, and the stationary point is nondegenerate. By the stationary phase method, 
\be\label{K_1 K}\bs
&K_1(x,\al,\nu)=|\nu|K(x,\al)+\Delta K(x,\al;\nu),\\ 
&K(x,\al):=\chi^2(x)|\dd_x \Phi(x,\al)|W(x,(\al,\Phi(x,\al))),\ x\in \us,\\
&K\in \coi(\us_b\times\I_\al),\ \Delta K\in \coi(\us_b\times\I_\al\times(\br\setminus 0)),\ \Delta K \in S^0(\br).
\end{split}
\ee
By \eqref{get lead term}--\eqref{K_1 K},
\be\label{get lead term v2}\bs
(\T g_\e)(x)=&\int_{I_\al} K(x,\al) (\Lambda g_\e)(\al,\Phi(x,\al))\dd\al\\
&+\int_{I_\al} (\Delta\mathcal K_{x,\al} g_\e)(\al,\Phi(x,\al))\dd\al,\ x\in\us.
\end{split}
\ee
Here $\Lambda\in L^1(\vs)$, and $\mathcal K_{x,\al}\in L^1(\vs)$ is a $\Psi$DO, which depends smoothly on $(x,\al)\in\us_b\times\I_\al$, 
\be\label{two PDOs}\bs
(\Lambda g)(\al,t)=&\frac1{2\pi}\int_{\I_p}|\nu|e^{i\nu(p-t)}g(\al,p)\dd p\dd\nu,\\
(\mathcal K_{x,\al^\prime}g)(\al,t)=&\frac1{2\pi}\int_{\I_p}\Delta K(x,\al^\prime;\nu)e^{i\nu(p-t)}g(\al,p)\dd p\dd\nu,\ (\al,t)\in \vs.
\end{split}
\ee

By Lemma~\ref{lem:conorm}, 
\be\bs
&g_\e\in C_0^{1/2}(V),\ \Delta g_\e\in C_0^{3/2}(V),\\ 
&\Vert g_\e\Vert_{C^{1/2}(V)}<c,\ \Vert \Delta g_\e\Vert_{C^{3/2}(V)}<c,\ 0<\e\ll1. 
\end{split}
\ee
Recall that $B_{\infty,\infty}^s(\br^n)=C^s(\br^n)$, $s>0$, $s\not\in\N$ \cite[Remark 6.4(2)]{Abels12}. 
By the continuity of $\Psi$DOs in H{\"o}lder-Zygmund spaces (see \cite[Theorem 6.19]{Abels12}), $\Vert \Lambda \Delta g_\e\Vert_{C^{1/2}(V_1)}<c$, $\Vert \Delta \mathcal K_{x,\al^\prime} g_\e\Vert_{C^{1/2}(V_1)}<c$, $0<\e\ll1$, for any compact $V_1\subset\vs$. It is now obvious that
\be\label{get lead term v3}\bs
\int_{I_\al} (\Lambda \Delta g_\e)(\al,\Phi(x,\al))\dd\al
+\int_{I_\al} (\Delta\mathcal K_{x,\al} g_\e)(\al,\Phi(x,\al))\dd\al\in C_0^{1/2}(\us),
\end{split}
\ee
with its $C^{1/2}(\us)$ norm uniformly bounded for $0<\e\ll1$.

\color{black}
Therefore, all that remains is to bound the following expression:
\be\label{ker B}\bs
F_\e^{(l)}(x):=&-\frac1\pi\int_{V}\frac{\chii(p-\Phi(x,\al)) K(x,\al)}{p-\Phi(x,\al)}\\
&\times\sum_{y_j\in V} \pa_p \ik_\e(y-y_j)a_0(y_j)(p_{j_2}-P(\al_{j_1}))_+^{1/2}\dd y,\ x\in\us_b.
\end{split}
\ee
Recall that $K(x,\al)\equiv0$ for any $x\in\us\setminus\us_b$, cf. \eqref{K_1 K}. To get \eqref{ker B} from \eqref{get lead term v2}, \eqref{two PDOs}, we used the identity $\op(|\nu|)=-1/(\pi t^2)$, integrated by parts with respect to $p$, and inserted a 1D cutoff function $\chii$. The latter identically equals one in a neighborhood of zero. Also, we omitted smooth terms with uniformly bounded $\Vert\cdot\Vert_{C^1(\us)}$ norm, $0<\e\ll1$, which arise due to the cutoff. 

Summarizing the above results yields
\be\label{Fe interm}\bs
&\Vert F_\e\Vert_{L^\infty(\us_b)}\le \Vert F_\e^{(l)}\Vert_{L^\infty(\us_b)}+c,\ 0<\e\ll1,\\
&F_\e(x_0+\e\chx)-F_\e(x_0)= F_\e^{(l)}(x_0+\e\chx)-F_\e^{(l)}(x_0)+O(\e^{1/2}),\ |\chx|<A,
\end{split}
\ee
for any fixed $A>0$. The $O(\e^{1/2})$ term is uniform in $\chx$.

\subsection{Simplification of $F_\e^{(l)}$, the leading order term of $F_\e$}\label{ssec:simple Fl}

In what follows we use the function
\be\label{F step 2}\bs
&Q(x,\al):=\Phi(x,\al)-P(\al)=\Phi(x,\al)-\Phi(x(\al),\al).
\end{split}
\ee
To prove that $F_\e^{(l)}(x)$ is uniformly bounded for all $x\in\us_b$, it suffices to assume that $x$ is confined to a sufficiently small open subset $U^\prime\subset\us_b$.

Change variable $p=w+\Phi(\al)$ in \eqref{ker B}: 
\be\label{G step 1}\bs
F_\e^{(l)}(x)=-\frac1\pi \sum_{j_1\in J_\al}\int_{\br} & \frac{\chii(w)}w \sum_{j_2\in J_p} \bigg[\pa_w\int_{\I_\al} K(x,\al)\ik_\e(y-y_j)\dd \al\bigg]\\
&\times a_0(y_j)(p_{j_2}-P(\al_{j_1}))_+^{1/2}\dd w,\ y=(\al,w+\Phi(\al)).
\end{split}
\ee

Consider the integral with respect to $\al$:
\be\label{G step 2}\bs
\int_{\I_\al} & K(x,\al)\ik_\al\bigg(\frac{\al-\al_{j_1}}{\mu\e}\bigg)\pa_w\ik_p\bigg(\frac{w+\Phi(\al)-p_{j_2}}{\e}\bigg)\dd\al\\
=&\e^{-1}\int_{\I_\al} K(x,\al)\ik_\al\bigg(\frac{\al-\al_{j_1}}{\mu\e}\bigg)\ik_p^\prime\bigg(\frac{\Phi(\al)-\Phi(\al_{j_1})}\e+u\bigg)\dd\al,\\ 
u:=&(w+\Phi(\al_{j_1})-p_{j_2})/\e.
\end{split}
\ee
We have used that $\ik$ is the product of two kernels, see \eqref{interp-data}. Replace $\al_{j_1}$ with $\theta$, change variable $\al=\theta+\e\mu s$, and define
\be\label{h_eps}\bs
h(u;*):=&\mu\int_\br K(x,\theta+\e\mu s)\ik_\al(s)\ik_p\bigg(\frac{\Phi(\theta+\e\mu s)-\Phi(\theta)}\e+u\bigg)\dd s,\\ 
*=&(x,\theta,\e),\ u\in\br,\ x\in U^\prime,\ \theta\in I_\al,\ 0<\e\ll1.
\end{split}
\ee
The sum with respect to $j_2$ in \eqref{G step 1} simplifies by the introduction of the following function
\be\label{sum j}
H(\cht,\chq;*):=\sum_{j_2\in J_p}h^\prime(\cht+\chq-j_2;*)a_0(\theta,\e j_2)(j_2-\chq)_+^{1/2}.
\ee
The prime in $h^\prime$ denotes the derivative with respect to the first argument. To get the sum in \eqref{G step 1} we observe that $\e$ is factored out from $p_{j_2}-P(\al_{j_1})$ and use \eqref{sum j} with
\be
\cht=(w+Q(x,\al_{j_1}))/\e,\ \chq=P(\al_{j_1})/\e,\ \theta=\al_{j_1},\ j_1\in J_\al.
\ee
Some properties of the function $H$ are obtained in Lemma~\ref{lem:H props} in Appendix~\ref{sec:prf G props}.

Next we consider the integral with respect to $w$ in \eqref{G step 1}:
\be\label{sum jth v2}\bs
&\int_{\br}\frac {\chii(w)}w H\bigg(\frac{w+Q(x,\theta)}\e,\frac{P(\theta)}\e;*\bigg) \dd w\\
&=\int_{\br}\frac{\chii(\e\chw)}{\chw} H\bigg(\chw+\frac{Q(x,\theta)}\e,\frac{P(\theta)}\e;*\bigg) \dd \chw,
\end{split}
\ee
where $\theta=\al_{j_1}$. 
Define a key intermediate function
\be\label{key int}\bs
G(\cht,\chq;*):=&\int_{\br}\chii(\e\chw)\frac {H(\chw+\cht,\chq;*)}\chw \dd \chw.
\end{split}
\ee
In terms of $G$, \eqref{G step 1} becomes
\be\label{Fel def}\bs
F_\e^{(l)}(x)=-\frac{\e^{1/2}}{\pi}\sum_{j_1\in J_\al}G\bigg(\frac{Q(x,\al_{j_1})}\e,\frac{P(\al_{j_1})}\e;x,\al_{j_1},\e\bigg).
\end{split}
\ee
The extra factor $\e^{1/2}$ appears because it has been factored out from the term $(p_{j_2}-P(\al_{j_1}))_+^{1/2}$ in \eqref{G step 1} (cf. \eqref{sum j}).

The following lemma is proven in Appendix~\ref{sec:prf G props}.

\begin{lemma}\label{lem:G props} Suppose Assumptions~\ref{ass:Phi}--\ref{ass:interp ker} are satisfied. Pick any $A>0$. One has
\be\label{G props}\bs
|G(\cht,\chq;*)| \le & c\begin{cases} \e^{1/2}+(1+\cht)^{-1},&\cht>0,\\ 
(1-\cht)^{-1/2},&\cht<0,
\end{cases}
\end{split}
\ee
for some $c$ independent of $|\cht|\le A/\e$, $|\chq|\le A/\e$, $x\in U^\prime$, $\theta\in I_\al$, and $0<\e\ll1$.
\end{lemma}

We use this lemma in the proof of the following result, see Appendix~\ref{sec:prf Fel bnd}.

\begin{lemma}\label{lem:Fel bnd} Suppose Assumptions~\ref{ass:Phi}--\ref{ass:interp ker} are satisfied. Then $|F_\e^{(l)}(x)|\le c$ for all $x\in \us_b$ and $0<\e\ll1$.
\end{lemma}

\textcolor{black}{The number of segments $\Gamma_l$ described above Lemma~\ref{lem:conorm} is uniformly bounded, so} combining Lemma~\ref{lem:Fel bnd} with \eqref{Fe interm} proves Lemma~\ref{lem:F bnd final}.

\section{Computing the leading term $\Delta F_0$}\label{sec:fplt}

In this section we consider $x=x_0+\e\chx$ and assume for simplicity $x_0=0_2$. The reconstruction point is then $x=\e\chx$. Fix any $A_1>0$. Throughout this section we assume $|\chx|< A_1$, including in Lemmas~\ref{lem:del F0}, \ref{lem:del F0 v2}. 


\color{black}
From the proof of Lemma~\ref{lem:G props} and \eqref{h_eps}, \eqref{sum j}, \eqref{key int} it follows that 
\be
G(\cht,\chq;\e\chx,\theta,\e)=G(\cht,\chq;0_2,\theta,\e)+O(\e).
\ee
By \eqref{Fel def},
\be\label{Fel def v2}\bs
F_\e^{(l)}(\e\chx)=-\frac{\e^{1/2}}{\pi}\sum_{j_1\in J_\al}G\bigg(\frac{Q(x,\al_{j_1})}\e,\frac{P(\al_{j_1})}\e;0_2,\al_{j_1},\e\bigg)+O(\e^{1/2}).
\end{split}
\ee

\color{black}
Define 
\be\label{delF0 def a}
\Delta F_0(\chx):=\lim_{\e\to0}\big[F_\e(\e\chx)-F_\e(0_2)\big]. 
\ee
By the second line in \eqref{Fe interm},
\be\label{delF0 def}
\Delta F_0(\chx)=\lim_{\e\to0}\Delta F_\e^{(l)}(\e\chx),\ \Delta F_\e^{(l)}(\e\chx):=F_\e^{(l)}(\e\chx)-F_\e^{(l)}(0_2). 
\ee
\textcolor{black}{
From \eqref{Fel def v2} it follows that when computing $\Delta F_0(\chx)$, we can replace $\e\chx$ with $0_2$ in the list of arguments of $G$ when representing  $F_\e^{(l)}(\e\chx)$ in terms of $G$ in \eqref{delF0 def}. Therefore, we can drop $x$ from the list of arguments of $h$, $H$, and $G$ and change the meaning of $*=(x,\theta,\e)$ to  $*=(\theta,\e)$. Also, if the first argument of $Q$ is $x=0_2$, then it is omitted from notation.} 

Using \eqref{sum j}, define another intermediate function
\be\label{key int dg}\bs
\Delta G(r,\cht,\chq;*):=&G(r+\cht,\chq;*)-G(\cht,\chq;*)\\
=&\int_{\br}\chii(\e\chw)\frac {\Delta H(r,\chw+\cht,\chq;*)}\chw \dd \chw,\\
\Delta H(r,\cht,\chq;*):=&H(r+\cht,\chq;*)-H(\cht,\chq;*).
\end{split}
\ee
\textcolor{black}{We underscore that all the functions introduced in \eqref{delF0 def a}--\eqref{key int dg} are computed using a local patch of data $g=\chi_V\hat f$ as described in Lemma~\ref{lem:conorm}.}

By \eqref{Fel def}, the expression for $\Delta F_\e^{(l)}$ becomes:
\be\label{F via G}\bs
&\Delta F_\e^{(l)}(\e\chx)\\
&=-\frac{\e^{1/2}}\pi\sum_{j_1\in J_\al}\Delta G\bigg(\frac{Q(\e\chx,\al_{j_1})-Q(\al_{j_1})}\e,\frac{Q(\al_{j_1})}\e,\frac{P(\al_{j_1})}\e;\al_{j_1},\e\bigg).
\end{split}
\ee

The following result is proven in Appendix~\ref{sec:prf delG props}.
\begin{lemma}\label{lem:delG props} Suppose Assumptions~\ref{ass:Phi}--\ref{ass:interp ker} are satisfied. Pick any $A_2>0$. One has:
\textcolor{black}{
\be\label{delG prop 1}\bs
|\Delta G(r,\cht,\chq;*)| \le c\frac{|r|\ln(1/|r|)}{1+|\cht|},
\end{split}
\ee
\be\label{delG prop 2}\bs
|\Delta G(r+h,\cht,\chq;*)-\Delta G(r,\cht,\chq;*)| \le c\frac{|h|\ln(1/|h|)}{1+|\cht|},
\end{split}
\ee}
and
\be\label{delG prop 3}
|\Delta G(r,\cht+h,\chq;*)-\Delta G(r,\cht,\chq;*)|\le c\frac{|h|\ln(1/|h|)}{1+|\cht|}.
\ee
In the above formulas, the $c$’s are independent of $|r|\le A_2$, $|\cht|\le A_2/\e$, $|\chq|\le A_2/\e$, $\theta\in I_\al$, $0<\e\ll1$, and $0<h\ll1$.
\end{lemma}

\textcolor{black}{
We will apply Lemma~\ref{lem:delG props} to \eqref{F via G}, so we select $A_2$ as follows:
\be\label{for A2}
A_2=\max_{|\chx|\le A_1,\al_{j_1}\in I_\al,0<\e\ll1}\max\bigg(\frac{|Q(\e\chx,\al_{j_1})-Q(\al_{j_1})|}\e,|Q(\al_{j_1})|,|P(\al_{j_1})|\bigg).
\ee
Recall that $Q(\al)$ stands for $Q(x_0=0_2,\al)$. Since $\chx$ is confined to a bounded set and $Q$ is smooth, the right-hand side of \eqref{for A2} is bounded.}

Suppose first that $P(\al)$ is {\it not} associated with $x_0$. This means there is a neighborhood $U^\prime$ of $x_0$ such that no curve $\s_{(\al,P(\al))}$, $\al\in I_\al$, is tangent to $\s\cap U^\prime$. There are two possibilities: (a) Neither of these curves passes through $x_0$ (in this case $Q(\al)$ is bounded away from zero on $I_\al$) and (b) one of the curves contains $x_0$ (in this case $Q(\tilde\al)=0$ for some $\tilde\al\in I_\al$).

(a) If $|Q(\al)|>c$, $\al\in I_\al$, the inequality \eqref{delG prop 1} implies 
\be\label{wrong P a}
\e^{1/2}\tsum_{j_1\in J_\al}|\Delta G|\le c \e^{1/2}\tsum_{j_1=1}^{1/\e}\e=O(\e^{1/2}).
\ee
Here and below, the arguments of $\Delta G$ are the same as in \eqref{F via G}.

(b) If $Q(\tilde \al)=0$, but $|Q^\prime(\al)|>c$ on $I_\al$, then $|Q(\al)|\asymp |\al-\tilde\al|$, $\al\in I_\al$, and \eqref{delG prop 1} leads to
\be\label{wrong P b}
\textcolor{black}{\e^{1/2}\tsum_{j_1\in J_\al}|\Delta G|\le c \e^{1/2}\tsum_{j_1=1}^{1/\e} \frac{\e}{\e j_1}=O(\e^{1/2}\ln(1/\e)).}
\ee

Suppose now $P(\al)$ is associated with $x_0$. Therefore, $I_\al$ is a small neighborhood of $\al_0$. Recall that we assumed $\al_0=0$. By the definition of the function $Q$ (see \eqref{F step 2}), $Q(0)=0$. Applying Lemma~\ref{lem:Q 2nd der} with $\tilde\al=\al_0=0$ and $\tilde x=x_0=0_2$ gives $Q_\al^\prime(0)=0$. 

Pick any $A\gg 1$. By Lemma~\ref{lem:Q 2nd der}, $Q_{\al\al}^{\prime\prime}(0)>0$. Hence, by continuity,
\be\label{Q bnd}
Q(\al)\ge c \al^2,\quad |\al|\le \de,
\ee
for some $0<\de\ll1$. Also, $Q(\al)\ge c$, $|\al|\ge\de$, $\al\in I_\al$.
By \eqref{delG prop 1}, this implies
\be\label{F via G part1}\bs
&\e^{1/2}\bigg(\sum_{A\e^{1/2}\le|\al_{j_1}|\le\de}+\sum_{j_1\in J_\al,|\al_{j_1}|>\de}\bigg)|\Delta G|\\
&\le c \e^{1/2}\bigg(\sum_{j_1>A\e^{-1/2}}\frac{\e}{(\e j_1)^2}+\sum_{j_1=\de/\e}^{1/\e}\e\bigg)=O(A^{-1})+O(\e^{1/2}),
\end{split}
\ee

By \eqref{F step 2}, \eqref{delG prop 2} and \eqref{delG prop 3}, for $|\al_{j_1}|\le A\e^{1/2}$:
\be\label{del G}\bs
&\Delta G\bigg(\frac{Q(\e\chx,\al_{j_1})-Q(\al_{j_1})}\e,\frac{Q(\al_{j_1})}\e,\frac{P(\al_{j_1})}\e;\al_{j_1},\e\bigg)\\
&=\Delta G\bigg(\dd_x\Phi\cdot\chx+O(\e^{1/2}),\frac{Q_{\al\al}^{\prime\prime}}2\e (\mu j_1)^2+ O(\e^{1/2}),\frac{P(\e\mu j_1)}\e;\al_{j_1},\e\bigg)\\
&=\Delta G\bigg(\dd_x\Phi\cdot\chx,\frac{Q_{\al\al}^{\prime\prime}}2 \e(\mu j_1)^2,\frac{P(\e\mu j_1)}\e;\al_{j_1},\e\bigg)+O(\e^{1/2}\ln(1/\e)).
\end{split}
\ee

Combine \eqref{wrong P a}, \eqref{wrong P b}, \eqref{F via G part1}, \eqref{del G}, and use \eqref{F via G} to obtain
\be\label{F via G part2}\bs
\Delta F_\e^{(l)}(\e\chx)= & -\frac{\e^{1/2}}\pi\sum_{|\al_{j_1}|\le A\e^{1/2}}\Delta G\bigg(\dd_x\Phi\cdot\chx,\frac{Q_{\al\al}^{\prime\prime}}2 \e (\mu j_1)^2,\frac{P(\e\mu j_1)}\e;\al_{j_1},\e\bigg)\\
&+O(\e^{1/2}\ln(1/\e))+O(1/A).
\end{split}
\ee
Here we have used that there are $O(\e^{1/2})$ values of $j_1$ such that $|\al_{j_1}|\le A\e^{1/2}$.

\textcolor{black}{Our analysis shows (see \eqref{wrong P a}, \eqref{wrong P b}) that if $\Gamma_l$ is not associated with $x_0$, then the data in $V_l=I_\al^l\times I_p^l$ does not contribute to the DTB at $x_0$.}

The following result is proven in Appendix~\ref{sec:prf del F0}.
\begin{lemma}\label{lem:del F0} Under the assumptions of Theorem~\ref{thm:main res} one has
\be\label{del F0}\bs
\Delta F_0(\chx)=-\frac1{\pi}\int_{|s|\le A}\int_0^1 &\big[G_0\big(\dd_x\Phi\cdot\chx+(Q_{\al\al}^{\prime\prime}/2)s^2,\chq\big)\\
&-G_0\big((Q_{\al\al}^{\prime\prime}/2)s^2,\chq\big)\big]\dd\chq\dd s+O(1/A),
\end{split}
\ee
where
\be\label{G0H0}\bs
&G_0(\cht,\chq):=\int_{\br}\frac {H_0(\chw+\cht,\chq)}\chw \dd \chw,\\
&H_0(\cht,\chq):=a_0(y_0)\sum_{j_2\in\mathbb Z}h_0^\prime(\cht+\chq-j_2)(j_2-\chq)_+^{1/2},\\
&h_0(u):= K(x_0,\al_0)\int_\br \ik_\al(s)\ik_p(\mu s\Phi_{\al}^\prime(0)+u)\dd s.
\end{split}
\ee
\end{lemma}

See \eqref{K_1 K} for the definition of $K(x,\al)$. Letting $A\to\infty$ in \eqref{del F0} gives
\be\label{last sum lim2}\bs
&\Delta F_0(\chx)\\
&=-\frac1{\pi}\int_{\br}\int_0^1 \big[G_0\big(\dd_x\Phi\cdot\chx+(Q_{\al\al}^{\prime\prime}/2)s^2,\chq\big)
-G_0\big((Q_{\al\al}^{\prime\prime}/2)s^2,\chq\big)\big]\dd\chq\dd s\\
&=-\frac{a_0(y_0)}{\pi}\int_{\br}\int_{\br}\ioi\frac 1\chw \big[h_0^\prime\big(\chw+\dd_x\Phi\cdot\chx+(Q_{\al\al}^{\prime\prime}/2) s^2-\chq\big)\\
&\hspace{3.5cm}-h_0^\prime\big(\chw+(Q_{\al\al}^{\prime\prime}/2) s^2-\chq\big)\big]\chq^{1/2}\dd \chq\dd s \dd \chw.
\end{split}
\ee
\color{black}
Comparing \eqref{sum j} with \eqref{G0H0}, we see that the functions $H$ and $H_0$ are analogous and they have the same asymptotic behavior as $\cht\to\infty$. More precisely, $H_0$ satisfies Lemma~\ref{lem:H props} with $\e=0$. To apply this lemma to obtain the asymptotics of $H_0$, we set $\e=0$ and drop all the restrictions of the kind $\cht<A/\e$. The same modifications should be made in the proof of Lemma~\ref{lem:H props} in Appendix~\ref{sec:prf H props}.

Likewise, the functions $G$ of \eqref{key int} and $G_0$ of  \eqref{G0H0} are analogous and they also have the same asymptotic behavior as $\cht\to\infty$. Repeating the argument in the proof of \eqref{delG prop 1} (see Appendix~\ref{sec:prf delG props} from the beginning through the paragraph following \eqref{J2pr1}) and replacing $H$ with $H_0$, we see that $G_0(\cht+r,\chq)-G_0(\cht,\chq)$ satisfies the same estimates as $\Delta G$ in \eqref{delG prop 1}, so the first integral in \eqref{last sum lim2} is absolutely convergent. 

\color{black}
The following result is proven in Appendix~\ref{sec:prf lem del F0 v2}.
\begin{lemma}\label{lem:del F0 v2} Under the assumptions of Theorem~\ref{thm:main res} one has
\be\label{last sum lim3}
\Delta F_0(\chx)=C\int_0^{\dd_x\Phi\cdot\chx} h_0(-t)\dd t,\ C:=\frac{a_0(y_0)}{4(Q_{\al\al}^{\prime\prime}/2)^{1/2}}.
\ee
\end{lemma}

\section{Computing the DTB function. End of proof of Theorem~\ref{thm:main res}}\label{sec:DTB fn}
Pick any $A_0>0$. Consistent with Theorem~\ref{thm:main res}, throughout this section we assume $|\chx|< A_0$. \textcolor{black}{As established in the preceding section, it suffices to consider only the segment $\Gamma_l$ (and the corresponding $a_0(y),P(\al)$) that is associated with $x_0$.}

\subsection{Reduction of the DTB function to an integral}\label{ssec:red to int}
Now we study $\G_\e$ introduced in \eqref{lead sing fe}. As usual, suppose $x_0=0_2$. Change variables $\hxi=\e\xi$ and $\chx=x/\e$, $\chz=z/\e$ in \eqref{lead sing fe} and express $\G_\e$ as follows:
\be\bs\label{to scl PDO 2}
&\G_\e(\chx):=\frac{1}{(2\pi)^2}\int_{\us/\e}\int_{\br^2} \big(\tilde D_0(\e \chx,\hxi)+\kappa|\hxi|^3\big)^{-1}e^{-i\hxi\cdot(\chx-\chz)}\dd\hxi\, F_\e(\e\chz) \dd \chz.
\end{split}
\ee
Since 
\be\label{at infty}
\tilde D_0(\e \chx,\hxi)+\kappa|\hxi|^3\asymp |\hxi|^3,\ |\hxi|\to\infty,
\ee
the integral with respect to $\hxi$ in \eqref{to scl PDO 2} is absolutely convergent. 

Our aim is to reconstruct the jump of $f$ at $x_0=0_2$, so we compute $\Delta \G_\e(\chx):=\G_\e (\chx)-\G_\e(0_2)$. By \eqref{to scl PDO 2}, $\Delta \G_\e(\chx)=\Delta \G_\e^\os(\chx)+\Delta \G_\e^\tp(\chx)$, where
\be\bs\label{del bfe 1}
\Delta \G_\e^\os(\chx)=\frac1{(2\pi)^2}\int_{\us/\e}\int_{\br^2} & \big[(\tilde D_0(\e \chx,\hxi)+\kappa|\hxi|^3)^{-1}-(\tilde D_0(0_2,\hxi)+\kappa|\hxi|^3)^{-1}\big]\\
&\times e^{-i\hxi\cdot(\chx-\chz)}\dd\hxi\, F_\e(\e\chz) \dd \chz,
\end{split}
\ee
and
\be\bs\label{del bfe 2}
&\Delta \G_\e^\tp(\chx)=\frac1{(2\pi)^2}\int_{\us/\e}\int_{\br^2}  \frac{e^{-i\hxi\cdot\chx}-1}{\tilde D_0(0_2,\hxi)+\kappa|\hxi|^3}e^{i\hxi\cdot\chz}\dd\hxi\, F_\e(\e\chz) \dd \chz.
\end{split}
\ee

The following result is proven in Appendix~\ref{ssec:prf int decay 2d}.

\begin{lemma}\label{lem:int decay 2d}
Let $I\subset \I_\al$ be an interval. Let $D(\la,\theta,\e)\in C^\infty\big((0,\infty)\times I \times (0,\e_0)\big)$ be a function which satisfies
\be\label{D ineqs}
\big|\pa_\theta^l\pa_\la^k D(\la,\theta,\e)\big|\le c_{k,l}\e^n(1+\la)^{-3},\ (\la,\theta,\e)\in (0,\infty)\times I\times (0,\e_0),
\ee
for any $k=0,1,2$, $l\in\N_0$, and some $c_{k,l}$ independent of $\e$, $\e_0>0$, and $n\in\N_0$. Define 
\be\label{lim1}
J(r):=\int_0^\infty \int_I D(\la,\theta,\e)e^{i r\la\sin\theta}\dd\theta\la \dd\la.
\ee
Then
\be\label{lim3}
J(r)=\e^n\begin{cases} O(r^{-2})& \\ O(r^{-3}),& \text{if }D(\la=0^+,\theta,\e)\equiv 0, \end{cases}\ r\to\infty,
\ee
uniformly in $\e\in(0,\e_0)$. 
\end{lemma}

The expression in brackets in \eqref{del bfe 1} satisfies \eqref{D ineqs} with $n=1$ uniformly in $\chx$ within bounded sets (see \eqref{homog 1}). Here $\hxi=\la(\cos\theta,\sin\theta)$, $\la=|\hxi|$. By the top case in \eqref{lim3} with $n=1$, the integral with respect to $\hxi$ in \eqref{del bfe 1} is $\e O(|\chz|^{-2})$, $|\chz|\to\infty$. Recall that $\chx$ is confined to a bounded set. Additionally, the support of $F_\e(\e\chz)$ is of size $O(1/\e)$. Hence 
\be\label{del G 1st}
|\Delta \G_\e^\os(\chx)|\le c\e\int_{|\chz|\le 1/\e}(1+|\chz|)^{-2}\dd\chz=O(\e\ln(1/\e)). 
\ee

The fraction in \eqref{del bfe 2} satisfies \eqref{D ineqs} with $n=0$ uniformly in $\chx$ within bounded sets. Moreover, its value at $\la=0^+$ equals zero. By the bottom case in \eqref{lim3} with $n=0$, the integral with respect to $\hxi$ is $O(|\chz|^{-3})$, $|\chz|\to\infty$. This means that the integral with respect to $\chz$ in \eqref{del bfe 2} is absolutely convergent and uniformly bounded for $0<\e\ll1$. In both integrals we have used Lemma~\ref{lem:F bnd final}. 

In Section~\ref{sec:fplt} we computed the limit
\be\label{F0 def}
\Delta F_0(\chx)=\lim_{\e\to0}\Delta F_\e(\e\chx),\ |\chx|\le A_1,  
\ee
for any $A_1>0$ (see \eqref{delF0 def a} and \eqref{last sum lim3}). Pick any $A_1\gg 1$ and use \eqref{del G 1st} to obtain from \eqref{to scl PDO 2}, \eqref{del bfe 1}, \eqref{del bfe 2}:
\be\bs\label{to scl PDO 3}
\Delta \G_\e (\chx)=&\frac{1}{(2\pi)^2}\int_{|\chz|\le A_1}\bigg[\int_{\br^2} \frac{e^{-i\hxi\cdot\chx}-1}{\tilde D_0(0_2,\hxi)+\kappa|\hxi|^3}e^{i\hxi\cdot\chz}\dd\hxi\bigg] \Delta F_\e(\e\chz)\dd \chz\\
&+\frac{F_\e(0)}{(2\pi)^2}\int_{\br^2} \frac{e^{-i\hxi\cdot\chx}-1}{\tilde D_0(0_2,\hxi)+\kappa|\hxi|^3}\tilde\vartheta_{A_1}(\hxi)\dd\hxi\\
&+O(1/A_1)+O(\e\ln(1/\e)),
\end{split}
\ee
where $\tilde\vartheta_{A_1}(\hxi)$ is the Fourier transform of the characteristic function of the ball $\{|\chx|\le A_1\}$. By Lemma~\ref{lem:int decay 2d} with $n=0$, the integral in brackets above is $O(|\chz|^{-3})$, $|\chz|\to\infty$. Use \eqref{F0 def}, and the Lebesgue dominated convergence theorem to take the limit as $\e\to0$ in the first integral in \eqref{to scl PDO 3} to obtain:
\be\bs\label{rhs lim}
\frac{1}{(2\pi)^2}\int_{|\chz|\le A_1}\int_{\br^2} \frac{e^{-i\hxi\cdot\chx}-1}{\tilde D_0(0_2,\hxi)+\kappa|\hxi|^3} e^{i\hxi\cdot\chz}\dd\hxi\,\Delta F_0(\chz)\dd \chz.
\end{split}
\ee
Since $A_1\gg 1$ can be arbitrarily large and $\tilde\vartheta_{A_1}(\hxi)\to\de(\hxi)$ as $A_1\to\infty$, \eqref{to scl PDO 3}, \eqref{rhs lim}, and Lemma~\ref{lem:F bnd final} imply
\be\bs\label{rhs lim 2}
\lim_{\e\to0}\Delta \G_\e (\chx)=&\frac{1}{(2\pi)^2}\int_{\br^2}\int_{\br^2} \frac{e^{-i\hxi\cdot\chx}-1}{\tilde D_0(0_2,\hxi)+\kappa|\hxi|^3} e^{i\hxi\cdot\chz}\dd\hxi\,\Delta F_0(\chz)\dd \chz.
\end{split}
\ee

\subsection{Evaluation of the integral in \eqref{rhs lim 2}}\label{ssec:eval of int}
The last integral simplifies using that $\Delta F_0(\chx)=\varphi(\dd_x\Phi\cdot\chx)$ for some $\varphi$ (see \eqref{last sum lim3}). Simple transformations give
\be\bs\label{rhs lim 3}
\Delta \G_0 (\chx):=&\lim_{\e\to0}\Delta \G_\e (\chx)\\
=&\frac{1}{2\pi}\int_{\br}\bigg[ \int_{\br} \frac{e^{-i\la\vec\Theta\cdot\chx}-1}{\tilde D_0(0_2,\vec\Theta_0)+\kappa|\la|^3}e^{i\la t}\dd\la\bigg] \varphi(|\dd_x \Phi|t)\dd t,\\
\varphi(t):=&C\int_0^t h_0(-s)\dd s,\ \vec\Theta_0:=\dd_x \Phi/|\dd_x \Phi|,
\end{split}
\ee
where $C$ is defined in \eqref{last sum lim3}. Denote
\be\label{AAtilde}
\tilde R(\la):=\big[\tilde D_0(0_2,\vec\Theta_0)+\kappa|\la|^3\big]^{-1},\ 
R:=\CF^{-1}\tilde R.
\ee
Since $\tilde R(\la)$ is $C^2$ near $\la=0$ (and smooth otherwise), $R(t)$ is absolutely integrable. Combine \eqref{last sum lim3}, \eqref{rhs lim 3}, and \eqref{AAtilde} to obtain
\be\bs\label{rhs lim 4}
&\Delta \G_0 (\chx)=C\int_{\br} (R(r-t)-R(-t)) \int_{-\infty}^{|\dd_x\Phi|t} h_0(-s)\dd s\dd t,\ r:=\vec\Theta_0\cdot \chx.
\end{split}
\ee
We have extended the integral with respect to $s$ to $(-\infty,|\dd_x\Phi|t)$ because $\int (R(r-t)-R(-t))\dd t\equiv0$ for all $r$. Transforming \eqref{rhs lim 4} further gives
\be\bs\label{rhs lim 4 v2}
&\Delta \G_0 (\chx)=(C/|\dd_x \Phi|)\int_{\br} h_0(s)\int_{s}^{s+\dd_x\Phi\cdot\chx}R(t/|\dd_x \Phi|)\dd t\dd s.
\end{split}
\ee

For convenience, we collect here all the constants that will be used in subsequent calculations.
\begin{alignat}{2}\label{tilde D v2}
&\tilde D(0_2,\dd_x\Phi)=2\pi\frac{W^2|\dd_x\Phi|}{|\Delta_\Phi|} \quad&&\text{(by \eqref{pxi  v2}, \eqref{R*R D0})}\\
&K(x_0,\al_0)=|\dd_x\Phi|W \quad&&\text{(by \eqref{K_1 K})}\\
&Q_{\al\al}^{\prime\prime}(x_0,\al_0)=\frac{|\Delta_\Phi||x_\al^\prime|}{|\dd_x\Phi|}\quad&&\text{(by \eqref{Q 2nd der}, \eqref{Del Phi res})}\\
&\psi(y_0)=\frac{|\dd_x H|}{|\dd_x\Phi|}\quad&&\text{(by \eqref{psi val})}\\
\label{big C}
&C=2\pi\Delta f\frac{W}{|\Delta_{\Phi}|}\quad&&\text{(by \eqref{last sum lim3}, \eqref{g-lead-sing}, \eqref{Rf-lim-coefs}, \eqref{Del Phi res})}.
\end{alignat}
For simplicity, we dropped all the arguments on the right.

Let $\bar h_0$ be defined the same way as $h_0$ in \eqref{G0H0}, but without the factor $K(x_0,\al_0)$. By \eqref{AAtilde} and \eqref{rhs lim 4 v2}--\eqref{big C}:
\be\bs\label{rhs lim 5}
&\Delta \G_0 (\chx)=\Delta f \Upsilon\big(\dd_x\Phi\cdot\chx\big),\\
&\Upsilon(r)=2\pi\frac{W^2}{|\Delta_\Phi|}\int_{\br} \bar h_0(u)\int_{u}^{u+r}\CF^{-1}\bigg(\frac1{ 2\pi\frac{W^2|\dd_x \Phi|}{|\Delta_\Phi|}+\kappa|\la|^3}\bigg)(t/|\dd_x \Phi|)\dd t\dd u.
\end{split}
\ee 
A simple transformation yields:
\be\bs\label{upsilon simple}
\Upsilon(r)=&\frac1{|\dd_x \Phi|}\int_{\br} \bar h_0(u)\int_{u}^{u+r}\CF^{-1}\bigg(\frac1{1+\kappa C_1|\la|^3}\bigg)(t/|\dd_x \Phi|)\dd t\dd u,\\
C_1=&\frac{|\Delta_\Phi|}{2\pi W^2|\dd_x \Phi|}.
\end{split}
\ee 
Further simple transformations prove \eqref{DTB eq}, \eqref{DTB aux}. The rest of the claims follow from the next lemma, which is proven in Appendix~\ref{sec:prf Ups props}.

\begin{lemma}\label{lem:Ups props} One has $\Upsilon(0)=0$ and $\Upsilon(\pm\infty)=\pm1/2$.
\end{lemma} 

Truncated-view artifacts are reasonably low. This will likely allow for successful image segmentation and motion estimation.

\section{Redundant data. Proof of Theorem~\ref{thm:main res mult}}\label{sec:red dat}

Suppose the GRT data are redundant and there are multiple points $y_l$ from which the singularity at $(x_0,\xi_0)$ is visible. This means that for all $(x,\xi)$ in a conic neighborhood of $(x_0,\xi_0)$ there are multiple local solutions $\nu_l(x,\xi)$ and $\al_l(x,\xi)$ to the equation $\xi=-\nu\dd_x\Phi(x,\al)$. The local uniqueness and smoothness of the solutions follows from Assumption~\ref{geom GRT}\eqref{li}. By Assumption~\ref{ass:f props}\eqref{del curv}, their number is uniformly bounded. Set $p_l(x,\xi)=\Phi(x,\al_l(x,\xi))$ and $y_l=(\al_l,p_l)$. 
\color{black}
The principal symbol of $\R^*\R$ is given by (cf. \eqref{R*R pr symb})
\be\label{R*R pr symb mult}\bs
\tsum_l\frac{2\pi}{|\nu_l|} \frac{W^2(x,y_l)}{|\Delta_\Phi(x,\al)|},\ y_l=y(x,\xi),\nu_l=\nu(x,\xi).
\end{split}
\ee
Similarly to \eqref{R*R D0}, define
\be\label{R*R D0 mult}\bs
\tilde D_0(x,\xi):=& 2\pi\tsum_l\frac{|\xi|}{|\nu_l|} \frac{W^2(x,y_l)}{|\Delta_\Phi(x,\al_l)|},\ y_l=y_l(x,\xi),\nu_l=\nu_l(x,\xi),\\
(x,\xi)\in &\, \us\times(\br^2\setminus 0_2).
\end{split}
\ee
\color{black}
The analog of \eqref{lead sing fe} becomes
\be\bs\label{tilde B mult}
&\G_\e=\op\big(\big[\textstyle\sum_l \tilde D_l(x,\xi/|\xi|)+\kappa|\e\xi|^3\big]^{-1}\big)F_\e,\ x\in\us_0,\\
&\tilde D_l(x_0,\vec\Theta_0):= 2\pi\frac{W^2(x_0,y_l)|\dd_x \Phi(x_0,\al_l)|}{|\Delta_\Phi(x_0,\al_l)|}.
\end{split}
\ee
Here we have used that 
\be\label{theta_0}
\vec\Theta_0=\xi_0/|\xi_0|=\dd_x H(x_0)/|\dd_x H(x_0)|=\dd_x \Phi(x_0,\al_l)/|\dd_x \Phi(x_0,\al_l)|
\ee
for all $l$, cf. the convention about the directions of $\dd_x H$ and $\dd_x\Phi$ stated above Assumption~\ref{ass:f props}. Therefore, we get from \eqref{AAtilde} that the analog of the function $R$ is
\be\label{AAtilde mult}
\tilde R(\la):=\big[\tsum_l \tilde D_l(x_0,\vec\Theta_0)+\kappa|\la|^3\big]^{-1},\ 
R:=\CF^{-1}\tilde R.
\ee
Using \eqref{g-lead-sing} and \eqref{ker B}, we see that $\Delta \G_0 (\chx)$ equals to the sum of the terms analogous to \eqref{rhs lim 5}:
\be\bs\label{rhs lim 5 mult}
&\Delta \G_0 (\chx)=\Delta f \textstyle\sum_l\Upsilon_l\big(\dd_x \Phi(x_0,\al_l)\cdot\chx\big),\\
&\Upsilon_l(r)=2\pi\frac{W^2(x_0,y_l)}{|\Delta_\Phi(x_0,\al_l)|}\int_{\br} \bar h_l(u)\\
&\hspace{1cm}\times\int_{u}^{u+r}\CF^{-1}\left(\big[\textstyle\sum_l \tilde D_l(x_0,\vec\Theta_0)+\kappa|\la|^3\big]^{-1}\right)\bigg(\frac{t}{|\dd_x \Phi(x_0,\al_l)|}\bigg)\dd t\dd u,\\ 
&\bar h_l(u):=\int_\br \ik_\al(s)\ik_p(\mu s\Phi_{\al}^\prime(x_0,\al_l)+u)\dd s.
\end{split}
\ee 
Simplifying we get similarly to \eqref{rhs lim 5}, \eqref{upsilon simple}
\be\bs\label{upsilon simple mult}
&\Delta \G_0 (\chx)=\Delta f  \frac{\tsum_l\nu_l\Upsilon_l\big(\dd_x \Phi(x_0,\al_l)\cdot\chx/|\dd_x \Phi(x_0,\al_l)|\big)}{\sum_l\nu_l},\\
&\Upsilon_l(r)=\int_{\br} \bar h_0(u)\int_{u_l}^{u_l+r}\CF^{-1}\bigg(\frac1{1+\kappa C_1|\la|^3}\bigg)(t)\dd t\dd u,\\
&\nu_l:=\frac{ W^2(x_0,y_l)|\dd_x \Phi(x_0,\al_l)|}{|\Delta_\Phi(x_0,\al_l)|},\ u_l:=\frac u{|\dd_x \Phi(x_0,\al_l)|},\ C_1=\frac{1}{2\pi\tsum_l\nu_l}.
\end{split}
\ee 
Using \eqref{theta_0} finishes the proof.

\section{Numerical experiment}\label{sec:numerics}

We choose $f$ to be the characteristic function of the disk centered at $x_c=(1,1)$ with radius $r=2$. Thus, $\s=\{x\in\br^2:|x-x_c|=r\}$. The center of a local region of interest (ROI), i.e. the point on the boundary $x_0\in\s$, is $x_0=x_c+r\vec\bt_0$, where $\bt_0=-0.17\pi$. The GRT integrates over circles $\s_y$, $y=(\al,\rho)$, with various radii $\rho>0$ and centers $R\vec\al$, $\al\in[0,2\pi)$, where $R=10$. The integration weight is $W\equiv 1$. Thus, $\rho=\Phi(x,\al)=|x-R\vec\al|$ and
\be\label{needed quantities}\bs
&\dd_x\Phi(x,\al)=\frac{x-R\vec\al}{|x-R\vec\al|}=:\vec\Theta,\
\Phi_\al^\prime(x,\al)=R\vec\Theta\cdot \vec\al^\perp,\\ 
&\dd_x\Phi_\al^\prime(x,\al)=-\frac R{|x-R\vec\al|}\vec\al^\perp+c\vec\Theta,\\ 
&\Delta_\Phi(x,\al)=\det \bma \vec\Theta \\ -\frac R{|x-R\vec\al|}\vec\al^\perp+c\vec\Theta \ema
=\frac {R\vec\al\cdot\vec\Theta}{|x-R\vec\al|}\not=0,\ |x|<R.
\end{split}
\ee
Here $\vec\al^\perp=(-\sin\al,\cos\al)$, and $\vec\Theta^\perp$ is defined similarly. The continuous data corresponds to $\al\in\I_\al:=[0,2\pi)\sim S^1$ and $\rho\in\I_p:=(0,2R)$. The image domain is $\us=\{|x|<R\}$. The support of the object is contained in the square reconstruction region $\us_b:=\{(x_1,x_2)\in\br^2:|x_1|< R_{rec},\ |x_2|< R_{rec}\}$, $R_{rec}=3.7$. This corresponds to the set $\us_b$ used in \eqref{min pb}. 

\color{black}
Further (cf. \eqref{DTB aux mult}),
\be\bs\label{nu12}
&\nu_l=|\Delta_\Phi(x_0,\al_l)|^{-1},\ u_l=u,\ l=1,2.
\end{split}
\ee 
This implies that the DTB function given by the fraction in \eqref{DTB eq mult} becomes
\be\label{DTB circ RT}
\frac{\tsum_l\nu_l\Upsilon_l(r)}{\tsum_l\nu_l}
=\int_{\br} \frac{\tsum_l\nu_l h_l(u)}{\tsum_l\nu_l} \int_u^{u+r}R(t)\dd t\dd u.
\ee

Let us now make sure that the circular GRT used in this section satisfies  Assumptions~\ref{ass:Phi} and \ref{geom GRT}. Clearly, Assumption~\ref{ass:Phi} is satisfied. We next check Assumption~\ref{geom GRT}. By the first two lines in \eqref{needed quantities}, Assumption~\ref{geom GRT}\eqref{li} is satisfied since $\vec\al\cdot\vec\Theta\not=0$ when $|x|<R$. Next, suppose $|x_1-R\vec\al|=|x_2-R\vec\al|$ for some $\vec\al$ and $x_1\not= x_2$, $|x_1|,|x_2|<R$. The first line in \eqref{needed quantities} implies $\Phi_\al^\prime(x_1,\al)\not=\Phi_\al^\prime(x_2,\al)$, therefore Assumption~\ref{geom GRT}\eqref{bolk} is satisfied. 

For any $(x_0,\vec\Theta_0)$, $|x_0|<R$, there are two circles $\s_y$ in the data, which contain $x_0$ and are orthogonal to $\vec\Theta_0$ at $x_0$. Thus, the conormal bundles of all the $\s_y$ cover $T^*\us_b$ and Assumption~\ref{geom GRT}\eqref{visib} is satisfied. Since $W\equiv 1$, Assumption~\ref{geom GRT}\eqref{pos W} trivially holds. 

By \cite[Corollary 5]{AmbK05}, $\R$ is injective. As stated at the end of Section~\ref{ssec:grt}, together with the visibility of every singularity and the fact that $\R$ and $\R^*$ are FIOs of order $-1/2$, this ensures that Assumption~\ref{geom GRT}\eqref{invert} is satisfied.
\color{black}

The discrete data are given at the points
\be\label{d data}\bs
&\al_{j_1}=\Delta_\al j_1,\ \Delta\al=2\pi/N_\al,\ 0\le j_1<N_\al,\ N_\al=300,\\
&\rho_{j_2}=\rho_{\min}+j_2 \Delta\rho,\ \e:=\Delta_\rho=(\rho_{\max}-\rho_{\min})/(N_\rho-1),\ 
0\le j_2<N_\rho,\\
&\rho_{\min}=R- R_{rec}\sqrt2,\ \rho_{\max}=R+R_{rec}\sqrt2,\ N_\rho=451.
\end{split}
\ee

As mentioned above, for each $x\in\s$ there are two points, $y_l=(\al_l,\rho_l)$, $l=1,2$, such that $\s_{y_l}$ is tangent to $\s$ at $x$. To find $y_{1,2}$ for the selected $x_0$ we first solve $\Vert x_c+t\vec\bt_0\Vert=R$. This gives two values $t_1,t_2$, $t_2<0<t_1$. Then $\al_l$ are determined from $x_c+t_l\vec\bt_0=R\vec\al_l$, and $\rho_l=|x_0-R\vec\al_l|$, $l=1,2$. In view of \eqref{needed quantities}, in numerical computations we use $|x_0-R\vec\al_1|=t_1-r$ and $|x_0-R\vec\al_2|=-t_2+r$. \textcolor{black}{By a simple geometric argument (cf. \eqref{needed quantities}), 
\be
-R\vec\al_l\cdot\vec\Theta_l=(t_1-t_2)/2,\ l=1,2.
\ee}

The regularization parameter is set at $\kappa=0.5$. In order to approximate the continuous GRT $\R$ and its adjoint $\R^*$, the reconstruction is performed on a dense $801\times801$ grid covering the same square region $\us_b$, and the interpolated data $g_\e$ are assumed to be given on a dense grid of the kind \eqref{d data} with $N_\al^\prime=800$, $N_\rho^\prime=1201$. Prior to the reconstruction, the GRT data $g(y_j)$ are interpolated from the grid \eqref{d data} to the more dense grid using \eqref{interp-data}. For this we use the Keys interpolation kernel \cite{Keys1981, btu2003}
\be\label{keys}
\ik_\al(t)=\ik_\rho(t)=\ik(t)=3B_3(t+2) - (B_2(t+2) + B_2(t + 1)),
\ee
where $B_n$ is the cardinal $B$-spline of degree $n$ supported on $[0, n+1]$. Therefore $\tsp\ik=[-2,2]$. The kernel is a piecewise-cubic polynomial with continuous $\ik,\ik^\prime$ and bounded $\ik^{\prime\prime}$, so $\ik\in C_0^2(\br)$. 

The minimization of the functional \eqref{min pb} is performed using gradient descent. 
\textcolor{black}{The stopping criterion is met when the $L^\infty$-norm of the image update is less than $10^{-6}$ for three consecutive iterations. The algorithm converged in 214 iterations.}
The reconstruction results are shown in Figures~\ref{fig:global_exp=2} and \ref{fig:local_exp=2}. The reconstruction of the entire region $\us_b$ is shown in Figure~\ref{fig:global_exp=2}. A profile of the reconstruction through the center of the ball, $x_c$, is shown on the right. The location of the profile is shown in the left panel.

A small section of the reconstruction, which is located inside the square ROI shown in Figure~\ref{fig:global_exp=2}, is extracted from the global reconstruction and shown in Figure~\ref{fig:local_exp=2}. The right panel of Figure~\ref{fig:local_exp=2} shows the reconstructed (green) and predicted (blue) profiles. The latter is computed by using \eqref{needed quantities} in \eqref{DTB eq mult}, \eqref{DTB aux mult}.
The plot in Figure~\ref{fig:local_exp=2} corresponds to the right jump in Figure~\ref{fig:global_exp=2}. Overall, the match between the reconstruction and prediction is very accurate. 

\begin{figure}[h]
{\centerline{
{\hbox{
{\epsfig{file={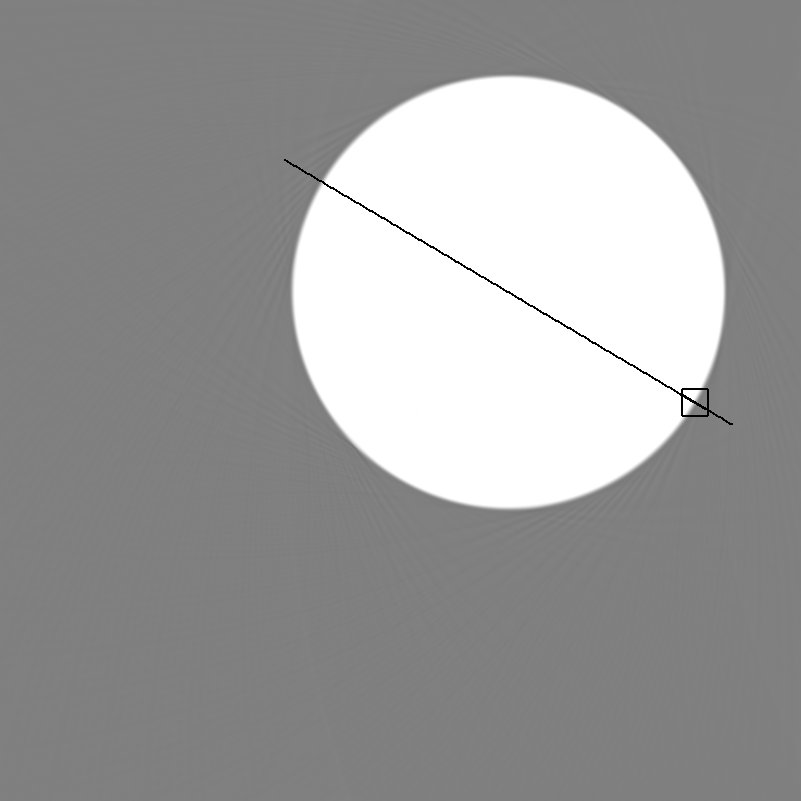}, height=5.4cm}}
{\epsfig{file={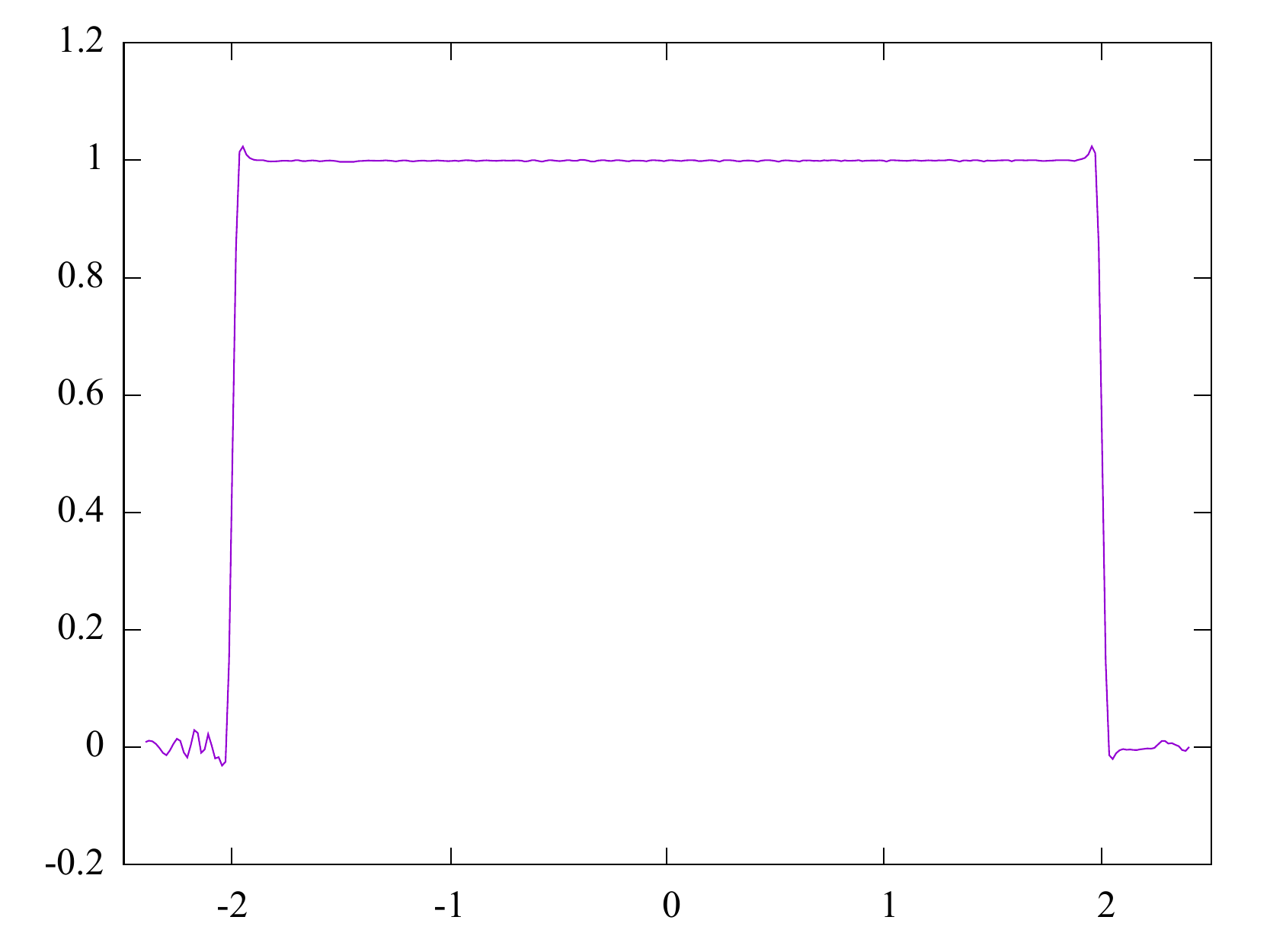}, height=5.4cm}}
}}}}
\caption{Left: global reconstruction of the disk phantom on a $801\times801$ grid. The region $\us_b=(-3.7,3.7)\times(-3.7,3.7)$ is shown. Right: profile of the reconstruction through the center of the ball, $x_c$. The $x$-axis is the distance along the profile with the origin at the center of the disk. The location of the profile is shown in the left panel.}
\label{fig:global_exp=2}
\end{figure}

\begin{figure}[h]
{\centerline{
{\hbox{
{\epsfig{file={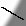}, height=2cm}}
{\epsfig{file={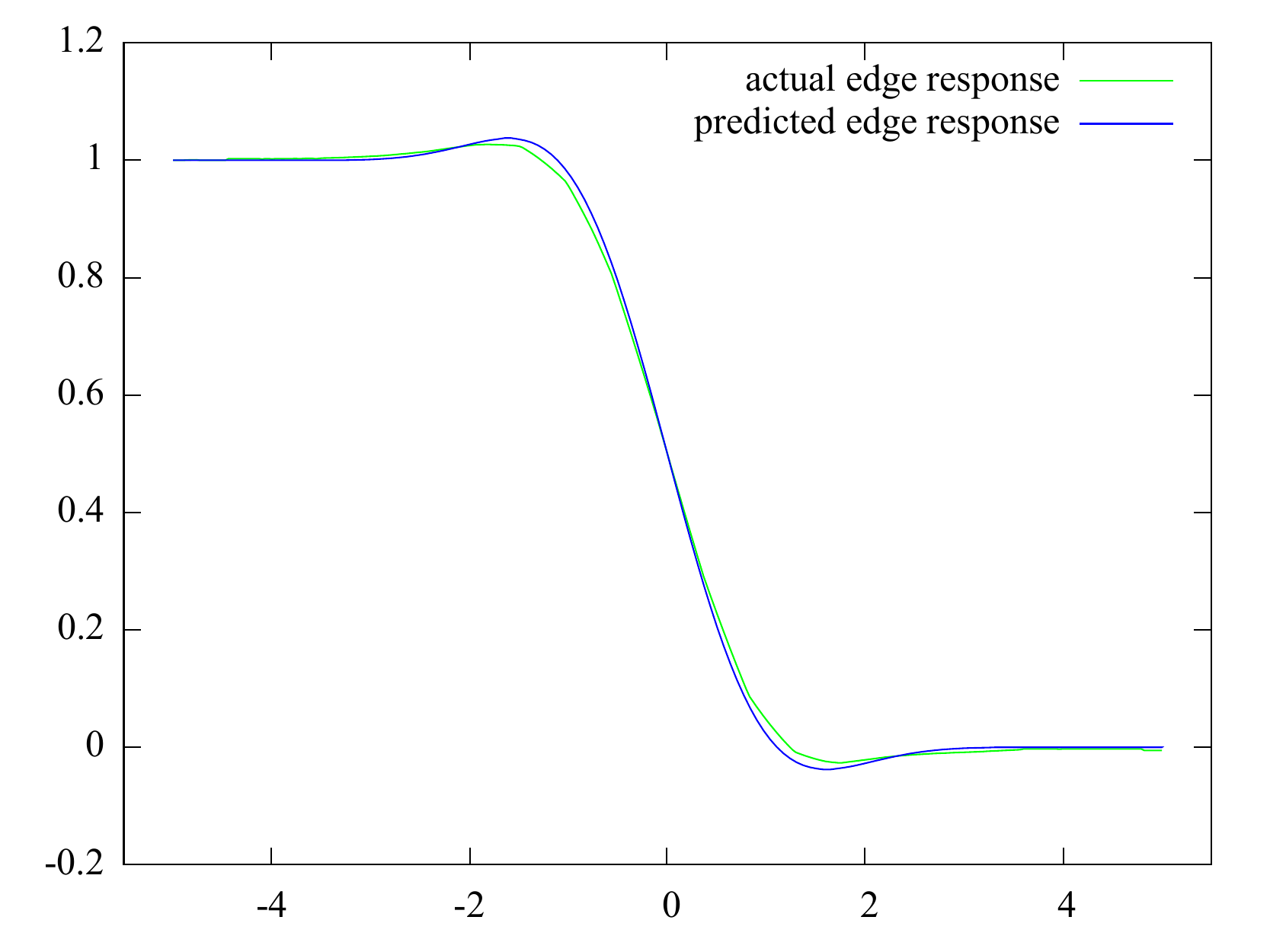}, height=5.5cm}}
}}}}
\caption{Left: local reconstruction of the disk phantom inside an ROI. The ROI is a $26\times 26$ square shown in the left panel of Figure~\ref{fig:global_exp=2}. Right: profile of the reconstruction through the center of the local ROI, $x_0$. The $x$-axis is the (signed) distance in units of $\e$ along the profile with the origin at $x_0$. The location of the profile is shown in the left panel. The numerically computed profile is in green, and the predicted profile is in blue.}
\label{fig:local_exp=2}
\end{figure}

\appendix

\section{Proof of Lemma~\ref{lem:unique sol}}\label{sec: prf lem unique sol}
\textcolor{black}{As the following shows,} $\Psi:H_0^1(\us_b)\to\br$ is strongly convex \cite[Definition 1.1.48]{Ceg2012} and \cite[Section 2.3]{Peyp2015} for each $\e>0$. \textcolor{black}{Indeed,} strong convexity means that there exists $c>0$ such that 
\be\label{strong conv}
\Psi(\la f_1+(1-\la)f_2)\le \la\Psi(f_1)+(1-\la)\Psi(f_2)-c\la(1-\la)\Vert f_2-f_1\Vert_{H^1(\us_b)}^2
\ee
for any $f_{1,2}\in H_0^1(\us_b)$ and $\la\in(0,1)$. By an easy calculation,
\be\label{ineq grad}\bs
\Vert \la\pa_x f_1+(1-\la)\pa_x f_2\Vert_{L^2(\us)}^2
\le & \la \Vert \pa_x f_1\Vert_{L^2(\us)}^2+(1-\la)\Vert \pa_x f_2\Vert_{L^2(\us)}^2\\
&-\la(1-\la)\Vert \pa_x (f_1-f_2)\Vert_{L^2(\us)}^2
\end{split}
\ee
and
\be\bs
\Vert \R(\la f_1&+(1-\la) f_2)-g_\e\Vert_{L^2(\vs)}^2\\
\le & \la \Vert \R f_1-g_\e\Vert_{L^2(\vs)}^2+(1-\la)\Vert \R f_2-g_\e \Vert_{L^2(\vs)}^2\\
&-\la(1-\la)\Vert \R (f_1-f_2)\Vert_{L^2(\vs)}^2.
\end{split}
\ee
Therefore 
\be\bs\label{interm ineq}
\Psi(\la f_1&+(1-\la) f_2)\le \la\Psi(f_1)+(1-\la)\Psi(f_2)\\
&-\la(1-\la)\big(\Vert \R (f_1-f_2)\Vert_{L^2(\vs)}^2+\Vert \pa_x (f_2-f_1)\Vert_{L^2(\us)}\big).
\end{split}
\ee
By \eqref{inverse bound},
\be\label{R-ineq}
\Vert f \Vert_{H^{-1/2}(\us)}\le c\Vert \R f \Vert_{L^2(\vs)},\ f\in\coi(\us_b).
\ee
Also
\be\label{sobolev ineq}
\Vert f \Vert_{H^1(\us)}\le c\big(\Vert f \Vert_{H^{-1/2}(\us)}+\Vert \pa_x f \Vert_{L^2(\us)}\big),\ f\in\coi(\us_b).
\ee
Applying \eqref{R-ineq} and \eqref{sobolev ineq} to \eqref{interm ineq} proves \eqref{strong conv}. Moreover, the functional is proper (its domain is not empty) and continuous $H_0^1(\us_b)\to\br$. Thus the solution to \eqref{min pb} exists and is unique for each $\e>0$ \cite[Theorem 1.3.1]{Ceg2012}, \cite[Corollary 2.20]{Peyp2015}.

\section{Proof of Lemmas~\ref{lem: matr M} and \ref{lem:conorm}}\label{sec: prf lem M conorm}

\subsection{Proof of Lemma~\ref{lem: matr M}}\label{sec: prf lem matr M}
Consider the function $H(x_*(y))$. Differentiating the first equation in \eqref{stat in x} with respect to $p$ and then using the second equation gives 
\be\label{dHdp}\bs
\pa_p H(x_*(y))=&\nu_{1*}(y)=\frac{|\dd_x H(x_*(y))|}{|\dd_x \Phi(x_*(y),\al)|}>0.
\end{split}
\ee
\textcolor{black}{The last equality follows from \eqref{two vecs} and the second equation in \eqref{stat in x}.}

Let $e$ be a unit vector tangent to $\s$ at $x_*(y)$. Denote 
\be\label{e0e matr}
e_0:=-\pa_x\Phi(x_*(y),\al)/|\pa_x\Phi(x_*(y),\al)|,\ V:=\bma e_0 & e\ema. 
\ee
Clearly, $V$ is a $2\times 2$ orthogonal matrix. Recall that $M$ is defined in \eqref{M def}. As is easily checked,
\be
M_1:=\bma 1 & 0_2^T\\ 0_2 & V^T\ema M \bma 1 & 0_2^T\\ 0_2 & V\ema
=\bma 0 & |\dd_x\Phi| & 0\\ |\dd_x\Phi|  & a & b\\ 0 & b & \dd_x^2[H-\nu_{1*}\Phi](e,e)\ema,
\ee
where $a,b$ are some quantities, whose values are irrelevant. 
Therefore, $\det M_1=\det M$ and $\text{sgn}M=\text{sgn}M_1$. 
 
Using \eqref{dHdp} and \eqref{curv S} gives
\be\label{det M prelim}\bs
\det M(y)&=-\dd_x^2[H(x_*(y))-\nu_{1*}(y)\Phi(x_*(y),\al)](e,e)|\dd_x \Phi(x,\al)|^2\\
&=(\varkappa_{\s}(x)-\varkappa_{\s_y}(x)){|\dd_x \Phi(x,\al)|^2}{|\dd_x H(x)|}.
\end{split}
\ee
Assumption~\ref{ass:f props}\eqref{del curv} completes the proof of the first statement.

To prove the second assertion in \eqref{det sign} we use Haynsworth’s inertia additivity formula \cite{Hay1968},
\be\label{sign 1}\bs
\text{sgn}M_1=&
\text{sgn}\bigg(\dd_x^2[H-\nu_{1*}\Phi](e,e)-\textcolor{black}{\bma 0 & b \ema} 
\tilde M_1^{-1} \textcolor{black}{\bma 0 \\ b \ema} \bigg)\\
&+\text{sgn}\tilde M_1,\quad \tilde M_1:=\bma 
0 & |\dd_x\Phi| \\
|\dd_x\Phi| & a \ema.
\end{split}
\ee
An easy calculation shows that \textcolor{black}{
\be\label{inert aux}
\text{sgn}\tilde M_1=0,\quad \bma 0 & b \ema 
\tilde M_1^{-1} \bma 0 \\ b \ema=b^2\big(\tilde M_1^{-1}\big)_{22}=0.
\ee}
Therefore, arguing as in \eqref{det M prelim} and using Assumption~\ref{ass:f props}\eqref{del curv}, we conclude that 
\be\label{sign 2}\bs
\text{sgn}M_1=&\text{sgn}\big(\dd_x^2[H-\nu_{1*}\Phi](e,e)\big)
=-1.
\end{split}
\ee

\subsection{Proof of Lemma~\ref{lem:conorm}}\label{sec:prf conorm}

Eqs. \eqref{GRT-f-alt}, \eqref{Rf-lim-coefs} imply that $\hat f=\R f\in I^{-3/2}(\vs,\Gamma)$ is a conormal (or, more generally, Lagrangian) distribution, see \cite[Section 25.1]{hor4}. By \textcolor{black}{\cite[Lemma 2.1 and its proof]{GW23} with $\mu=-3/2$ and $k=1$, 
$g=\chi_V\hat f\in C_0^{1/2}(V)$.} 
The leading order term in \eqref{g-lead-sing} is obtained using \cite[eq. 26, p. 360]{gs}:
\be
\CF((t-i0)^{-\sigma})=\frac{2\pi e(\sigma)}{\Gamma(\sigma)} \mu_+^{\sigma-1},\ \sigma\in\br,
\ee
\textcolor{black}{where $\mu$ is the Fourier transform variable.}

Using the decay of the lower order terms in $\fs$ as $\la\to\infty$ (namely, $\tilde R$ in \eqref{Rf-lim-coefs}), gives $\Delta g\in C_0^{3/2}(V)$. \textcolor{black}{Here we apply \cite[Lemma 2.1 and its proof]{GW23} with $\mu=-3/2$ and $k=1$ to $\pa_p\Delta g$ and $\pa_\al\Delta g$ and then use the first line in \eqref{holder} with $s=3/2$ and $k=1$.}

It remains to prove the properties of the interpolated functions $g_\e$ and $\Delta g_\e$. Pick any $g\in C_0^s(V)$. Let $\ik\in C_0^{\lceil s \rceil}(\br^2)$ be an interpolation kernel exact to degree $1$, and $g_\e$ be the corresponding interpolated function. Denote $y_h:=y+h$. We assume here for simplicity that $y_j=\e j$, $j\in\mathbb Z^2$. 

Suppose first $0<s<1$. If $|h|>\e$, it suffices to prove that $|g_\e(y)-g(y)|\le c \e^s$. Then $|g_\e(y_h)-g_\e(y)|\le c |h|^s$ follows from the triangle inequality and the assumption $g\in C_0^s(V)$. By the exactness of $\ik$,
\be\label{delta ge g}\bs
g_\e(y)-g(y)=\sum_j \ik\bigg(\frac{y-\e j}\e\bigg)\big(g(\e j)-g(y)\big).
\end{split}
\ee
Since $\ik$ is compactly supported, $|\e j-y|\le c\e$ and the number of terms in the sum is finite, and the desired assertion follows.

If $|h|<\e$, we have
\be\label{delta fs}\bs
&g_\e(y_h)-g_\e(y)
=\sum_j \bigg[\ik\bigg(\frac{y_h-\e j}\e\bigg)-\ik\bigg(\frac{y-\e j}\e\bigg)\bigg]\big(g(\e j)-g(y)\big).
\end{split}
\ee
By assumption, $\ik$ is at least as smooth as $g$, so
\be\label{delta fs aux}\bs
\bigg|\frac{\ik\big((y_h-\e j)/\e\big)-\ik\big((y-\e j)/\e\big)}{(|h|/\e)^s}\frac{g(\e j)-g(y)}{\e^s}\bigg|<c.
\end{split}
\ee
Hence $|g_\e(y_h)-g_\e(y)|\le c |h|^s$. The fact that $\Vert g_\e\Vert_{L^\infty(V)}<c$ is obvious. 

Suppose now $g\in C_0^s(V)$, $1<s<2$, i.e., we study $\Delta g$ of Lemma~\ref{lem:conorm}. As above, $\Vert g_\e\Vert_{L^\infty(V)}<c$. By the exactness of $\ik$ up to order one, differentiation of \eqref{interp-data} gives 
\be\label{ge diff}\bs
&g_\e^\prime(y)=g^\prime(y)+\sum_j \ik^\prime\bigg(\frac{y-\e j}\e\bigg)\frac{g(\e j)-\big[g(y)+\dd_y g(y)\cdot(\e j-y)\big]}\e,
\end{split}
\ee
where all the primes stand for the same derivative $\pa_{y_l}$ for some $l=1,2,\dots n$. The fraction on the right is $O(\e^{s-1})$, therefore $\Vert g_\e\Vert_{C^1(V)}<c$, $0<\e\ll1$.

To prove $\Vert g_\e^\prime\Vert_{C^{s-1}(V)}<c$, we argue similarly to the case $s<1$. As before, if $|h|\ge\e$, we only need to show that $|g_\e^\prime(y)-g^\prime(y)|\le c\e^{s-1}$. The latter inequality immediately follows from \eqref{ge diff}.

If $|h|<\e$, from \eqref{ge diff} and the exactness of $\ik$:
\be\label{delta fs prime}\bs
&g_\e^\prime(y_h)-g_\e^\prime(y)\\
&=\sum_j \bigg[\ik^\prime\bigg(\frac{y_h-\e j}\e\bigg)-\ik^\prime\bigg(\frac{y-\e j}\e\bigg)\bigg]\frac{g(\e j)-[g(y)+\dd_y g(y)\cdot(\e j-y)]}\e.
\end{split}
\ee
Dividing by $|h|^{s-1}=(|h|/\e)^{s-1}\e^{s-1}$ we get similarly to \eqref{delta fs aux} that the ratio is bounded. The proof is complete.

\section{Proof of Lemma~\ref{lem:femge}}\label{sec:prf lem femge}

The idea of the proof is to reduce $f_\e$ to $\G_\e$ by first neglecting smooth terms in $f_\e$ and then by neglecting lower order terms of finite smoothness, which still do not contribute to the DTB. This is done in Sections~\ref{ssec:appl param} and \ref{ssec:lead loc sing}, respectively.

\subsection{Simplification of $f_\e$ by subtracting smooth terms}\label{ssec:appl param}

Let $\T_\e$ denote the operator on the left in \eqref{eqn checks oper 2}. By Assumption~\ref{geom GRT}, $\R^*\R\in L^{-1}(\us)$ is elliptic. Clearly, 
\be
\T_\e=\op(\chi(x)|\xi|)\chi(x)\op\big(\tilde D(x,\xi)+\kappa\e^3|\xi|^2\big)\in L^3(\us). 
\ee
Recall that $\chi\in\coi(\us_b)$. Let $\CB_\e\in L^{-3}(\us_b)$ be a local parametrix for $\T_\e$ on the open set $\us_{1/2}:=\{x\in\us_b:\chi(x)>1/2\}\Supset\us_0$, see Figure~\ref{fig:Us}. Below, we omit the part “$x\in\us_b$” in similar sets. Then
\be\label{almost soln}\bs
&f_\e(x)=(\CB_\e F_\e)(x)+(\CH_\e f_\e)(x),\\ 
&\tilde B_\e(x,\xi)- \chi_1(\xi)\big(|\xi|\tilde D(x,\xi)+\kappa|\e\xi|^3\big)^{-1}\in S^{-1}(\us_{1/2}),
\end{split}
\ee
where 
\begin{enumerate}
\item $\tilde B_\e\in S^{-3}(\us_b)$ is the complete symbol of $\CB_\e$, i.e. $\CB_\e=\op(\tilde B_\e)$,
\item $\CH_\e: \CE^\prime(\us_b)\to C^\infty(\us_0)$ is an operator with a smooth kernel, $H_\e(x,z)\in C^\infty(\us_0\times\us_b)$, and 
\item $\chi_1\in C^\infty(\br^2)$ satisfies $\chi_1(\xi)\equiv0$, $|\xi|\le R$, and $\chi_1(\xi)\equiv 1$, $|\xi|\ge R+1$, for some $R\gg1$. 
\end{enumerate}

\begin{figure}[h]
{\centerline{
{\epsfig{file={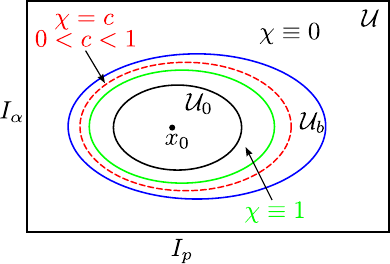}, height=4.5cm}}
}}
\caption{The sets $\us_0$, $\us_b$, $\us$, $\{\chi(x)=1\}$, and the level set $\{\chi(x)=c\}$, $0<c<1$.}
\label{fig:Us}
\end{figure}

To prove \eqref{almost soln}, use a partition of unity to write $f_\e=f_\e^{(1)}+f_\e^{(2)}$, where 
\be
\tsp f_\e^{(1)}\subset \us_{1/2}=\{\chi(x)>1/2\},\quad \tsp f_\e^{(2)}\subset \{\chi(x)<3/4\}.
\ee
Then $\CB_\e \T_\e f_\e=f_\e^{(1)}+\CH_\e^{\prime}f_\e^{(1)}+\CB_\e \T_\e f_\e^{(2)}$, where $\CH_\e^\prime\in L^{-\infty}(\us_{1/2})$. The first two terms on the right follow by the properties of a parametrix, \cite[Section I.4, Appendix]{trev1}. By construction, $\text{dist}(\{\chi(x)<3/4\},\us_0)>0$, and the property 
\be
\CB_\e \T_\e: \CE^\prime(\{\chi(x)<3/4\})\to C^\infty(\us_0)
\ee
follows by the pseudolocality of $\Psi$DOs.

The derivatives $\pa_x^{m_1}\pa_z^{m_2}H_\e(x,z)$, $m_1,m_2\in\N_0^2$, are uniformly bounded with respect to $x\in\us_0$, $z\in\us_b$, and $0<\e\ll1$. This follows from the above discussion, the way a symbol of $\CB_\e$ is constructed (see \cite[Ch. I, eqs. (4.39), (4.40)]{trev1}), and the observation that all the seminorms of $\chi_1(\xi)\big(|\xi|\tilde D(x,\xi)+\kappa|\e\xi|^3\big)^{-1}$ as member of $S^0(\us_{1/2})$ (i.e., the minimal constants $c_m$ and $c_{m_1,m_2}$ in \eqref{symbols def}, where $r=0$) are uniformly bounded for $0<\e\ll1$. Likewise, the essence of the second line in \eqref{almost soln} is that all the seminorms of the left-hand side as member of $S^{-1}(\us_{1/2})$ are uniformly bounded for $0<\e\ll1$ as well.

Therefore, by Lemma~\ref{lem:bdd sol}, 
\be\label{Hf bnd}
\CH_\e f_\e\in C^\infty(\us_0),\
\Vert \CH_\e f_\e \Vert_{C^1(\us_0)}< c,\ 0<\e\ll 1.
\ee

\subsection{Extracting the leading local singularity of $f_\e$}\label{ssec:lead loc sing}

In the preceding section we converted $f_\e$ to $\CB_\e F_\e$ modulo a $C^\infty(\us_0)$ function. We now represent $\CB_\e=\CB_\e^{(0)}+\Delta\CB_\e$ with some convenient $\CB_\e^{(0)}\in L^0(\us_b)$, $\Delta\CB_\e\in L^{-1}(\us_b)$. Then we use the mapping properties of $\Delta\CB_\e$ to show that $\Delta\CB_\e F_\e$ does not contribute to the DTB.

Using the results in \cite[Chapter VIII, Section 6.2]{trev2} and \eqref{Delta det prop}, the principal symbol of $\R^*\R$ is given by
\be\label{R*R pr symb}\bs
\frac{2\pi}{|\nu|} \frac{W^2(x,y)}{|\Delta_\Phi(x,\al)|},\ y=y(x,\xi),\nu=\nu(x,\xi),\ 
(x,\xi)\in \us\times(\br^2\setminus 0_2),
\end{split}
\ee
where the functions $y(x,\xi)\in V$ and $\nu(x,\xi)\in\br$ are obtained by solving the equations:
\be\label{pxi v2}
p=\Phi(x,\al),\ \xi=-\nu\dd_x\Phi(x,\al),\ (\al,p)\in V.
\ee
The existence of a local solution follows from Assumption~\ref{geom GRT}\eqref{visib}. Its local uniqueness and smoothness follows from Assumption~\ref{geom GRT}\eqref{li}. \textcolor{black}{In Theorem~\ref{thm:main res} we assume that the GRT data are not redundant, so only one solution exists. Note that the symbol in \eqref{R*R pr symb} is integrable at the origin $\xi=0$. More precisely, it satisfies the first inequality in \eqref{symbols def} with $n=2$ and $\de=1$.}

Define
\be\label{R*R D0}\bs
\tilde D_0(x,\xi):=& 2\pi\frac{|\xi|}{|\nu|} \frac{W^2(x,y)}{|\Delta_\Phi(x,\al)|},\ y=y(x,\xi),\nu=\nu(x,\xi),\\
(x,\xi)\in &\, \us\times(\br^2\setminus 0_2).
\end{split}
\ee
It is clear from \eqref{pxi v2} and \eqref{R*R D0} that
\be\label{homog 1}\bs
&\tilde D_0(x,\la\xi)=\tilde D_0(x,\xi),\ \la\not=0,x\in\us,\xi\in\br^2\setminus 0_2.
\end{split}
\ee
Recall that $\tilde D$ is the complete symbol of $\R^*\R$. From \eqref{R*R pr symb} and \eqref{R*R D0},
\be\label{tilde D}
|\xi|\tilde D(x,\xi)-\tilde D_0(x,\xi/|\xi|)\in S^{-1}(\us_0).
\ee
Thus we can find $\Delta\tilde B_\e\in S^{-1}(\us_b)$ (with all the seminorms uniformly bounded for $0<\e\ll1$) such that (see \eqref{almost soln})
\be\bs\label{tilde B}
&\tilde B_\e(x,\xi)=\big(\tilde D_0(x,\xi/|\xi|)+\kappa|\e\xi|^3\big)^{-1}+\Delta\tilde B_\e(x,\xi),\ x\in\us_0.
\end{split}
\ee

From \eqref{eqn checks oper 2},
\be\label{del b fe}\bs
&\op\big(\Delta\tilde B_\e(x,\xi)\big)F_\e=\CB_1\R^* g_\e,\\ 
&\CB_1:=\op\big(\Delta\tilde B_\e(x,\xi)\big)\op(\chi(x)|\xi|)\chi(z)\in L^0(\us).
\end{split}
\ee
\color{black}
Arguing similarly to Section~\ref{ssec:prelims} (from the beginning through \eqref{get lead term v3}), we get
\color{black}
\be\label{Tg Holder}\begin{split}
&\Vert \CB_1\R^* g_\e\Vert_{L^\infty(\us_b)}<c,0<\e\ll1;\\ 
&|(\CB_1\R^* g_\e)(x_0+\e\chx)-(\CB_1\R^* g_\e)(x_0)|=O(\e^{1/2}),\ \chx\in\mathcal K,
\end{split}
\ee 
where $\mathcal K$ is any compact set.

Combining \eqref{lead sing fe}, \eqref{almost soln}, \eqref{Hf bnd}, \eqref{tilde B} and \eqref{Tg Holder} finishes the proof.

\section{Proof of Lemmas~\ref{lem:G props}, \ref{lem:Fel bnd}}\label{sec: G Fel}

\subsection{Proof of Lemma~\ref{lem:G props}}\label{sec:prf G props}
We begin by stating a lemma. Its proof is in Appendix~\ref{sec:prf H props}.
\begin{lemma}\label{lem:H props} Suppose Assumptions~\ref{ass:Phi}--\ref{ass:interp ker} are satisfied. Pick any $A>0$. One has
\be\label{H props 1}
H(\cht,\chq;*)\equiv 0,\ \cht<-c.
\ee
Also, for some smooth and bounded $\varkappa(t,q;\theta)$,
\be\bs
\label{H props 2 a}
&\pa_{\cht}^k \big[H(\cht,\chq;*)-\varkappa(\e\cht,\e\chq;\theta)\cht^{-1/2}\big]=O(\cht^{-3/2}),\\
&k=0,1,\ \cht\to+\infty,\ \cht<A/\e,
\end{split}
\ee
and
\be\bs
\label{H props 2 b}
& \pa_{\cht} \big[\varkappa(\e\cht,\e\chq;\theta)\cht^{-1/2}\big]=O(\cht^{-3/2}),\  
\cht\to+\infty,\ \cht<A/\e.
\end{split}
\ee
The big-$O$ terms are uniform with respect to $\chq$, $\theta$, and $\e$. These variables are restricted to the sets $|\chq|\le A/\e$, $\theta\in I_\al$,  and $0<\e\ll1$.
\end{lemma}

For simplicity, we drop the arguments $\chq$, $\theta$ and $*=(x,\theta,\e)$ from all the functions. 
From \eqref{H props 1} and \eqref{H props 2 a} with $k=0$ we get $H(\cht)=O(\cht^{-1/2})$, $s\to\infty$. From Assumption~\ref{ass:interp ker}\eqref{ikcont}, \eqref{h_eps}, and \eqref{sum j}, it follows that $H(\cht,\chq;*)$ and $\pa_{\cht}H(\cht,\chq;*)$ are uniformly bounded. The limit as $\cht\to-\infty$ in \eqref{G props} easily follows from \eqref{key int}.

Next, suppose $\cht\to\infty$. Then
\be\label{key int 2 parts}\bs
&G(\cht)=\int_{\br}\chii(\e(s-\cht))\frac {H_1(s)}{s-\cht} \dd s
+\ioi\chii(\e (s-\cht))\frac {\varkappa(\e (s+\chq))}{(s-\cht)s^{1/2}} \dd s,\\
&H_1(s):=H(s)-\varkappa(\e (s+\chq)) s_+^{-1/2},
\end{split}
\ee
where $\varkappa(\cdot)$ is the same as in \eqref{H props 2 a}, \eqref{H props 2 b}. Using \eqref{H props 1}, split the first of the integrals into four
\be\label{5 subparts}\bs
&\bigg(\int_{-c}^{t/2}+\int_{t/2}^{\cht-1}+\int_{\cht-1}^{\cht+1}+\int_{\cht+1}^\infty\bigg)\chii(\e(s-\cht))\frac {H_1(s)}{s-\cht} \dd s\\
&=J_1+\dots+J_4.
\end{split}
\ee
By \eqref{H props 2 a} with $k=0$, $H_1(s)=O(s^{-3/2})$, $s\to\infty$. It is trivial to see that $J_1=O(1/\cht)$ and  $J_2=O(\cht^{-3/2}\ln\cht)$. In the first integral we use $s-\cht\asymp \cht$, and in the second one we use that $s\asymp \cht$ to bound above $H_1(s)$ by $c\cht^{-3/2}$. Breaking up the integral over $[\cht+1,\infty)$ into the integrals over $[\cht+1,2\cht]$ and $[2\cht,\infty)$ gives $J_4=O(\cht^{-3/2}\ln\cht)$. Finally, using \eqref{H props 2 a} with $k=1$ and that $|s-\cht|\le 1$, $\e\cht\le c$, we get
\be
\pa_s\big[\chii(\e(s-\cht))H_1(s)\big]=O(\cht^{-3/2}),
\ee
which implies $J_3=O(\cht^{-3/2})$. Combining the four estimates proves that the first integral in \eqref{key int 2 parts} is $O(1/\cht)$.

Consider the second integral in \eqref{key int 2 parts}, which we denote $J(\cht,\e)$. Recall that $\cht\to\infty$. Write $J(\cht,\e)$ in the form
\be\bs
J(\cht,\e):=&\e^{1/2}\ioi \frac{\chii(s-t)\varkappa(s+q)}{(s-t)s^{1/2}}\dd s\\
=&\e^{1/2}\ioi \frac{(\chii(s-t)-\chii(0))\varkappa(s+q)}{(s-t)s^{1/2}}\dd s\\
&+\e^{1/2}\chii(0)\ioi \frac{\varkappa(s+q)}{(s-t)s^{1/2}}\dd s,\ q=\e\chq,\ t=\e\cht.
\end{split}
\ee
It is not hard to show that the integral on the second line is uniformly bounded. Recall that $\chii\in\coi(\br)$. The identity \cite[Eq. 2.2.5.26]{pbm1} with $\al=1/2$:
\be
\ioi \frac{\dd s}{(s-t)s^{1/2}}=0,\ t>0,
\ee
implies
\be
\ioi \frac{\varkappa(s+q)}{(s-t)s^{1/2}}\dd s=\ioi \frac{\varkappa(s+q)-\varkappa(t+q)}{(s-t)s^{1/2}}\dd s
=O(1),
\ee
Here the term $O(1)$ is uniform with respect to $0<t<A,|q|<A$. Therefore $J(\cht,\e)=O(\e^{1/2})$. 

Given that $\pa_{\cht} H$ is bounded, it is clear that $G$ is bounded with $\cht$ in compact sets. Combining all the results proves \eqref{G props}.

\subsection{Proof of Lemma~\ref{lem:H props}}\label{sec:prf H props}
Assertion \eqref{H props 1} is obvious from \eqref{sum j} because $h$ is compactly supported. To prove \eqref{H props 2 a}, \eqref{H props 2 b}, begin by proving the following auxiliary lemma.
\begin{lemma}\label{lem:aux asymp}
Pick any $A>0$ and a function $f\in C^2(\br)$. Let $\ik\in C_0^2(\br)$ be an interpolating kernel, which is exact up to degree one. One has 
\be\label{exp accuracy}\bs
\pa_{\cht}^k\big[&\tsum_j \ik(\cht+\chq-j) f(\e j)(j-\chq)_+^{1/2}-f(\e (\cht+\chq))\cht^{1/2}\big]=O\big(\cht^{-3/2}\big),\\ 
&\cht\to+\infty,\ \e\to0,\ \e \cht<A,\ k=0,1,2.
\end{split}
\ee
The big-$O$ term is uniform in $|\chq|\le A/\e$.
\end{lemma}
\begin{proof}
Denote $g(\cht):=f(\e \cht)(\cht-\chq)^{1/2}$ and $u:=\cht+\chq$. The difference in brackets in \eqref{exp accuracy} becomes $g_\e(u)-g(u)$, where $g_\e$ stands for the interpolated $g$ similarly to \eqref{interp-data} (see \eqref{interp main}). Taylor expanding $g$ gives
\be\label{T exp}\bs
&g(j)=g(u)+g^\prime(u)(j-u)+R_1(j,u),\ j-u\in\tsp\ik,\\ 
&R_1(j,u)=\int_u^j g^{\prime\prime}(s)(j-s)\dd s.
\end{split}
\ee
The exactness of $\ik$ up to degree one implies:
\be \label{interp main}
g_\e(u):=\tsum_j \ik(u-j) g(j)=g(u)+\tsum_j\ik(u-j)R_1(j,u).
\ee
By assumption, $0<\e t<A$ and $\e|\chq|\le A$, so $|\e\cht+\e\chq|<2A$ and
\be\label{g derivs}
|g^{\prime\prime}(u)|=\left|\pa_{\cht}^2\left(f(\e \cht+\e\chq)\cht^{1/2}\right)\right|=O\big(\cht^{-3/2}\big),\ \e \cht<A,\ \cht\to+\infty. 
\ee
The values of $j,s,u$ used in \eqref{T exp} and \eqref{interp main} satisfy $|j-s|\le|j-u|\le\text{diam}(\tsp\ik)$. Hence, by \eqref{T exp} and \eqref{g derivs}
\be\label{rem ders}\bs
&R_1(j,u)=O(\cht^{-3/2}),\ j-u\in\tsp\ik,\ \e \cht<A,\ \cht\to+\infty.
\end{split}
\ee
Furthermore, using that $\ik$ is exact up to order one,
\be\label{R 1rst der}\bs
\pa_u R_1(j,u)=&-g^{\prime\prime}(u)(j-u),\\ 
g_\e^{\prime}(u)-g^{\prime}(u)=&\pa_u \tsum_j\ik(u-j)R_1(j,u)\\
=&\tsum_j\ik^\prime(u-j)R_1(j,u)-g^{\prime\prime}(u)\tsum_j\ik(u-j)(j-u)\\
\hspace{2cm}=&\tsum_j\ik^\prime(u-j)R_1(j,u),
\end{split}
\ee
and 
\be\label{R 2nd der}\bs
g_\e^{\prime\prime}(u)-g^{\prime\prime}(u)
=&\tsum_j\ik^{\prime\prime}(u-j)R_1(j,u)-g^{\prime\prime}(u)\tsum_j\ik^\prime(u-j)(j-u)\\
\hspace{2cm}=&\tsum_j\ik^{\prime\prime}(u-j)R_1(j,u)-g^{\prime\prime}(u).
\end{split}
\ee
Combining \eqref{interp main}--\eqref{R 2nd der} and using that $\ik^{\prime\prime}\in L^\infty(\br)$ finishes the proof.
\end{proof}

To prove \eqref{H props 2 a}, in view of \eqref{h_eps} and \eqref{sum j}, we consider
\be\label{sum j2 v1}
H_{aux}(\cht):=\tsum_{j_2}\ik_p(\cht+\chq-j_2)a_0(\theta,\e j_2)(j_2-\chq)_+^{1/2}.
\ee
By Lemma~\ref{lem:aux asymp}, 
\be\label{aymp v1}
\pa_{\cht}^{k-1}\left(\pa_{\cht}H_{aux}(\cht)-\pa_{\cht}\big[a_0(\theta,\e (\cht+\chq))\cht^{1/2}]\right)=O(\cht^{-3/2}),\ k=1,2.
\ee
We moved one derivative inside the parentheses because $H$ is defined in terms of $h^\prime$ in \eqref{sum j}.
Replacing $\cht$ with $[\Phi(\theta+\e\mu s)-\Phi(\theta)]/\e+\cht$, multiplying both sides by 
$K(\theta+\e\mu s)\ik_\al(s)$, and integrating with respect to $s$ (cf. \eqref{h_eps}, \eqref{sum j}) finishes the proof. This also gives a formula for $\varkappa$:
\be\label{varkappa def}
\varkappa(t,q;\theta)=t \pa_t a_0(\theta,t+q)+(1/2)a_0(\theta,t+q).
\ee
This shows that $\varkappa$ is smooth and bounded with all derivatives on any compact set. The claim \eqref{H props 2 b} immediately follows as well. 

\subsection{Proof of Lemma~\ref{lem:Fel bnd}}\label{sec:prf Fel bnd}

\subsubsection{Estimation away from the stationary points of $Q$.}
It follows from \eqref{G props} that the magnitude of $G$ can be controlled only by its first argument, $\cht$. Therefore, in this section, we omit the dependence of $G(\cht,\chq;*)$ on $\chq$ and $*=(x,\theta,\e)$. In view of \eqref{Fel def}, we study the convergence of the sum
\be\label{last sum}
J_\e:=\e^{1/2}\tsum_{j_1\in J_\al}G(Q(x,\al_{j_1})/\e),\ x\in U^\prime.
\ee
Recall that $Q$ depends on
the selected $P(\al)$. 

Fix any $0<\de\ll1$ and $x\in U^\prime$. Since $x$ is fixed, we can drop it from notation. If $|Q(\al)|\ge \de$ on $I_\al$, using \eqref{G props} in \eqref{last sum} gives the same bound $|J_\e|\le c\e^{1/2}(1/\e)\e^{1/2}=O(1)$ for each of the partial sums: over $Q(\al_{j_1})>0$ and over $Q(\al_{j_1})<0$. Thus, in what follows we may assume $|Q(\al)|\le\de$ for some $\al\in I_\al$. If $|Q_\al^\prime(\al)|\ge\de$ on $I_\al$, then $|Q_\al^\prime(\al)|\asymp |\al-a_0|$ on $I_\al$ and
\be\label{Je est 1}
|J_\e|\le c\e^{1/2}\sum_{j_1=1}^{1/\e}\big(j_1^{-1/2}+\e^{1/2}+j_1^{-1}\big)<c,\ \e\to0.
\ee

\subsubsection{Estimation near a stationary point of $Q$.}
\textcolor{black}{Based on the results of the preceding subsection, it remains to consider the case when $|Q(a_1)|$, $|Q_\al^\prime(a_2)|\le\de$ for some $a_1,a_2\in I_\al$. Since $I_\al$ can be made as small as we like (but finitely small), it is instructive to consider the case when $Q(z,\al)=Q_\al^\prime(z,\al)=0$ for some $(z,\al)\in \us\times \I_\al$.} This happens if and only if $z\in\s$ and $P(\al)$ is associated with $z$. 

To prove sufficiency, suppose there is a point $(z,\al)\in \us\times \I_\al$ such that 
\be\bs
&Q(z,\al)=\Phi(z,\al)-\Phi(x(\al),\al)=0, \\
&Q_\al^\prime(z,\al)=\Phi_\al^\prime(z,\al)-\Phi_\al^\prime(x(\al),\al)=0.
\end{split}
\ee
The Bolker condition, Assumption~\ref{geom GRT}\eqref{bolk}, gives $z=x(\al)$.  

The necessity is established in the following Lemma~\ref{lem:Q 2nd der}, which is proven in Appendix~\ref{sec:Qprpr}.

\begin{lemma}\label{lem:Q 2nd der} Suppose Assumptions~\ref{ass:Phi}, \ref{geom GRT}, and \ref{ass:f props} are satisfied. Pick any $\tilde y\in\Gamma$. Let $\s_{\tilde y}$ be tangent to $\s$ at $\tilde x$. Let $x(\al)$ be the solution to \eqref{for Q 1} associated with $\tilde x$ which satisfies $\tilde x=x(\tilde \al)$, see Figure~\ref{fig:tangency}. One has 
\be\label{Q 2nd der}\begin{split}
&Q(\tilde x,\tilde \al)=Q_\al^{\prime}(\tilde x,\tilde \al)=0,\\ 
&Q_{\al\al}^{\prime\prime}(\tilde x,\tilde \al)=\det M(\tilde y)\frac{|x_\al^\prime(\tilde \al)|^2}{|\dd_x\Phi(\tilde x,\tilde \al)||\dd_x H(\tilde x)|}>0
\end{split}
\ee
and
\be\label{Del Phi res}
|\Delta_\Phi(\tilde x,\tilde \al)|=|\det M(\tilde y)|\frac{|x_\al^\prime(\tilde \al)|}{|\dd_x H(\tilde x)|}.
\ee
\end{lemma}

By the above argument, we can suppose that the point $x$ selected following \eqref{last sum} is located sufficiently close to the curve $x(\al),\al\in I_\al$. \textcolor{black}{Otherwise, by shrinking $I_\al$ if necessary, we can ensure that either $|Q(\al)|\ge\de$ or $|Q_\al^\prime(\al)|\ge\de$ on $I_\al$. By Lemma~\ref{lem:Q 2nd der}, 
$Q_{\al\al}^{\prime\prime}(x(\al),\al)\ge c>0$ for all $\al\in I_\al$.} By continuity, $Q_{\al\al}^{\prime\prime}(\al)\ge c>0$ for all $\al\in I_\al$. 

Due to the lower bound on $Q_{\al\al}^{\prime\prime}$, the number of $\al_{j_1}$ in the set $\{\al\in I_\al:|Q(\al)|\le c\e\}$ is at most $A\e^{-1/2}$, where $A$ is independent of $x$. Since $G$ is bounded, the total contribution of such $\al_{j_1}$ to $J_\e$ is bounded. Hence, by \eqref{G props}, it remains to estimate the following two sums
\be\label{last sums pm}\bs
J_\e^+:=&\e^{1/2}\sum_{\al_{j_1}\in I_\al^+}\bigg[\e^{1/2}+\frac\e{Q(\al_{j_1})}\bigg],\ I_\al^+:=\{\al\in I_\al:\, Q(\al)\ge\e\},\\
J_\e^-:=&\e^{1/2}\sum_{\al_{j_1}\in I_\al^-}\bigg(-\frac \e{Q(\al_{j_1})}\bigg)^{1/2},\ I_\al^-:=\{\al\in I_\al:\, Q(\al)\le-\e\}.
\end{split}
\ee
For estimation purposes, we can assume $Q^{\prime\prime}(\al)\ge 1$, $\al\in I_\al$ (rather than the more cumbersome $Q^{\prime\prime}(\al)\ge c$). When the argument $x$ of $Q$ is omitted and we write $Q(\al)$, then a prime denotes the derivative with respect to $\al$.

\color{black}
Since $Q^{\prime\prime}(\al)\ge 1$, $I_\al^+$ is the union of no more than two intervals, and $I_\al^-$ is a single interval. Also, the equation $Q^\prime(\al)=0$, $\al\in I_\al$, has at most one solution. 

Begin with $J_\e^+$. Let $[a,b]$ be one of at most three intervals that make up $I_\al^+$, subject to the condition $Q^\prime(\al)\not=0$ on $(a,b)$. Suppose $Q^\prime(\al)>0$ on $(a,b)$. The proof in the case $Q^\prime(\al)<0$ is analogous.

By construction, $Q(a)\ge\e$. Consider the function $Q_1(\al)=\e+(\al-a)^2/2$, $\al\in[a,b]$, which satisfies $Q_1(a)=\e$, $Q_1^\prime(a)=0$, and $Q_1^{\prime\prime}(\al)=1$ on $[a,b]$. Clearly, $Q_1(\al)\le Q(\al)$, $\al\in[a,b]$. Therefore, the value of $J_\e^+$ is increased if we replace $Q$ with $Q_1$. There are $O(1/\e)$ terms in the first sum in \eqref{last sums pm}, so the term $O(\e^{1/2})$ in the brackets in $J_\e^{+}$ contributes an $O(1)$ term. Then
\be\label{Jpl est v1}\bs
J_\e^{+}\le & c+c\e^{1/2}\sum_{j} \frac \e{\e+(\e j)^2/2}\le c.
\end{split}
\ee

Next, consider $J_\e^-$. Let $[a,b]$ be one of at most two intervals that make up $I_\al^-$, subject to the condition $Q^\prime(\al)\not=0$ on $(a,b)$. Again, we can assume $Q^\prime(\al)>0$ on $(a,b)$. Consider the function $Q_1(\al)=Q(a)+(\al-a)^2/2$, $\al\ge a$, which satisfies $Q_1(a)=Q(a)$, $Q_1^\prime(a)=0$, and $Q_1^{\prime\prime}(\al)=1$ for $\al\ge a$ (see the green curve in Fig.~\ref{fig:diag below}). By construction, $Q_a^\prime(\al)\le Q^\prime(\al)$, $\al\in[a,b]$. Pick any $t\in[a,b]$, and find $t_1$ by solving $Q_1(t_1)=Q(t)$, i.e. $t_1=a+[2(Q(t)-Q(a))]^{1/2}$ (see Fig.~\ref{fig:diag below}). We claim that $Q_1^\prime(t_1)\le Q^\prime(t)$.

\begin{figure}[h]
{\centerline{\hbox{
{\epsfig{file={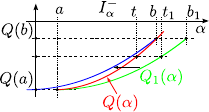}, width=5.5cm}}
}}}
\caption{Construction used to estimate $J_\e^-$.}
\label{fig:diag below}
\end{figure}

To prove the claim, we need to show that $t_1-a=[2 (Q(t)-Q(a))]^{1/2}<Q^\prime(t)$. This is equivalent to
\be
2\bigg[Q^\prime(a)t+\int_a^t (t-s)Q^{\prime\prime}(s)\dd s\bigg] \le \bigg[Q^\prime(a)+\int_a^t Q^{\prime\prime}(s)\dd s\bigg]^2,\ a\le t\le b.
\ee
Squaring the right-hand side and using that $Q^\prime(a)\ge 0$ and $Q^{\prime\prime}(s)\ge 1$, we see that the desired result follows if we prove
\be
2\int_a^t (t-s)Q^{\prime\prime}(s)\dd s \le \bigg[\int_a^t Q^{\prime\prime}(s)\dd s\bigg]^2.
\ee
Writing $Q^{\prime\prime}(s)=1+\phi(s)$, where $\phi(s)\ge0$, and simplifying proves the claim.

Let $b_1$ be obtained by solving $Q_1(b_1)=Q(b)$. We have $Q_1^\prime(\al)\le Q^\prime(\al)$, so $b_1\ge b$. The preceding argument implies that if the graph of $Q_1$ is shifted by $b_1-b$ to the left (to obtain the blue curve in Fig.~\ref{fig:diag below} as indicated by the black arrow), then it will lie above the graph of $Q$ (shown in red). This means that $Q_1(\al+(b_1-b))\ge Q(\al)$, $\al\in[a,b]$. Hence, if we replace $Q(\al)$ by $Q_1(\al+(b_1-b))$ and $[a,b]$ by $[a-(b_1-b),b]$, then the value of $J_\e^-$ is increased. The shift, of course, is irrelevant for estimation purposes, so we instead consider $Q_1(\al)$ and $[a,b_1]$. This leads to the following bound
\be\label{Jmn est v1}\bs
J_\e^-&\le c\e^{1/2}\sum_{h+(\e j)^2/2<-\e} \bigg[\frac {\e}{-h-(\e j)^2/2}\bigg]^{1/2}
\le c\int_0^{|2h|^{1/2}} \frac{\dd x}{(|2h|-x^2)^{1/2}}<\infty.
\end{split}
\ee
Note that the value of the last integral is independent of $h$.

\color{black}
It is clear that in the above arguments all the bounds are uniform with respect to $x\in U^\prime$.

\subsection{Proof of Lemma~\ref{lem:Q 2nd der}}\label{sec:Qprpr}

Recall that $P(\al)=\Phi(x(\al),\al)$ (see the paragraph below \eqref{for Q 1}). Therefore
\be\label{for Q def}
Q(\al)=\Phi(\tilde x,\al)-P(\al)=\Phi(\tilde x,\al)-\Phi(x(\al),\al).
\ee
By construction, $x^\prime(\al)$ is tangent to $\s$ and $\dd_x\Phi(x(\al),\al)$ is conormal to $\s$ at $x(\al)$ (see \eqref{for Q 1}). Differentiation of \eqref{for Q def} with respect to $\al$ gives
\be\label{Q_al prime}\bs
Q^\prime(\al)=&\Phi_\al^\prime(\tilde x,\al)-\big[\dd_x\Phi(x(\al),\al)\cdot x_\al^\prime+\Phi_\al^\prime(x(\al),\al)\big]\\
=&\Phi_\al^\prime(\tilde x,\al)-\Phi_\al^\prime(x(\al),\al).
\end{split}
\ee
Using that $x(\tilde \al)=\tilde x$ proves the first line in \eqref{Q 2nd der}.

Next, differentiate \eqref{Q_al prime} with respect to $\al$ and set $\al=\tilde\al$ to obtain
\be\label{for Q 2}
Q_{\al\al}^{\prime\prime}=-\dd_x\Phi_{\al}^{\prime}\cdot x_\al^\prime.
\ee
Differentiate the second equation in \eqref{for Q 1} with respect to $\al$ and set $\al=\tilde\al$:
\be\label{for Q 3}
\dd_x^2 H\, x_\al^\prime=\nu_1^\prime\dd_x\Phi+\nu_1\big[\dd_x^2\Phi\,x_\al^\prime+\dd_x\Phi_\al^\prime\big].
\ee
Apply \eqref{for Q 3} to $x_\al^\prime$ and use that $\dd_x\Phi \cdot x_\al^\prime\equiv0$ a second time to get
\be\label{for Q 4}
 \dd_x^2 H(x_\al^\prime,x_\al^\prime)=\nu_1\big[\dd_x^2\Phi(x_\al^\prime,x_\al^\prime)+\dd_x\Phi_\al^\prime\cdot x_\al^\prime\big].
\ee
Using \eqref{for Q 4} in \eqref{for Q 2} and applying Lemma~\ref{lem: matr M} finally gives:
\be\label{for Q 6}\bs
Q_{\al\al}^{\prime\prime}=&-\big[(1/\nu_1)\dd_x^2 H(x_\al^\prime,x_\al^\prime)-\dd_x^2 \Phi(x_\al^\prime,x_\al^\prime)\big]\\
=&-\dd_x^2[H-\nu_1\Phi](e,e)\frac{|x_\al^\prime|^2}{\nu_{1}},\ e:=x_\al^\prime/|x_\al^\prime|.
\end{split}
\ee
Combining this with \eqref{dHdp} and \eqref{det M prelim} proves the second line in \eqref{Q 2nd der}.

Next we prove \eqref{Del Phi res}. By \eqref{Delta det}, \eqref{Delta det prop},
$\Delta_\Phi(x,\al)=\det \bma \dd_x\Phi \\ \dd_x \Phi_\al^\prime\ema$. Considering the matrix product $\Delta_\Phi(x,\al) V$, where $V=\bma e_0 & e\ema$ (see \eqref{e0e matr}), gives
\be\label{det Del prelim}
|\Delta_\Phi(x,\al)|=|\dd_x\Phi||\dd_x\Phi_\al^\prime\cdot e|.
\ee
Using \eqref{for Q 2}, \eqref{for Q 6}, and then \eqref{det M prelim} gives 
\be\label{det Del 2}\bs
|\dd_x\Phi||\dd_x\Phi_\al^\prime\cdot e|&=
\big|\dd_x^2[H-\nu_1\Phi](e,e)\big|\frac{|\dd_x\Phi||x_\al^\prime|}{|\nu_1|}
=\frac{|\text{det}M||x_\al^\prime|}{|\nu_1||\dd_x\Phi|},
\end{split}
\ee
and \eqref{Del Phi res} follows from \eqref{dHdp}.

\section{Proof of Lemmas~\ref{lem:delG props}, \ref{lem:del F0}, and \ref{lem:del F0 v2} }\label{sec:prf delG and F}

\subsection{Proof of Lemma~\ref{lem:delG props}}\label{sec:prf delG props}
For convenience, we drop the arguments $\chq$, $\theta$ and $\e$ from all the functions. 
\color{black}
By \eqref{key int} and \eqref{key int dg},
\be\label{ki 2 parts}\bs
&\Delta G(r,\cht)=G(r+\cht)-G(\cht)=\int_{\br}\chii(\e(s-\cht))\frac {\Delta H(r,s)}{s-\cht} \dd s.
\end{split}
\ee

Begin with \eqref{delG prop 1}. Adding \eqref{H props 2 a} with $k=1$ and \eqref{H props 2 b} yields $|\Delta H(r,\cht)|\le c|r|(1+|\cht|)^{-3/2}$. The bound $|\Delta G|\le c|r|/|\cht|$, $\cht\to-\infty$, follows from \eqref{H props 1}. 

Suppose $\cht\to+\infty$. Using \eqref{H props 1}, the integral in \eqref{ki 2 parts} is split into three
\be\label{3 subparts}\bs
&\bigg(\int_{-c}^{\cht-1}+\int_{\cht-1}^{\cht+1}+\int_{\cht+1}^\infty\bigg)\chii(\e(s-\cht))\frac {\Delta H(r,s)}{s-\cht} \dd s\\
&=:J_1+J_2+J_3.
\end{split}
\ee
Breaking up the integral over $[-c,\cht-1]$ into the integrals over $[-c,\cht/2]$ and $[\cht/2,\cht-1]$ we get  $|J_1|\le c|r|/\cht$. Breaking up the integral over $[\cht+1,\infty)$ into the integrals over $[\cht+1,2\cht]$ and $[2\cht,\infty)$ gives $|J_3|\le c|r|\cht^{-3/2}\ln\cht$. 

Using the function $(\chii(\e \chw)-1)/\chw$, we obtain
\be\label{J2 est1}\bs
&J_2=J_2^\prime+O(\e |r|\cht^{-3/2}),\\
&J_2^\prime=\int_{-1}^{1}\frac1{\chw} \Delta H(r,\chw+\cht)\dd \chw=\int_{-1}^{1}\frac1{\chw} \int_{w}^{r+w} \pa_{\cht}H(s+\cht)\dd s\dd \chw.
\end{split}
\ee
To estimate $J_2^\prime$, we may assume without loss of generality that $0<r<1$ and rewrite $J_2^\prime$ as follows 
\be\label{J2pr}\bs
J_2^\prime=&\int_{-1}^{-1+r}\pa_{\cht}H(s+\cht)\int_{-1}^s\frac{\dd \chw}{\chw}\dd s
+\int_{-1+r}^1\pa_{\cht}H(s+\cht)\int_{s-r}^s\frac{\dd \chw}{\chw}\dd s\\
&+\int_1^{1+r}\pa_{\cht}H(s+\cht)\int_{s-r}^1\frac{\dd \chw}{\chw}\dd s.
\end{split}
\ee
Adding \eqref{H props 2 a} with $k=1$ and \eqref{H props 2 b} yields $\pa_{\cht} H(\cht)=O(\cht^{-3/2})$. After simple but tedious calculations, \eqref{J2pr} implies:
\be\label{J2pr1}\bs
|J_2^\prime|\le & c\cht^{-3/2}\bigg[\int_{-1}^{-1+r}|\ln|s||\dd s
+\int_{-1+r}^1\bigg|\ln\frac{|s|}{|s-r|}\bigg|\dd s+\int_1^{1+r}|\ln|s-r||\dd s\bigg]\\
=& O(r\ln(1/r)\cht^{-3/2}).
\end{split}
\ee
Therefore, $J_2=O(r\ln(1/r)\cht^{-3/2})$. 

The above argument also proves that $\Delta G(r,\cht)$ is bounded on compact sets. Combining the estimates of $J_1,J_2,J_3$ finishes the proof of \eqref{delG prop 1}.

It is easy to check that
\be\bs
&\Delta G(r+h,\cht)-\Delta G(r,\cht)=\Delta G(h,r+\cht),\\
&\Delta G(r,\cht+h)-\Delta G(r,\cht)=\Delta G(h,r+\cht)-\Delta G(h,\cht),
\end{split}
\ee
Since $|r|\le A_2$, \eqref{delG prop 2} and \eqref{delG prop 3} follow from \eqref{delG prop 1}.
\color{black}

\subsection{Proof of Lemma~\ref{lem:del F0}}\label{sec:prf del F0}

Recall that the parameter $A\gg1$ used in Lemma~\ref{lem:del F0} was selected above the equation \eqref{Q bnd}. 

First we investigate the dependence of $G$ (which is implicitly used in \eqref{F via G part2} as part of $\Delta G$) on $\e$ in its last argument. We need the following technical result, which is proven at the end of this section. Recall that $G_0$ is defined in \eqref{G0H0}, and $G$ -- in \eqref{key int}.
\begin{lemma}\label{lem:G approx} Pick any $A_1\gg1$. Suppose $|\cht|<A_1$, $|\e\chq-P(0)|<\e^{1/2} A_1$, and $|\theta|<\e^{1/2}A_1$. Under the assumptions of Theorem~\ref{thm:main res} one has
\be\label{G simpl v3}\bs
G(\cht,\chq;\theta,\e)=\mu G_0(\cht,\chq)+O(\e^{1/2})
\end{split}
\ee
uniformly in $\cht$, $\chq$, and $\theta$ restricted to the indicated sets.
\end{lemma}

Partition the interval $[-A,A]$ into the union of $K\gg1$ non-overlapping intervals of length $2A/K$. Let these intervals be denoted $B_k$, and $s_k$ be their centers, $k=1,2,\dots,K$. Recall that $P(\al)$ is associated with $x_0=0_2\in\s$. Clearly, 
\be\label{P exp}\begin{split}
\frac{P(\e^{1/2}s)}\e=&\frac{P(\e^{1/2}s_k)}\e+P^\prime(\e^{1/2}s_k)\frac{s-s_k}{\e^{1/2}}+O(1/K^{2})\\
=&\frac{P(\e^{1/2}s_k)}\e+\big[P^\prime(0)+O(\e^{1/2}s_k)\big]\frac{s-s_k}{\e^{1/2}}+O(1/K^{2})\\
=&\frac{P(\e^{1/2}s_k)}\e+P^\prime(0)\frac{s-s_k}{\e^{1/2}}+O(1/K),\ s\in B_k.
\end{split}
\ee
Here we have used that $A>0$ is fixed and $|s_k|<A$. Otherwise, the term $O(1/K)$ may grow with $A$. The same applies to many other bounds and big-$O$ terms in this section. 

\color{black}
Observe that $G_0(\cht,\chq)$ in \eqref{G0H0} is obtained from $G(\cht,\chq;x,\theta,\e)$ by setting $x=0_2$, $\theta=0$, $\e=0$ and $a_0(\theta,\e j_2)\equiv a_0(y_0)$ in \eqref{h_eps}, \eqref{sum j} and \eqref{key int}. Therefore, $G_0$ satisfies \eqref{delG prop 1}:
\be\label{del G0 bnd a}\bs
&|G_0(\cht+r,\chq)-G_0(\cht,\chq)|\le c r\ln(1/r),\  0<r<1,\ |\cht|<A,\ \chq\in\br.
\end{split}
\ee

To study the smoothness of $G_0$ with respect to the second argument, denote similarly to \eqref{G0H0}:
\be\label{G0H0-alt}\bs
&\check G_0(\cht,\chq_1,\chq_2):=\int_{\br}\frac {H_0(\chw+\cht,\chq_1,\chq_2)}\chw \dd \chw,\\
&\check H_0(\cht,\chq_1,\chq_2):=a_0(y_0)\sum_{j_2\in\mathbb Z}h_0^\prime(\cht+\chq_1-j_2)(j_2-\chq_2)_+^{1/2}.
\end{split}
\ee
Then, $G_0(\cht,\chq)=\check G_0(\cht,\chq,\chq)$. Furthermore,
\be\label{del G0 2 parts}\bs
G_0(\cht,\chq+r)-G_0(\cht,\chq)=&\big[G_0(\cht,\chq+r)-G_0(\cht-r,\chq+r)\big]\\
&+\big[\check G_0(\cht,\chq,\chq+r)-\check G_0(\cht,\chq,\chq)\big].
\end{split}
\ee
The difference in the first brackets is bounded as in \eqref{del G0 bnd a}.

Clearly, $G_0(\cht,\chq+1)\equiv G_0(\cht,\chq)$. Hence, we can assume $0\le \chq<1$. 
Using that $h_0$ is compactly supported and $|\cht|<A$, we obtain
\be\label{aux bnd 1}\bs
\int_{\br}\frac{h_0^\prime(\chw+\cht+\chq-j_2)}\chw \dd \chw = O(j_2^{-2}),\ j_2\to\infty.
\end{split}
\ee
Also,
\be\label{aux bnd 2}
\big|(x-r)_+^{1/2}-x_+^{1/2}\big|\le c[r/(1+|x|)]^{1/2},\ 0\le r<1,\ x\in\br.
\ee
The estimates \eqref{aux bnd 1}, \eqref{aux bnd 2} and \eqref{G0H0-alt} imply that the difference in the second brackets in \eqref{del G0 2 parts} is $O(r^{1/2})$. Together with \eqref{del G0 bnd a} this gives
\be\label{del G0 bnd b}
|G_0(\cht,\chq+r)-G_0(\cht,\chq)|\le c r^{1/2},\  0\le r<1,\ |\cht|<A.
\ee
\color{black}

Combining \eqref{P exp} and  \eqref{del G0 bnd b} with \eqref{G simpl v3} implies
\be\label{G in Bk}\begin{split}
G\bigg(\cht,\frac{P(\e^{1/2}s)}\e;\theta,\e\bigg)
=&\mu G_0\bigg(\cht,\frac{P(\e^{1/2}s)}\e\bigg)+O(\e^{1/2})\\
=&\mu G_0\bigg(\cht,\frac{P(\e^{1/2}s_k)}\e+P^\prime(0)\frac{s-s_k}{\e^{1/2}}\bigg)\\
&+O(1/K^{1/2})+O(\e^{1/2}), 
\end{split}
\ee
where $s=\e^{1/2}\mu j_1\in B_k$ and $|\theta|<\e^{1/2}A$.
The equation \eqref{G simpl v3} applies because the assumptions of Lemma~\ref{lem:del F0} imply
\be\label{th small}\bs
&\cht=\frac{Q_{\al\al}^{\prime\prime}}2 \e(\mu j_1)^2=O(1),\\
&\theta=\al_{j_1}=\e\mu j_1=O(\e^{1/2}),\\ 
&\e\chq=P(\e\mu j_1)=P(0)+O(\e^{1/2}),\ \e^{1/2}\mu|j_1|\le A,
\end{split}
\ee
see \eqref{F via G part2}. Further, using that $s^2=s_k^2+O(1/K)$ and \eqref{del G0 bnd a}, gives
\be\label{delG in Bk}\begin{split}
G\bigg(\frac{Q_{\al\al}^{\prime\prime}}2 s^2,\frac{P(\e^{1/2}s)}\e;\theta,\e\bigg)
=&\mu G_0\bigg(\frac{Q_{\al\al}^{\prime\prime}}2 s_k^2,\frac{P(\e^{1/2}s_k)}\e+P^\prime(0)\frac{s-s_k}{\e^{1/2}}\bigg)\\
&+O(1/K^{1/2})+O(\e^{1/2}),\ s\in B_k,|\theta|<\e^{1/2}A.
\end{split}
\ee
There are $|B_k|/(\e^{1/2}\mu)$ values of $j_1$ such that $\e^{1/2}\mu j_1\in B_k$. Substitute $s=\e^{1/2}\mu j_1$ into \eqref{delG in Bk} and sum with respect to $j_1$ to obtain
\be\label{sum delG in Bk}\begin{split}
\e^{1/2}&\sum_{\e^{1/2}\mu j_1\in B_k}G\bigg(\frac{Q_{\al\al}^{\prime\prime}}2 \e (\mu j_1)^2,\frac{P(\e\mu j_1)}\e;\theta,\e\bigg)\\
=&\e^{1/2}\mu\sum_{\e^{1/2}\mu j_1\in B_k}G_0\bigg(\frac{Q_{\al\al}^{\prime\prime}}2 s_k^2,u_\e+P^\prime(0)\mu j_1\bigg)\\
&+|B_k|\big[O(1/K^{1/2})+O(\e^{1/2})\big],\ u_\e:=\frac{P(\e^{1/2}s_k)-P^\prime(0)\e^{1/2}s_k}\e.
\end{split}
\ee
\textcolor{black}{The term $O(1/K^{1/2})$ is uniform with respect to $1\le k\le K$ and $0<\e\ll1$. 
The term $O(\e^{1/2})$ is uniform with respect to $1\le k\le K$ and $K\ge 1$.}
It follows from \eqref{G0H0} and \eqref{del G0 bnd b} that $G_0$ is 1-periodic in its last argument and H{\"o}lder continuous. 

From \eqref{for Q 1} and the definition $P(\al)=\Phi(x(\al),\al)$, 
\be\begin{split}
P^\prime(\al)&=[\Phi(x(\al),\al)]_\al^\prime=\dd_x\Phi(x(\al),\al)\cdot x_\al^\prime(\al)+\Phi_\al^\prime(x(\al),\al)\\
&=\Phi_\al^\prime(x(\al),\al).
\end{split}
\ee
By the assumption of the theorem, $P^\prime(0)\mu$ is irrational, so the points $P^\prime(0)\mu j_1$ are well distributed mod 1 \cite[Definition 5.1 and Example 5.2]{KN_06}. Take the limit $\e\to0$ in \eqref{sum delG in Bk} and use \cite[Theorem 5.1]{KN_06} to find
\be\label{int delG in Bk}\begin{split}
\lim_{\e\to0}\e^{1/2}&\sum_{\e^{1/2}\mu j_1\in B_k}G\bigg(\frac{Q_{\al\al}^{\prime\prime}}2 \e (\mu j_1)^2,\frac{P(\e\mu j_1)}\e;\e\mu j_1,\e\bigg)\\
=& \frac{|B_k|}\mu\bigg[\mu\int_0^1 G_0\bigg(\frac{Q_{\al\al}^{\prime\prime}}2 s_k^2,\chq\bigg)\dd\chq+O(1/K^{1/2})\bigg].
\end{split}
\ee
Sum \eqref{int delG in Bk} over all $B_k$ and use that $K\gg1$ can be arbitrarily large:
\be\label{int delG}\begin{split}
\lim_{\e\to0}\e^{1/2}&\sum_{\e^{1/2}\mu|j_1|\le A}G\bigg(\frac{Q_{\al\al}^{\prime\prime}}2 \e (\mu j_1)^2,\frac{P(\e\mu j_1)}\e;\e\mu j_1,\e\bigg)\\
=&\int_{|s|\le A}\int_0^1 G_0\bigg(\frac{Q_{\al\al}^{\prime\prime}}2 s^2,\chq\bigg)\dd\chq\dd s.
\end{split}
\ee
Clearly the same arguments apply when we add $\dd_x\Phi\cdot\chx$ to the first argument of $G$, cf. \eqref{F via G part2}. 
Therefore, combining \eqref{int delG}, \eqref{F via G part2} and \eqref{delF0 def}, proves \eqref{del F0}.

\subsubsection{Proof of Lemma~\ref{lem:G approx}}
Let $A_1$, $\cht$, $\chq$ and $\theta$ be selected as in the statement of Lemma~\ref{lem:G approx}. Fix any $\bar h\in C_0^2(\br)$ and denote 
\be\label{Q def}\bs
\Q(u):=\int_\br \frac{\chii(\e\chw)}{\chw} \bar h^\prime(\chw+u)\dd\chw.
\end{split}
\ee
We make use of the following simple result, which is stated without proof:
\be\label{aux int asymp}\bs
&|\Q(u)|\le c(1+u^2)^{-1},\ |u|\le c_1/\e,\quad \Q(u)\equiv0,\ |u|\ge c_1/\e,
\end{split}
\ee
for some $c_1>0$, which is the same in both places.

Substitute $\bar h(u)=h(u;\theta,\e)$ of \eqref{h_eps} into \eqref{Q def} to obtain $\Q(u;\theta,\e)$. In view of \eqref{sum j}, \eqref{key int}, and \eqref{Q def}, we estimate 
\be\label{key int appr}\bs
&\bigg|\sum_{j_2\in J_p}\Q(\cht+\chq-j_2;\theta,\e)\big[a_0(\theta,\e j_2)-a_0(\theta,\e \chq)\big](j_2-\chq)_+^{1/2}\bigg|\\
&\le c\e \sum_{j_2=1}^{1/\e}\frac{1}{j_2^2}j_2^{3/2}=O(\e^{1/2}).
\end{split}
\ee
This implies
\be\label{G simpl v2}\bs
G(\cht,\chq;\theta,\e)=&a_0(\theta,\e \chq)\sum_{j_2\in J_p}\Q(\cht+\chq-j_2;\theta,\e)(j_2-\chq)_+^{1/2}+O(\e^{1/2}).
\end{split}
\ee
Since $|\cht|<A_1$, \eqref{aux int asymp} implies that the sum in \eqref{G simpl v2} is absolutely convergent and its value is uniformly bounded. Using the assumptions of Lemma~\ref{lem:G approx}, we find
\be\label{G appr 1}\bs
G(\cht,\chq;\theta,\e)=&\big[a_0(0,P(0))+O(\e^{1/2})\big]\\
&\times\sum_{j_2\in J_p}\Q(\cht+\chq-j_2;\theta,\e)(j_2-\chq)_+^{1/2}+O(\e^{1/2})\\
=&a_0(0,P(0))\sum_{j_2\in J_p}\Q(\cht+\chq-j_2;\theta,\e)(j_2-\chq)_+^{1/2}+O(\e^{1/2}).
\end{split}
\ee

To show that $\chi$ can be dropped from the integral in \eqref{G appr 1} (which is inside the definition of $\Q$), we consider similarly to \eqref{Q def}:
\be\label{Qc def}\bs
\Q_c(u):=\int_\br \frac{1-\chii(\e\chw)}{\chw} \bar h^\prime(\chw+u)\dd\chw=\int_{|\chw|\ge c/\e} \frac{1-\chii(\e\chw)}{\chw} \bar h^\prime(\chw+u)\dd\chw.
\end{split}
\ee
As is easily seen,
\be\label{Qc bnd}\bs
&|\Q_c(u)|\le c\big(u^{-2}+(\e/|u|)\big),\ |u|>c_1/\e;\quad \Q_c(u)=0,\ |u|\le c_1/\e,
\end{split}
\ee
for some $c_1>0$, which is the same in both places. Given that $|\cht|<A_1$, this implies
\be\label{Qc bnd used}\bs
\sum_{j_2\in J_p}\Q_c(\cht+\chq-j_2;\theta,\e)(j_2-\chq)_+^{1/2}=O(\e^{1/2}).
\end{split}
\ee
Due to the term $\e/|u|$ in the bound in \eqref{Qc bnd}, it is essential that the sum in \eqref{Qc bnd used} is over a subset of the range $0<j_2-\chq<c/\e$. Therefore, by \eqref{G appr 1} and \eqref{Qc bnd used},
\be\label{G simpl v4}\bs
G(\cht,\chq;\theta,\e)=&a_0(0,P(0))\int_{\br}\frac 1{\chw} \sum_{j_2\in J_p}h^\prime(\chw+\cht+\chq-j_2;\theta,\e)(j_2-\chq)_+^{1/2} \dd \chw\\
&+O(\e^{1/2}).
\end{split}
\ee

\color{black}
It is obvious that if $\e=0$ in \eqref{Q def} (i.e., if $\chii$ is replaced by the constant function 1), then $\Q(u)=O(u^{-2})$, $u\in\br$. This implies that
\be\label{Q alt use}\bs
\sum_{j_2-\chq>c/\e}\Q(\cht+\chq-j_2;\theta,\e)(j_2-\chq)_+^{1/2}=O(\e^{1/2}).
\end{split}
\ee
Hence, in \eqref{G simpl v4}, we can extend the summation to all of $j_2\in\mathbb Z$.
Further,
\be\label{del Phi}\bs
K(\theta+\e\mu s)=&K(\theta)+O(\e),\ \frac{\Phi(\theta+\e\mu s)-\Phi(\theta)}\e=\mu s\Phi^\prime(\theta)+O(\e),
\end{split}
\ee
where $s\in\tsp\ik_\al$. Here and in what follows, we drop the first argument of $K(x_0=0_2,\al)$ from notation (see the paragraph following \eqref{delF0 def a}). Using \eqref{h_eps}, introduce the variable 
\be
s_1=s_1(s;\theta,\e)=\frac{\Phi(\theta+\e\mu s)-\Phi(\theta)}{\e\mu\Phi^\prime(0)}.
\ee
By the assumption in Theorem~\ref{thm:main res},  $\mu\Phi_\al^\prime(0_2,\al_0)$ is irrational. Hence $\Phi^\prime(0)=\Phi_\al^\prime(0_2,\al_0)\not=0$. The assumption $|\theta|=O(\e^{1/2})$ implies 
\be\label{s1 and s}
s_1(s;\theta,\e)-s=O(\e^{1/2}),\ \pa s_1(s;\theta,\e)/\pa s=1+O(\e^{1/2}). 
\ee
Changing variable $s\to s_1$ in \eqref{h_eps} gives
\be\label{h_eps v2}\bs
h(u;\theta,\e)
&=\mu\int_\br K(\theta+\e\mu \zeta)\ik_\al(\zeta)|\pa\zeta/\pa s|\ik_p(\mu s\Phi^\prime(0)+u)\dd s,\ 
\zeta=\zeta(s;\theta,\e).
\end{split}
\ee
To simplify the next step, in \eqref{h_eps v2}, the new variable $s_1$ is denoted $s$, and the old variable $s$ is denoted $\zeta$. Together with \eqref{G0H0} this implies
\be\label{del h_eps v2}\bs
&h(u;\theta,\e)-\mu h_0(u)\\
&=\mu\int_\br \big[K(\theta+\e\mu \zeta)\ik_\al(\zeta)|\pa\zeta/\pa s|-K(0)\ik_\al(s)\big]\ik_p(\mu s\Phi^\prime(0)+u)\dd s.
\end{split}
\ee
Using that $\ik_\al,\ik_p\in C_0^2(\br)$ (see Assumption~\ref{ass:interp ker}\eqref{ikcont}) and \eqref{s1 and s}, we conclude
\be\label{delu der bnds}
\pa_u^k(h(u;\theta,\e)-\mu h_0(u))=O(\e^{1/2}),\ k=0,1,2.
\ee
\color{black}

Similarly to \eqref{Q def} and \eqref{aux int asymp}, pick any $\bar h_\e\in C_0^2(\br)$ such that $\bar h_\e^{(k)}=O(\e^{1/2})$, $k=0,2$, and define
\be\label{Q1 def}\bs
\Q_\e(u):=\int_\br \frac1{\chw} \bar h_\e^\prime(\chw+u)\dd\chw.
\end{split}
\ee
Then
\be\label{Q1 aux int asymp}\bs
&|\Q_\e(u)|\le c\e^{1/2}(1+u^2)^{-1},\ u\in\br.
\end{split}
\ee
For large $|u|$, \eqref{Q1 aux int asymp} is proven by integrating by parts in \eqref{Q1 def} and using the bound on $\bar h_\e$. For $|u|$ in a bounded set, we use the bound on $\bar h_\e^{\prime\prime}$.

Set $\bar h_\e(u)=h(u;\theta,\e)-\mu h_0(u)$. Using \eqref{del h_eps v2} -- \eqref{Q1 aux int asymp}, we get from \eqref{G simpl v4}
\be\label{G simpl v5}\bs
&|G(\cht,\chq;\theta,\e)-\mu G_0(\cht,\chq)|\le c |J_\e|+O(\e^{1/2}),\\
&J_\e:=\tsum_{j_2\in\mathbb Z}\Q_\e(\cht+\chq-j_2)(j_2-\chq)_+^{1/2}=O(\e^{1/2}).
\end{split}
\ee
This proves the lemma.

\subsection{Proof of Lemma~\ref{lem:del F0 v2}}\label{sec:prf lem del F0 v2}

We evaluate the integrals in \eqref{last sum lim2} similarly to \cite[eqs. (3.16)--(3.20)]{kat19a}. Begin by considering
\be\label{last sum transf}\bs
J(r):=&\int_{\br}\int_{\br}\ioi\frac 1\chw \big[h_0^\prime\big(\chw+r+(Q_{\al\al}^{\prime\prime}/2) s^2-\chq\big)\\
&\hspace{3cm}-h_0^\prime\big(\chw+(Q_{\al\al}^{\prime\prime}/2) s^2-\chq\big)\big]\chq^{1/2}\dd \chq\dd s \dd \chw\\
=&(Q_{\al\al}^{\prime\prime}/2)^{-1/2}\int_{\br}\int_{\br}\int_{\br}\frac 1\chw \big[h_0^\prime\big(\chw+r+u-\chq\big)\\
&\hspace{3cm}-h_0^\prime\big(\chw+ u-\chq\big)\big]u_+^{-1/2}\chq_+^{1/2}\dd \chq\dd u \dd \chw.\\
\end{split}
\ee
Substitute $h_0$ expressed in terms of its Fourier transform, $\tilde h_0=\CF h$, and use the identities \cite[p. 360, eqs. 19 and 21]{gs}:
\be\bs
\CF(x_+^{a-1})&=e^{ai\pi/2}\Gamma(a)(\la+i0)^{-a},\ a\not=0,-1,-2,\dots;\\
\CF(1/x)&=i\pi\text{sgn}\,\la,
\end{split}
\ee
to obtain
\be\label{last sum ft1}\bs
J(r)=\frac1{2\pi} \int_{\br} & \big(\CF(1/\chw)\big)(-\la)\big(\CF u_+^{-1/2}\big)(-\la)\big(\CF\chq_+^{1/2}\big)(\la)\\
&\times \big(e^{-i\la r}-1\big)(-i\la)\tilde h_0(\la)\dd\la\\
=\frac1{2\pi} \int_{\br} & \big(-i\pi\text{sgn}(\la)\big)\big(e^{3i\pi/4}\Gamma(3/2)(\la+i0)^{-3/2}\big)\\
&\times \big(e^{-i\pi/4}\Gamma(1/2)(\la-i0)^{-1/2}\big)\big(e^{-i\la r}-1\big)(-i\la)\tilde h_0(\la)\dd\la.
\end{split}
\ee
This simplifies to
\be\label{last sum ft2}\bs
J(r)=&-\frac{i\Gamma(1/2)\Gamma(3/2)}{2} \int_{\br}  \frac{e^{-i\la r}-1}{\la}\tilde h_0(\la)\dd\la\\
=&\frac{\pi}{8}\int_{\br}\big[\text{sgn}(-r-t)-\text{sgn}(-t)\big]h_0(t)\dd t=-\frac{\pi}{4}\int_0^r h_0(-t)\dd t.
\end{split}
\ee
Substitution of \eqref{last sum ft2} back into \eqref{last sum lim2} yields \eqref{last sum lim3}.

\section{Proof of Lemmas~\ref{lem:int decay 2d} and \ref{lem:Ups props}}

\subsection{Proof of Lemma~\ref{lem:int decay 2d}}\label{ssec:prf int decay 2d}

We may assume $n=0$ in \eqref{D ineqs}. The case $n\ge 1$ is completely analogous.

If $0$ is not in the closure of $I$, integration by parts in \eqref{lim1} twice with respect to $\la$  
yields the top line in \eqref{lim3}. Therefore, using a partition of unity, we can assume that $I=(-\de,\de)$ for some small $0<\de\ll1$, and $D(\la,\theta,\e)\in \coi(I)$ as a function of $\theta$. Change variable $\theta\to t=\sin\theta$ on $I$ to get 
\be\label{lim prf 1}
J(r)=\int_0^\infty \int_{I_1} D_1(\la,t,\e)e^{i r\la t}\dd t \la \dd\la=\int_0^\infty \tilde D_1(\la,r\la,\e) \la\dd\la.
\ee
Here
\be
I_1=(-\sin\de,\sin\de),\ D_1(\la,t,\e)=D(\la,\sin^{-1}t,\e)/(1-t^2)^{1/2},
\ee
and $\tilde D_1(\la,\mu,\e)$ is the Fourier transform of $D_1(\la,t,\e)$ with respect to $t$. Thus $\tilde D_1(\la,\mu,\e)$ decays rapidly as $\mu\to\infty$, i.e. for any $N\in\N_0$ there exists $c_N$ such that
\be\label{tilde D bnd}
|\tilde D_1(\la,\mu,\e)|\le c_N(1+|\mu|)^{-N},\ (\la,\mu,\e)\in (0,\infty)\times \br\times (0,\e_0).
\ee
Selecting $N=3$ gives
\be\label{lim prf 2}
|J(r)|\le c\int_0^\infty \frac{\la}{1+|r\la|^3}\dd\la.
\ee
A change of variable $\la\to r\la$ yields the top case in \eqref{lim3}. 


If $D(\la=0,\theta,\e)\equiv 0$, then, away from $\theta=0$, we integrate by parts in \eqref{lim1} three times instead of two to prove the bottom line in \eqref{lim3}. In a neighborhood of $\theta=0$, the analog of \eqref{lim prf 1} becomes
\be\label{lim prf 3}
J(r)=\int_0^\infty (\tilde D_1(\la,r\la,\e)/\la) \la^2\dd\la.
\ee
Selecting $N=4$ in \eqref{tilde D bnd} then gives
\be\label{lim prf 2b}
|J(r)|\le c\int_0^\infty \frac{\la^2}{1+|r\la|^4}\dd\la=O(r^{-3}),\ r\to\infty,
\ee
and the proof is complete.

\subsection{Proof of Lemma~\ref{lem:Ups props}}\label{sec:prf Ups props}
The statement $\Upsilon(0)=0$ is immediate. To prove the second claim, compute
\be\bs
\Upsilon(+\infty)-\Upsilon(-\infty)&=\int_{\br} \bar h_0(u)\dd u=1.
\end{split}
\ee
The last equality follows from \eqref{norm-deriv} and \eqref{G0H0} (recall that $\bar h_0$ is defined without the prefactor in \eqref{G0H0}). Similarly,
\be\bs
&\Upsilon(+\infty)+\Upsilon(-\infty)\\
&=\frac1{|\dd_x\Phi|}\int_{\br} \bar h_0(u)\int_{\br}\text{sgn}(t-u)\CF^{-1}\bigg(\frac1{1+\kappa C_1|\la|^3}\bigg)\left(\frac{t}{|\dd_x\Phi|}\right)\dd t\dd u=0.
\end{split}
\ee
The last equality follows because the signum function is odd, while the functions $\bar h_0$ and $\CF^{-1}\big((1+\kappa C_1|\la|^3)^{-1}\big)$ are even. The fact that $\bar h_0$ is even follows from Assumption~\ref{ass:interp ker}\eqref{ikeven} and \eqref{G0H0}.


\bibliographystyle{siam}
\bibliography{My_Collection}

\begin{thebibliography}{10}

\bibitem{Abels12}
{\sc H.~Abels}, {\em Pseudodifferential and {{Singular Integral Operators}}},
  De Gruyter, Berlin/Boston, 2011.

\bibitem{AKW2024_1}
{\sc A.~Abhishek, A.~Katsevich, and J.~W. Webber}, {\em Local reconstruction
  analysis of inverting the {{Radon}} transform in the plane from noisy
  discrete data}, SIAM Journal on Imaging Sciences, 18 (2025), pp.~936--962.

\bibitem{AgrQ96}
{\sc M.~L. Agranovsky and E.~T. Quinto}, {\em Injectivity {{Sets}} for the
  {{Radon Transform}} over {{Circles}} and {{Complete Systems}} of {{Radial
  Functions}}}, Journal of Functional Analysis, 139 (1996), pp.~383--414.

\bibitem{AhnLeahy08}
{\sc S.~Ahn and R.~M. Leahy}, {\em Analysis of resolution and noise properties
  of nonquadratically regularized image reconstruction methods for {{PET}}},
  IEEE Transactions on Medical Imaging, 27 (2008), pp.~413--424.

\bibitem{Alazard2024}
{\sc T.~Alazard}, {\em Analysis and {{Partial Differential Equations}}},
  Springer, 2024.

\bibitem{AmbK05}
{\sc G.~Ambartsoumian and P.~Kuchment}, {\em On the injectivity of the circular
  {{Radon}} transform}, Inverse Problems, 21 (2005), pp.~473--485.

\bibitem{BKN15}
{\sc R.~Beauwens}, {\em Iterative solution methods}, in Applied {{Numerical
  Mathematics}}, O.~Scherzer, ed., vol.~51, Springer, New York, 2004,
  pp.~437--450.

\bibitem{btu2003}
{\sc T.~Blu, P.~Th{\'e}venaz, and M.~Unser}, {\em Complete parameterization of
  piecewise-polynomial interpolation kernels}, IEEE Transactions on Image
  Processing, 12 (2003), pp.~1297--1309.

\bibitem{Ceg2012}
{\sc A.~Cegielski}, {\em Iterative Methods for Fixed Point Problems in
  {{Hilbert}} Spaces}, vol.~2057 of Lecture {{Notes}} in {{Mathematics}},
  Springer, Heidelberg, 2012.

\bibitem{ChunFess12}
{\sc S.~Y. Chun and J.~A. Fessler}, {\em Noise properties of motion-compensated
  tomographic image reconstruction methods}, IEEE Transactions on Medical
  Imaging, 31 (2012), pp.~1413--1425.

\bibitem{dhara2016combination}
{\sc A.~K. Dhara, S.~Mukhopadhyay, A.~Dutta, M.~Garg, and N.~Khandelwal}, {\em
  A {{Combination}} of {{Shape}} and {{Texture Features}} for
  {{Classification}} of {{Pulmonary Nodules}} in {{Lung CT Images}}}, Journal
  of Digital Imaging, 29 (2016), pp.~466--475.

\bibitem{dhara2016differential}
{\sc A.~K. Dhara, S.~Mukhopadhyay, P.~Saha, M.~Garg, and N.~Khandelwal}, {\em
  Differential geometry-based techniques for characterization of boundary
  roughness of pulmonary nodules in {{CT}} images}, International Journal of
  Computer Assisted Radiology and Surgery, 11 (2016), pp.~337--349.

\bibitem{far04}
{\sc A.~Faridani}, {\em Sampling theory and parallel-beam tomography}, in
  Sampling, Wavelets, and Tomography, vol.~63, Birkhauser Boston, Boston, MA,
  2004, pp.~225--254.

\bibitem{Fess03}
{\sc J.~A. Fessler}, {\em Analytical approach to regularization design for
  isotropic spatial resolution}, IEEE Nuclear Science Symposium Conference
  Record, 3 (2003), pp.~2022--2026.

\bibitem{fsu-08}
{\sc B.~Frigyik, P.~Stefanov, and G.~Uhlmann}, {\em The {{X-ray}} transform for
  a generic family of curves and weights}, Journal of Geometric Analysis, 18
  (2008), pp.~89--108.

\bibitem{GW23}
{\sc O.~Gannot and J.~Wunsch}, {\em Semiclassical diffraction by conormal
  potential singularities}, Annales Scientifiques de l {\'E}cole Normale
  Sup{\'e}rieure,  (2023).

\bibitem{gs}
{\sc I.~M. Gel'fand and G.~E. Shilov}, {\em Generalized {{Functions}}, {{Volume
  I}}. {{Properties}} and {{Operations}}.}, AMS Chelsea Publishing, Providence,
  RI, 1964.

\bibitem{GouyetRosso1996}
{\sc J.-F. Gouyet, M.~Rosso, and B.~Sapoval}, {\em Fractal {{Surfaces}} and
  {{Interfaces}}}, in Fractals and {{Disordered Systems}}, A.~Bunde and {\relax
  Sh}.~Havlin, eds., Springer, Berlin, Heidelberg, 1996, pp.~263--302.

\bibitem{Grig1994}
{\sc A.~Grigis and J.~Sj{\"o}strand}, {\em Microlocal {{Analysis}} for
  {{Differential Operators}}}, Cambridge University Press, 1994.

\bibitem{HJL21}
{\sc P.~C. Hansen, J.~J{\o}rgensen, and W.~R.~B. Lionheart}, eds., {\em
  Computed {{Tomography}}: {{Algorithms}}, {{Insight}}, and {{Just Enough
  Theory}}}, SIAM, 2021.

\bibitem{Hay1968}
{\sc E.~V. Haynsworth}, {\em Determination of the inertia of a partitioned
  {{Hermitian}} matrix}, Linear Algebra and Its Applications, 1 (1968),
  pp.~73--81.

\bibitem{Holm24}
{\sc S.~Holman}, {\em Generalized {{Radon}} transforms}, in Microlocal
  {{Analysis}} and {{Inverse Problems}} in {{Tomography}} and {{Geometry}},
  T.~Quinto, P.~Stefanov, and G.~Uhlmann, eds., Walter de Gruyter GmbH,
  Berlin/Boston, 2024, pp.~1--34.

\bibitem{Homan2017}
{\sc A.~Homan and H.~Zhou}, {\em Injectivity and {{Stability}} for a {{Generic
  Class}} of {{Generalized Radon Transforms}}}, Journal of Geometric Analysis,
  27 (2017), pp.~1515--1529.

\bibitem{hor3}
{\sc L.~Hormander}, {\em The {{Analysis}} of {{Linear Partial Differential
  Operators III}}. {{Pseudo-Differential Operators}}}, Springer-Verlag, Berlin,
  1994.

\bibitem{hor}
\leavevmode\vrule height 2pt depth -1.6pt width 23pt, {\em The {{Analysis}} of
  {{Linear Partial Differential Operators I}}. {{Distribution Theory}} and
  {{Fourier Analysis}}}, Classics in {{Mathematics}}, Springer-Verlag, Berlin,
  2003.

\bibitem{hor4}
\leavevmode\vrule height 2pt depth -1.6pt width 23pt, {\em The {{Analysis}} of
  {{Linear Partial Differential Operators IV}}. {{Fourier Integral
  Operators}}}, Springer-Verlag, Berlin, 2009.

\bibitem{Katsevich2017a}
{\sc A.~Katsevich}, {\em A local approach to resolution analysis of image
  reconstruction in tomography}, SIAM Journal on Applied Mathematics, 77
  (2017), pp.~1706--1732.

\bibitem{kat19a}
\leavevmode\vrule height 2pt depth -1.6pt width 23pt, {\em Analysis of
  reconstruction from discrete {{Radon}} transform data in $\mathbb{R}^3$ when
  the function has jump discontinuities}, SIAM Journal on Applied Mathematics,
  79 (2019), pp.~1607--1626.

\bibitem{Katsevich2020b}
\leavevmode\vrule height 2pt depth -1.6pt width 23pt, {\em Analysis of
  resolution of tomographic-type reconstruction from discrete data for a class
  of distributions}, Inverse Problems, 36 (2020).

\bibitem{Katsevich2020a}
\leavevmode\vrule height 2pt depth -1.6pt width 23pt, {\em Resolution analysis
  of inverting the generalized {{Radon}} transform from discrete data in
  $\mathbb{R}^3$}, SIAM Journal on Mathematical Analysis, 52 (2020),
  pp.~3990--4021.

\bibitem{Katsevich2022a}
\leavevmode\vrule height 2pt depth -1.6pt width 23pt, {\em Novel {{Resolution
  Analysis}} for the {{Radon Transform}} in $\mathbb{R}^2$ for {{Functions With
  Rough Edges}}}, SIAM Journal on Mathematical Analysis, 55 (2023),
  pp.~4255--4296.

\bibitem{Katsevich2021a}
\leavevmode\vrule height 2pt depth -1.6pt width 23pt, {\em Resolution
  {{Analysis}} of {{Inverting}} the {{Generalized}} ${N}$-{{Dimensional Radon
  Transform}} in $\mathbb {R}^n$ from {{Discrete Data}}}, Journal of Fourier
  Analysis and Applications, 29 (2023).

\bibitem{Katsevich2023a}
\leavevmode\vrule height 2pt depth -1.6pt width 23pt, {\em Resolution of 2
  {{Dimensional Reconstruction}} of {{Functions With Nonsmooth Edges From
  Discrete Radon Transform Data}}}, SIAM Journal on Applied Mathematics, 83
  (2023), pp.~695--724.

\bibitem{Kats_noise_2024}
\leavevmode\vrule height 2pt depth -1.6pt width 23pt, {\em Analysis of
  reconstruction from noisy discrete generalized {{Radon}} data}, May 2024.

\bibitem{Katsevich_aliasing_2023}
\leavevmode\vrule height 2pt depth -1.6pt width 23pt, {\em Analysis of {{View
  Aliasing}} for the {{Generalized Radon Transform}} in $\mathbb{R}^2$}, SIAM
  Journal on Imaging Sciences, 17 (2024), pp.~415--440.

\bibitem{Katsevich_2025_BV}
\leavevmode\vrule height 2pt depth -1.6pt width 23pt, {\em Analysis of
  reconstruction of functions with rough edges from discrete {{Radon}} data},
  Journal of Fourier Analysis and Applications, 31, article 38 (2025).

\bibitem{Keys1981}
{\sc R.~G. Keys}, {\em Cubic {{Convolution Interpolation}} for {{Digital Image
  Processing}}}, IEEE Transactions on Acoustics, Speach, and Signal Processing,
  ASSP-29 (1981), pp.~1153--1160.

\bibitem{KN_06}
{\sc L.~Kuipers and H.~Niederreiter}, {\em Uniform {{Distribution}} of
  {{Sequences}}}, John Wiley \& Sons, Inc., New York, 1974.

\bibitem{Li2019}
{\sc X.~Li, M.~Luo, and J.~Liu}, {\em Fractal characteristics based on
  different statistical objects of process-based digital rock models}, Journal
  of Petroleum Science and Engineering, 179 (2019), pp.~19--30.

\bibitem{Monard2021}
{\sc F.~Monard and P.~Stefanov}, {\em Sampling the x-ray transform on simple
  surfaces}, SIAM Journal on Mathematical Analysis, 53 (2023), pp.~1707--1736.

\bibitem{nat93}
{\sc F.~Natterer}, {\em Sampling in fan beam tomography}, SIAM Journal on
  Applied Mathematics, 53 (1993), pp.~358--380.

\bibitem{nat3}
{\sc F.~Natterer}, {\em The Mathematics of Computerized Tomography}, vol.~32 of
  Classics in {{Applied Mathematics}}, SIAM, Philadelphia, 2001.

\bibitem{soilfractals2000}
{\sc Y.~Pachepsky, J.~W. Crawford, and W.~J. Rawls}, eds., {\em Fractals in
  Soil Science}, Developments in {{Soil Science}} 27, Elsevier, 2000.

\bibitem{pal95}
{\sc V.~P. Palamodov}, {\em Localization of harmonic decomposition of the
  {{Radon}} transform}, Inverse Problems, 11 (1995), pp.~1025--1030.

\bibitem{Peyp2015}
{\sc J.~Peypouquet}, {\em Convex {{Optimization}} in {{Normed Spaces}}.
  {{Theory}}, {{Methods}} and {{Examples}}}, Springer {{Briefs}} in
  {{Optimization}}, Springer, Heidelberg, 2015.

\bibitem{PowerTullis1991}
{\sc W.~L. Power and T.~E. Tullis}, {\em Euclidean and fractal models for the
  description of rock surface roughness}, Journal of Geophysical Research, 96
  (1991), pp.~415--424.

\bibitem{pbm1}
{\sc A.~P. Prudnikov, {\relax Yu}.~A. Brychkov, and O.~I. Marichev}, {\em
  Integrals and {{Series}} of {{Elementary Functions}}.}, {Gordon and Breach},
  New York, 1986.

\bibitem{QiLeahy00}
{\sc J.~Qi and R.~M. Leahy}, {\em Resolution and noise properties of {{MAP}}
  reconstruction for fully 3-{{D PET}}}, IEEE Transactions on Medical Imaging,
  19 (2000), pp.~493--506.

\bibitem{qu-80}
{\sc E.~T. Quinto}, {\em The dependence of the generalized {{Radon}} transform
  on defining measures}, Transactions of the American Mathematical Society, 257
  (1980), pp.~331--346.

\bibitem{ShiFess06}
{\sc H.~Shi and J.~A. Fessler}, {\em Quadratic regularization design for
  iterative reconstruction in {{3D}} multi-slice axial {{CT}}}, IEEE Nuclear
  Science Symposium Conference Record, 5 (2006), pp.~2834--2836.

\bibitem{ShiFess09}
{\sc H.~R. Shi and J.~A. Fessler}, {\em Quadratic regularization design for
  2-{{D CT}}}, IEEE Transactions on Medical Imaging, 28 (2009), pp.~645--656.

\bibitem{StayFess04b}
{\sc J.~W. Stayman and J.~A. Fessler}, {\em Compensation for {{Nonuniform
  Resolution Using Penalized-Likelihood Reconstruction}} in {{Space-Variant
  Imaging Systems}}}, IEEE Transactions on Medical Imaging, 23 (2004),
  pp.~269--284.

\bibitem{StayFess04a}
\leavevmode\vrule height 2pt depth -1.6pt width 23pt, {\em Efficient
  calculation of resolution and covariance for penalized-likelihood
  reconstruction in fully 3-{{D SPECT}}}, IEEE Transactions on Medical Imaging,
  23 (2004), pp.~1543--1556.

\bibitem{stef20}
{\sc P.~Stefanov}, {\em Semiclassical {{Sampling}} and {{Discretization}} of
  {{Certain Linear Inverse Problems}}}, SIAM Journal on Mathematical Analysis,
  52 (2020), pp.~5554--5597.

\bibitem{trev1}
{\sc F.~Treves}, {\em Introduction {{To Pseudodifferential}} and {{Fourier
  Integral Operators Vol}}. 1: {{Pseudodifferential Operators}}}, Plenum Press,
  New York and London, 1980.

\bibitem{trev2}
\leavevmode\vrule height 2pt depth -1.6pt width 23pt, {\em Introduction {{To
  Pseudodifferential}} and {{Fourier Integral Operators}}. {{Vol}}. 2:
  {{Fourier Integral Operators}}}, Plenum Press, New York and London, 1980.

\end{thebibliography}
\end{document}